\documentclass[a4paper,11pt]{article}

\usepackage{amsmath}
\usepackage{amsthm}
\usepackage{amsfonts}
\usepackage{amssymb}
\usepackage{t-angles}

\usepackage[all]{xy}
\usepackage{amscd}

\newcommand{\al}{\alpha}
\newcommand{\si}{\sigma}

\newcommand{\id}{\mathrm{id}}

\newcommand{\ot}{\otimes}
\newcommand{\ra}{\rightarrow}

\newcommand{\trl}{\triangleleft}
\newcommand{\trr}{\triangleright}

\def\ppr{\rightharpoonup}
\def\ppl{\leftharpoonup}

\newcommand{\li}{{}_{1}}
\newcommand{\lii}{{}_{2}}

\newcommand{\lmo}{{}_{(0)}} 

\newcommand{\loo}{{}_{(0)}}
\newcommand{\loi}{{}_{(-1)}}

\newcommand{\lmi}{{}_{(1)}}

\newcommand{\lmoi}{{}_{(-1)}}

\newcommand{\mo}{{}_{(0)}}

\newcommand{\moi}{{}_{(-1)}}

\newcommand{\bi}{{}_{[1]}}
\newcommand{\bii}{{}_{[2]}}

\newcommand{\qi}{{}_{\{1\}}}

\newcommand{\qoi}{{}_{\{-1\}}}
\newcommand{\qoo}{{}_{\{0\}}}

\def\rbiprod{{\cdot\kern-.33em\triangleright\!\!\!<}}
\def\lbiprod{{>\!\!\!\triangleleft\kern-.33em\cdot\, }}

\def\lrbiprod{{\ \cdot\kern-.60em\triangleright\kern-.33em\triangleleft\kern-.33em\cdot\, }}

\def\lprod{{>\!\!\!\triangleleft\kern-.33em\ \, }}

\newcommand{\lrcoprod}{{\,\blacktriangleright\!\!\blacktriangleleft\, }}

\allowdisplaybreaks
\newtheorem{theorem}{Theorem}[section]
\newtheorem{lemma}[theorem]{Lemma}

\theoremstyle{definition}
\newtheorem{definition}[theorem]{Definition}

\title{Extending structures for noncommutative Poisson bialgebras}
\author{Tao Zhang, Fang Yang}
\date{}

\begin{document}
 \maketitle

 \setcounter{section}{0}

\begin{abstract}
We introduce the concept of braided noncommutative Poisson bialgebras. The theory of cocycle bicrossproducts for noncommutative Poisson bialgebras is developed.
As an application, we solve the extending problem for noncommutative Poisson bialgebras by using some non-abelian cohomology theory.

\end{abstract}

\tableofcontents

\section{Introduction}

The study of Poisson algebra comes from Poisson  geometry. Poisson algebra is an algebra with a Lie algebra structure and a commutative associative algebra structure which are entwined by Leibniz rule. The concept of a noncommutative Poisson algebra was first given by Xu in \cite{X94}, which is especially suitable for geometric situations.
Noncommutative Poisson algebras differ from Poisson algebras in that Poisson algebras require associative algebras to be commutative, whereas noncommutative Poisson algebras do not require associative algebras to be commutative.
Pre-Poisson algebras are investigated by M.Aguiar in \cite{A1}.
Poisson bialgebra consists of a Lie bialgebra together with a infinitesimai bialgebra, satiafying the compatible conditions. Lie bialgebras have been studied
in \cite{Mas00,Z1,Z2}, and infinitesimai bialgebra have been studied
in \cite{A2,Z4}.
The concept of Poisson bialgebras was introduced by Ni and Bai in \cite{Bai11} which related to classical Yang-Baxter equation(CYBE) and associative Yang-Baxter equation(AYBE) uniformly. Liu, Bai and Sheng in \cite{LJ} introduce the notion of a noncommutative Poisson bialgebra and prove a skew-symmetric solution of Poisson Yang-Baxter equation naturally gives a noncommutative Poisson bialgebra.

The theory of extending structure for many types of algebras were well  developed by A. L. Agore and G. Militaru in \cite{AM1,AM2,AM3,AM4,AM5,AM6}.
Let $A$ be an algebra and $E$ a vector space containing $A$ as a subspace.
The extending problem is to describe and classify all algebra structures on $E$ such that $A$ is a subalgebra of $E$.
They show that associated to any extending structure of $A$ by a complement space $V$, there is a  unified product on the direct sum space  $E\cong A\oplus V$.
Recently, extending structures for 3-Lie algebras, Lie bialgebras,  infinitesimal bialgebras, Lie conformal superalgebras and weighted infinitesimal bialgebras were studied  in \cite{Z2,Z3,Z4}.

As a continue of our paper \cite{Z1} and \cite{Z2}, the aim of this paper is to study  extending structures for noncommutative Poisson bialgebras.
For this purpose, we will introduce the concept of braided noncommutative Poisson bialgebras.  Then we give the construction of cocycle bicrossproducts for noncommutative Poisson bialgebras.
We will show that these new concept  and  construction  will play a key role in considering extending problem for noncommutative Poisson bialgebras.
As an application, we solve the extending problem for noncommutative Poisson bialgebras by using some non-abelian cohomology theory.

This paper is organized as follows. In Section 2, we recall some
definitions and fix some notations. In Section 3, we introduced the concept of braided noncommutative Poisson bialgebras and proved the bosonisation
theorem associating braided noncommutative Poisson bialgebras to ordinary noncommutative Poisson bialgebras.
In  section 4, we define the notion of matched pairs of  braided noncommutative Poisson bialgebras
and construct cocycle bicrossproduct noncommutative Poisson bialgebras through two generalized braided noncommutative Poisson bialgebras.
In section 5, we studied the extending problems for noncommutative Poisson bialgebras and proof that they can be classified by some non-abelian cohomology theory.

Throughout the following of this paper, all vector spaces will be over a fixed field of character zero.
A Lie algebra or a Lie coalgebra is denoted by  $(A, [,])$ or $(A, \delta)$ and an associative algebra or  a coassociative coalgebra is denoted by  $(A, \cdot)$ or $(A, \Delta)$.
The identity map of a vector space $V$ is denoted by $\id_V: V\to V$ or simply $\id: V\to V$.
The flip map $\tau: V\ot V\to V\ot V$ is defined by $\tau(u \ot v)=v\ot u$ for any $u, v\in V$.

\section{Preliminaries}

\begin{definition}[\cite{LJ}] \label{dfnlb} A noncommutative Poisson algebra is a triple $(A, \,  [,], \, \cdot)$ where $A$ is a vector space equipped
 with two bilinear operations $[,], ~ \cdot: A\ot A\to A$, such that $(A, \,  [,])$ is a Lie algebra and $(A, \,  \cdot)$  is an associative algebra (not necessarily commutative) and the following
compatibility condition is satisfied,
\begin{equation}\label{eq:LB1}
 [x, y\cdot z]=[x, y]\cdot z+y\cdot[x, z],
 \end{equation}
for all $x, y, z\in A$ .

Sometimes, we just omit $``\cdot"$ in calculation of the following paper for convenience.

Note that the above identities are equivalent to the following identities:
\begin{equation}\label{eq:LB2}
 [x, yz]=[x, y]z+y[x, z].
 \end{equation}
\end{definition}

\begin{definition}(\cite{LJ})
A noncommutative Poisson coalgebra is a triple $(A, \, \delta, \, \Delta)$ where $A$ is a vector space equipped
 with two maps $\delta, ~ \Delta: A\to A\ot A$, such that $(A, \, \delta)$ is a Lie coalgebra and $(A, \, \Delta)$  is a coassociative coalgebra, such that the satisfy the following
compatibile condition :
\begin{equation}\label{eq:LB1}
 (\id\ot\Delta)\delta(x)=(\delta\ot\id)\Delta(x)+(\tau\ot\id)(\id\ot\delta)\Delta(x),
 \end{equation}
for all $x\in {A}$. 
\end{definition}

\begin{definition}(\cite{LJ})
A noncommutative Poisson bialgebra is a 5-triple $(A, \, [,], \, \cdot, \, \delta, \, \Delta)$ where $(A, \, [,], \, \cdot)$ is a noncommutative Poisson algebra, $(A, \, \delta, \, \Delta)$
is a noncommutative Poisson coalgebra, $(A, \, [,], \, \delta)$ is a Lie bialgebra and $(A, \, \cdot, \, \Delta)$ is an antisymmeetric
infinitesimal bialgebra, such that the following compatible conditions hold:
\begin{equation}\label{eq:LB01}
\delta(x  y)=\left(L_x \otimes \mathrm{id}\right) \delta(y)+\left(R_y \otimes \mathrm{id}\right) \delta(x)+\left(\mathrm{id} \otimes \operatorname{ad}_x\right) \Delta(y)+\left(\mathrm{id} \otimes \operatorname{ad}_y\right) \tau\Delta(x),
\end{equation}
\begin{equation}\label{eq:LB02}
\Delta([x, y])=\left(\operatorname{ad}_x \otimes \mathrm{id}+\mathrm{id} \otimes \operatorname{ad}_x\right) \Delta(y)+\left(R_y \otimes \mathrm{id}-\mathrm{id} \otimes L_y\right) \delta(x)
\end{equation}
where $L_x$ and $\operatorname{ad}_x$ are the left multiplication operator and  the adjoint operator defined by $L_x(y) = xy$ and $ad_x(y) = [x, y]$ respectively. Let $R_y$ denote the right multiplication operator, that is, $R_y(x)=xy$.
If we use the sigma notation $\Delta(x)=x\li\ot x\lii, ~ \delta(x)=x\bi\ot x\bii$, then the above two equations \eqref{eq:LB01} and \eqref{eq:LB02} can be written as
\begin{equation}\label{eq:LB1}
 \delta(x y)=x y\bi \ot y\bii+x\bi  y \ot x\bii  +y\li \ot [x, y\lii] +x\lii \ot [y, x\li],
\end{equation}
\begin{equation}\label{eq:LB2}
\Delta ([x, y])=[x, y\li] \ot y\lii +y\li \ot [x, y\lii] +  x\bi y\ot x\bii -x\bi \ot y x\bii,
\end{equation}
for all $x, y\in {A}$ .
\end{definition}

\begin{definition}(\cite{AM7})
Let ${H}$ be a noncommutative Poisson algebra, $V$ be a vector space. Then $V$ is called a ${H}$-noncommutative Poisson bimodule if there are three linear maps $ \trr : {H}\otimes V \to V, (x, v) \to x \trr v$,  $\ppr :  H\otimes {V} \to V, (x, v) \to x \ppr v$ and $\ppl :  V\otimes {H} \to V, (v, x) \to v \ppl x$  such that  $(V, \, \ppr, \, \ppl)$ is a bimodule of $(H, \, \cdot)$ as associative algebra and $(V, \, \trr)$ is a left module of $(H, \,  [,])$ as Lie algebra, i.e.,
\begin{eqnarray}
  \label{(6)}&&(xy)\ppr v=x\ppr(y\ppr v),\\
  \label{(7)}&&v\ppl (xy)=(v\ppl x)\ppl y,\\
  \label{(8)}&&x\ppr (v\ppl y)=(x\ppr v)\ppl y,\\
  \label{(9)}&&[x, y]\trr v=x\trr(y\trr v)-y\trr(x\trr v),
  \end{eqnarray}
and the following conditions hold:
  \begin{eqnarray}
  \label{(118)}&&(xy) \trr v =x\ppr(y\trr v)+(x\trr v)\ppl y,\\
  \label{(119)}&& [x, y]\ppr v=x\trr(y\ppr v)-y\ppr(x\trr v),\\
  \label{(120)}&& v\ppl [x, y]=x\trr(v\ppl y)-(x\trr v)\ppl y,
\end{eqnarray}
for all $x, y\in {H}$ and $v\in V$.
\end{definition}
The category of noncommutative Poisson bimodules over $H$ is denoted  by ${}_{H}\mathcal{M}{}_{H}$.

\begin{definition}
Let ${H}$ be a noncommutative Poisson coalgebra, $V$ be a vector space. Then $V$ is called a ${H}$-noncommutative Poisson bicomodule if there are three linear maps $\phi: V\to {H}\otimes V$, $\rho: V\to H\otimes {V}$ and
$\gamma: V\to V\otimes {H}$ such that $(V, \, \rho, \, \gamma)$ is a bicomodule of $(H, \, \Delta)$ as coassociative coalgebra  and $(V, \, \phi)$ is a left comodule of $(H, \,  \delta)$ as Lie coalgebra,  i.e.,
\begin{eqnarray}
 \label{(10)}&&\left(\Delta_{H} \otimes \id _{V}\right)\rho(v)=(\id_{H}\ot\rho)\rho(v),\\
 \label{(101)}&&\left(\id _{V}\ot \Delta_{H} \right)\gamma(v)=(\gamma\ot \id_{H})\gamma(v),\\
 \label{(102)}&&\left(\id_{H} \otimes \gamma\right)\rho(v)=(\rho\ot \id_H)\gamma(v),\\
  \label{(11)}&&(\delta_H\ot \id_{V})\phi(v)=(\id_H\ot\phi)\phi(v)-\tau_{12}(\id_H\ot\phi)\phi(v),
\end{eqnarray}
and the following conditions hold:
\begin{eqnarray}\label{comodcoalg12}
&&\left(\id _{V}\ot \Delta_{H} \right)\tau\phi(v)=(\tau\phi\ot \id_{H})\gamma(v)
 +\tau_{12}\left(\id _{H}\ot \tau\phi \right) \rho(v),\\
\label{comodcoalg113}
&&\left(\id_{H}\ot\rho \right)\phi(v)=\left( \delta_H\otimes \id _{V} \right) \rho(v)+\tau_{12}(\id_H\ot\phi)\rho(v),\\
\label{comodcoalg114}
\label{(114)}&&\left(\id_{H}\ot\gamma \right)\phi(v)=\left( \phi\otimes \id _{H} \right) \gamma(v)+\tau_{12}(\id_V\ot\delta_H)\gamma(v).
\end{eqnarray}
If we denote by  $\phi(v)=v_{\langle-1\rangle}\ot v_{\langle0\rangle}$ and $\rho(v)=v\loi\ot v\loo$ and $\gamma(v)=v\loo\ot v\lmi$ then the above equations \eqref{comodcoalg12}, \eqref{comodcoalg113} and \eqref{comodcoalg114} can be written as
\begin{eqnarray}
\label{(16)}&&v_{\langle0\rangle}\ot \Delta_{H}\left(v_{\langle-1\rangle}\right) =\tau\phi(v\loo)\ot v\lmi
+\tau_{12}(v\loi\ot \tau\phi(v\loo)),\\
\label{(17)}&&v_{\langle-1\rangle}\ot\rho(v_{\langle0\rangle})=\delta_H(v\loi)\ot v\loo+\tau_{12}(v\loi\ot\phi(v\loo)),\\
\label{(18)}&&v_{\langle-1\rangle}\ot\gamma(v_{\langle0\rangle})=\phi(v\loo)\ot v\lmi+\tau_{12}(v\loo\ot \delta_H(v\lmi)).
\end{eqnarray}
\end{definition}
The category of noncommutative Poisson bicomodules over $H$ is denoted by ${}^{H}\mathcal{M}{}^{H}$.

\begin{definition}\cite{AM7}
Let ${H}$ and  ${A}$ be noncommutative Poisson algebras. An action of ${H}$ on ${A}$ are three linear maps $ \trr : {H}\otimes V \to V, (x, v) \to x \trr v$,  $\ppr :  H\otimes {V} \to V, (x, v) \to x \ppr v$ and $\ppl :  V\otimes {H} \to V, (v, x) \to v \ppl x$  such that
 \begin{enumerate}
\item[(1)]$(A, \, \cdot, \, \ppr, \, \ppl)$ is an $H$-bimodule algebra over $(H, \, \cdot)$,  i.e.,
\begin{eqnarray}
x\ppr (ab)&=&(x\ppr a)b,\\
(ab)\ppl x&=&a(b\ppl x),\\
(a\ppl x)b&=&a(x\ppr b).
\end{eqnarray}
\item[(2)]$(A, \, [, ], \, \trr)$ is a left $H$-module Lie algebra over $(H, \, [, ])$, i.e.,
\begin{eqnarray}
x\trr[a, b]&=&[a, x\trr b]+[x\trr a,b].
\end{eqnarray}
\item[(3)] The following conditions are satisfied:
\begin{eqnarray}
\label{(27)}x\trr (ab)&=&(x\trr a)b+a(x\trr b),\\
\label{(28)}x\ppr[a, b]&=&[a, x\ppr b]+(x\trr a)b,\\
\label{(29)}[a, b]\ppl x&=&[a, b\ppl x]+b(x\trr a),
\end{eqnarray}
\end{enumerate}
for all $x\in {H}$ and $a, b\in {A}. $ In this case, we call $(A, \,  \ppr, \, \ppl, \trr)$ to be an $H$-noncommutative Poisson bimodule  algebra.
\end{definition}

\begin{definition}
Let ${H}$ and  ${A}$ be noncommutayive Poisson coalgebras. A coaction of ${H}$ on ${A}$ are three linear maps $\phi: V\to {H}\otimes V$, $\rho: V\to H\otimes {V}$ and $\gamma: V\to V\otimes {H}$ such that
 \begin{enumerate}
\item[(1)]$(A, \, \Delta_A, \, \rho, \, \gamma)$ is an $H$-bicomodule coalgebra over $(H, \, \Delta_H)$,  i.e.,
\begin{eqnarray}
(\id_H\otimes\Delta_A)\rho(a)=(\rho\otimes\id_A)\Delta_A(a),\\
(\Delta_A\ot \id_H)\gamma(a)=(\id_A\ot \gamma)\Delta_A(a),\\
(\gamma\ot \id_A)\Delta_A(a)=(\id_A\ot \rho)\Delta_A(a).
\end{eqnarray}
\item[(2)]$(A, \, \delta_A, \, \phi)$ is a left $H$-comodule Lie coalgebra over $(H, \, \delta_H)$,  i.e.,
\begin{eqnarray}
(\id_H\otimes\delta_A)\phi(a)=(\phi\otimes\id_A)\delta_A(a)+\tau_{12}(\id_A\otimes\phi)\delta_A(a).
\end{eqnarray}
\item[(3)] The following conditions are satisfied:
\begin{eqnarray}\label{comodcoalg01}
(\id_H \otimes \Delta_{A})\phi(a)&=&(\phi\otimes\id_A) \Delta_{A}(a)+\tau_{12}(\id_{A}\ot\phi) \Delta_{A}(a),\\
\label{comodcoalg02}
(\id_H\ot\delta_{A})\rho(a)&=&\tau_{12}(\id_A\otimes \rho) \delta_{A}(a)+(\phi\ot\id_{A}) \Delta_{A}(a),\\
\label{comodcoalg03}
(\id_A\ot \gamma)\delta_A(a)&=&(\delta_A\ot \id_H)\gamma(a)-\tau_{12}(\id_A\ot \tau\phi)\Delta_A(a).
\end{eqnarray}
\end{enumerate}
If we denote by  $\phi(a)=a_{\langle-1\rangle}\ot a_{\langle0\rangle}$, $\rho(a)=a\loi\ot a\loo$
and $\gamma(a)=a\loo\ot a\lmi$, then the above equations \eqref{comodcoalg01} and \eqref{comodcoalg02} can be written as
\begin{eqnarray}
\label{(18)}&&a_{\langle-1\rangle} \otimes \Delta_{A}\left(a_{\langle0\rangle}\right)=\phi\left(a_{1}\right) \otimes a_{2}
  +\tau_{12}(a\li\ot\phi(a\lii)),\\
\label{(19)}&&a_{(-1)}\ot\delta_{A}(a_{(0)})=\tau_{12}(a\bi\ot\rho(a\bii))+\phi(a\li)\ot a\lii,\\
\label{(20)}&&a\bi\ot \gamma(a\bii)=\delta_A(a\loo)\ot a\lmi-\tau_{12}(a\li\ot \tau\phi(a\lii),
\end{eqnarray}
for all $a\in {A}. $ In this case, we call $(A, \phi, \rho, \gamma)$ to be $H$-noncommutative Poisson bicomodule coalgebras.
\end{definition}

\begin{definition}
Let $({A}, \cdot)$ be a given noncommutative Poisson  algebra (Poisson coalgebra, Poisson  bialgebra), $E$ be a vector space.
An extending system of ${A}$ through $V$ is a noncommutative Poisson algebra (Poisson coalgebra, Poisson bialgebra) on $E$
such that $V$ a complement subspace of ${A}$ in $E$, the canonical injection map $i: A\to E, a\mapsto (a, 0)$  or the canonical projection map $p: E\to A, (a, x)\mapsto a$ is a noncommutative Poisson algebra(Poisson coalgebra, Poisson  bialgebra) homomorphism.
The extending problem is to describe and classify up to an isomorphism  the set of all noncommutative Poisson  algebra(Poisson coalgebra, Poisson  bialgebra) structures that can be defined on $E$.
\end{definition}

We remark that our definition of extending system of ${A}$ through $V$ contains not only extending structure in \cite{AM1,AM2,AM3}
but also the global extension structure in \cite{AM5}.
In fact, the canonical injection map $i: A\to E$ is a noncommutative Poisson (co)algebra homomorphism if and only if $A$ is a noncommutative Poisson sub(co)algebra of $E$.

\begin{definition}
Let ${A}$ be a noncommutative Poisson  algebra (Poisson coalgebra, Poisson  bialgebra), $E$  be a noncommutative Poisson algebra  (Poisson coalgebra, Poisson  bialgebra) such that
${A} $ is a subspace of $E$ and $V$ a complement of
${A} $ in $E$. For a linear map $\varphi: E \to E$ we
consider the diagram:
\begin{equation}\label{eq:ext1}
\xymatrix{
   0  \ar[r]^{} &A \ar[d]_{\id_A} \ar[r]^{i} & E \ar[d]_{\varphi} \ar[r]^{\pi} &V \ar[d]_{\id_V} \ar[r]^{} & 0 \\
   0 \ar[r]^{} & A \ar[r]^{i'} & {E} \ar[r]^{\pi'} & V \ar[r]^{} & 0.
   }
\end{equation}
where $\pi: E\to V$ are the canonical projection maps and $i: A\to E$ are the inclusion maps.
We say that $\varphi: E \to E$ \emph{stabilizes} ${A}$ if the left square of the diagram \eqref{eq:ext1} is  commutative.
Let $(E, \cdot)$ and $(E, \cdot')$ be two noncommutative Poisson  algebra (Poisson coalgebra, Poisson  bialgebra) structures on $E$. $(E, \cdot)$ and $(E, \cdot')$ are called \emph{equivalent}, and we denote this by $(E, \cdot) \equiv (E, \cdot')$, if there exists a noncommutative Poisson algebra (Poisson coalgebra, Poisson  bialgebra) isomorphism $\varphi: (E, \cdot)
\to (E, \cdot')$ which stabilizes ${A} $. Denote by $Extd(E, {A} )$ ($CExtd(E, {A} )$, $BExtd(E, {A} )$) the set of equivalent classes of noncommutative Poisson  algebra (Poisson coalgebra, Poisson  bialgebra) structures on $E$.
\end{definition}

\section{Braided noncommutative Poisson bialgebras}
In this section, we introduce  the concept of noncommutative Poisson-Hopf bimodules and braided noncommutative Poisson bialgebras which will be used in the following sections.

\subsection{Noncommutative Poisson-Hopf bimodules and braided noncommutative Poisson bialgebras}
\begin{definition}
Let $H$ be a noncommutative Poisson bialgebra. A noncommutative Poisson-Hopf bimodule over $H$ is a vector space $V$ endowed with linear maps
\begin{align*}
&\trr: H\otimes V \to V, \quad \ppr: H\otimes V \to V, \quad \ppl: V\otimes H \to V, \\
&\phi: V \to H \otimes V, \quad  \rho: V \to H\otimes V, \quad  \gamma: V \to V\otimes H,
\end{align*}
which are denoted by
\begin{eqnarray*}
&& \trr (x \otimes v) = x \trr v, \quad \ppr(x \otimes v) = x \ppr v, \quad \ppl(x \otimes v) = x \ppl v,\\
&& \phi (v)=\sum v_{\langle-1\rangle}\ot v_{\langle0\rangle}, \quad \rho (v) = \sum v\loi \ot v\loo,
\quad \gamma (v) = \sum v\loo \ot v\lmi,
\end{eqnarray*}
such that $V$ is simultaneously a bimodule, a bicomodule over $H$ and satisfying
 the following compatibility conditions
\begin{enumerate}
\item[(HM1)] $\phi(x\ppr v)=x v_{\langle-1\rangle}\ot v_{\langle0\rangle}+v\loi\ot (x\trr v\loo)-x\lii\ot (x\li\trr v)$,

\item[(HM2)] $\tau\phi(x\ppr v)=(x\ppr v_{\langle0\rangle})\ot v_{\langle-1\rangle}-v\loo\ot[x,v\lmi]-(x\bi\ppr v)\ot x\bii$,

\item[(HM3)] $\phi(v\ppl x)=v_{\langle-1\rangle} x\ot v_{\langle0\rangle}-x\li\ot (x\lii\trr v)+v\lmi\ot (x\trr v\loo),$

\item[(HM4)] $\tau\phi(v\ppl x)=(v\ppl x\bi)\ot x\bii-( v_{\langle0\rangle}\ppl x)\ot v_{\langle-1\rangle}+v\loo\ot [x, v\loi]$,
\end{enumerate}
\begin{enumerate}
\item[(HM5)]  $\rho(x\trr v)=[x, v\loi]\ot v\loo+v\loi\ot (x\trr v\loo)-x\bi\ot (v\ppl x\bii)$,

\item[(HM6)] $\rho(x \trr v)=x\li\ot (x\lii\trr v)+v_{\langle-1\rangle}\ot (x\ppr v_{\langle0\rangle})- v_{\langle-1\rangle}x\ot v_{\langle0\rangle}$,

\item[(HM7)] $\gamma(x\trr v)=(x\trr v\loo)\ot v\lmi+v\loo\ot[x, v\lmi]+(x\bi\ppr v)\ot x\bii$,

\item[(HM8)] $\gamma(x \trr v)=(x\li\trr v)\ot x\lii-v_{\langle0\rangle}\ot x v_{\langle-1\rangle}+(v_{\langle0\rangle}\ppl x)\ot v_{\langle-1\rangle}$,
\end{enumerate}
for all $x\in H$ and $v\in V$.
\end{definition}
We denote  the  category of noncommutative Poisson-Hopf bimodules over $H$ by ${}^{H}_{H}\mathcal{M}{}^{H}_{H}$.

\begin{definition} Let $H$  be a noncommutative Poisson bialgebra, $A$ be simultaneously a $H$-bimodule algebra (coalgebra) and $H$-bicomodule algebra (coalgebra).
We call $A$ to be a \emph{braided noncommutative Poisson bialgebra}, if the following conditions are  satisfied
\begin{enumerate}
\item[(BB1)]
$\delta_{A}(a b)=a\bi b\ot a\bii +a b\bi\ot b\bii +b\li\ot [a,b\lii] +a\lii\ot[b, a\li]+(a_{\langle-1\rangle}\ppr b)\ot a_{\langle0\rangle}\\
+(a\ppl b_{\langle-1\rangle})\ot b_{\langle0\rangle}-b\loo\ot (b\lmi\trr a)-a\loo\ot (a\loi\trr b),$
\item[(BB2)]
$\Delta_{A}([a, b])=[a, b\li]\ot b\lii+b\li\ot[a, b\lii]-a\bi\ot b a\bii+ a\bi b\ot a\bii\\
+a_{\langle0\rangle}\ot (b\ppl a_{\langle-1\rangle})+(a_{\langle-1\rangle}\ppr b)\ot a_{\langle0\rangle}-(b\loi\trr a)\ot b\loo-b\loo\ot (b\lmi\trr a).$
\end{enumerate}
\end{definition}

Now we construct noncommutative Poisson bialgebras from braided noncommutative Poisson  bialgebras.
Let $H$  be a noncommutative Poisson bialgebra, $A$ be a noncommutative Poisson algebra and a noncommutative Poisson coalgebra in ${}^{H}_{H}\mathcal{M}{}_{H}^{H}$.
We define multiplications and comultiplications on the direct sum vector space $E: =A \oplus H$ by
\begin{eqnarray}
\label{sp01}
&[(a, x), (b, y)]_{E}: =([a, b] +x\trr b-y \trr a, [x, y]), \\
\label{sp02}
&\delta_{E}(a, x): =\delta_{A}(a)+\phi(a)-\tau\phi(a)+\delta_{H}(x), \\
\label{sp03}
&(a, x) \cdot_{E} (b, y): =(a b+x\rightharpoonup b+a\ppl y, x y), \\
\label{sp04}
&\Delta_{E}(a, x): =\Delta_{A}(a)+\rho(a)+\gamma(a)+\Delta_{H}(x).
\end{eqnarray}
This is called biproduct of ${A}$ and ${H}$ which will be  denoted  by $A\lbiprod H$.

\begin{theorem} Let $H$  be a noncommutative Poisson bialgebra, $A$ be a noncommutative Poisson algebra and a noncommutative Poisson coalgebra in ${}^{H}_{H}\mathcal{M}{}_{H}^{H}$.
Then the biproduct $A\lbiprod H$ forms a noncommutative Poisson bialgebra if and only if  $A$ is a braided noncommutative Poisson bialgebra in ${}^{H}_{H}\mathcal{M}{}_{H}^{H}$.
\end{theorem}

\begin{proof}
First, it is obvious that $(A\lbiprod H, ~  [ , ])$  and  $(A\lbiprod H, ~ \cdot)$  are respectively a Lie algebra and an associative algebra. It is easy to prove that $A\lbiprod H$ is a noncommutative Poisson algebra and a noncommutative Poisson coalgebra with the multiplications \eqref{sp01} and \eqref{sp03} and comultiplications  \eqref{sp02} and \eqref{sp04}. Now we show the compatibility conditions:
$$
\begin{aligned}
\delta_{E}((a, x)\cdot_{E} (b,y))=& (a, x)\cdot_{E} (b, y)\bi \ot (b, y)\bii +(a, x)\bi \cdot_{E} (b, y) \ot (a, x)\bii \\
&+(b, y)\li \ot [(a, x), (b, y)\lii]_{E} +(a, x)\lii \ot [(b, y), (a, x)\li]_{E},
 \end{aligned}
$$
$$
\begin{aligned}
\Delta_{E} ([(a, x), (b, y)]_{E})=&[(a, x), (b, y)\li]_{E} \ot (b, y)\lii +(b, y)\li \ot [(a, x), (b, y)\lii]_{E} \\
&-(a, x)\bi \ot (b, y)\cdot_{E} (a, x)\bii+ (a, x)\bi\cdot_{E} (b, y) \ot (a, x)\bii .
\end{aligned}
$$
By direct computations, the left hand side of the first equation is equal to
\begin{eqnarray*}
&&\delta_{E}((a, x)\cdot_{E} (b, y))\\
&=&\delta_{E}(a b+x\rightharpoonup b +a\ppl y, x y)\\
&=&\delta_{A}(a b)+ \delta_{A}(x\rightharpoonup b) +\delta_{A}(a\ppl y) +\phi(a b) +\phi(x\rightharpoonup b) +\phi(a\ppl y)\\
&& -\tau\phi(a b) -\tau\phi(x\rightharpoonup b) -\tau\phi(a\ppl y) +\delta_{H}(x y),
\end{eqnarray*}
and the right hand side is equal to
\begin{eqnarray*}
&&(a, x)\cdot_{E} (b, y)\bi \ot (b, y)\bii +(a, x)\bi \cdot_{E} (b, y) \ot (a, x)\bii \\
&&+(b, y)\li \ot [(a, x), (b, y)\lii]_{E} +(a, x)\lii \ot [(b, y), (a, x)\li]_{E}\\
&=&a b\bi\ot b\bii+(x\ppr b\bi)\ot b\bii+(a\ppl b_{\langle-1\rangle})\ot b_{\langle0\rangle}+x b_{\langle-1\rangle}\ot b_{\langle0\rangle}\\
&&-a b_{\langle0\rangle}\ot b_{\langle-1\rangle}-(x\ppr b_{\langle0\rangle})\ot b_{\langle-1\rangle}+(a\ppl y\bi)\ot y\bii +x y\bi\ot y\bii\\
&&+a\bi b\ot a\bii +( a\bi\ppl y)\ot a\bii +(a_{\langle-1\rangle}\ppr b)\ot a_{\langle0\rangle} +a_{\langle-1\rangle} y\ot a_{\langle0\rangle}\\
&&- a_{\langle0\rangle} b\ot a_{\langle-1\rangle}-( a_{\langle0\rangle}\ppl y)\ot a_{\langle-1\rangle} +(x\bi\ppr b)\ot x\bii +x\bi y\ot x\bii\\
&&+b\li\ot [a, b\lii]+b\li\ot (x\trr b\lii)+b\loi\ot [a, b\loo]+b\loi\ot (x\trr b\loo)\\
&&+b\loo\ot [x, b\lmi]-b\loo\ot (b\lmi\trr a)+y\li\ot [x, y\lii]-y\li\ot (y\lii\trr a)\\
&&+a\lii\ot [b, a\li]+a\lii\ot (y\trr a\li)+a\loo\ot [y, a\loi]-a\loo\ot (a\loi\trr b)\\
&&+a\lmi\ot[b, a\loo]+a\lmi\ot (y\trr a\loo)+x\lii\ot[y, x\li]-x\lii\ot (x\li\trr b).
\end{eqnarray*}
Then the two sides are equal to each other if and only if

(1)$\delta_{A}(a b)=a\bi b\ot a\bii +a b\bi\ot b\bii +b\li\ot [a,b\lii] +a\lii\ot[b, a\li]+(a_{\langle-1\rangle}\ppr b)\ot a_{\langle0\rangle}$

$\qquad+(a\ppl b_{\langle-1\rangle})\ot b_{\langle0\rangle}-b\loo\ot (b\lmi\trr a)-a\loo\ot (a\loi\trr b),$

(2) $\delta_{A}(x \ppr b)=(x\ppr b\bi)\ot b\bii+b\li\ot (x\trr b\lii)$,

(3) $\delta_{A}(a\ppl y)=( a\bi\ppl y)\ot a\bii+a\lii\ot (y\trr a\li)$,

(4) $\phi(a b)=b\loi\ot[a, b\loo]+a\lmi\ot[b, a\loo]$,

(5) $\tau\phi(a b)=a_{\langle0\rangle} b\ot a_{\langle-1\rangle}+a b_{\langle0\rangle}\ot b_{\langle-1\rangle}$,

(6) $\phi(x\ppr b)=x b_{\langle-1\rangle}\ot b_{\langle0\rangle}+b\loi\ot (x\trr b\loo)-x\lii\ot (x\li\trr b)$,

(7) $\tau\phi(x\ppr b)=(x\ppr b_{\langle0\rangle})\ot b_{\langle-1\rangle}-b\loo\ot[x,b\lmi]-(x\bi\ppr b)\ot x\bii$,

(8) $\phi(a\ppl y)=a_{\langle-1\rangle} y\ot a_{\langle0\rangle}-y\li\ot (y\lii\trr a)+a\lmi\ot (y\trr a\loo)$,

(9) $\tau\phi(a\ppl y)=(a\ppl y\bi)\ot y\bii-( a_{\langle0\rangle}\ppl y)\ot a_{\langle-1\rangle}+a\loo\ot [y, a\loi]$.

For the second equation, the left hand side is equal to
$$\begin{aligned}
&\Delta_{E}[(a, x), (b, y)]_{E}\\
=&\Delta_{E}([a, b]+x\trr b -y\trr a , [x, y])\\
=&\Delta_{A}([a, b])+ \Delta_{A}(x\trr b) -\Delta_{A}(y\trr a) +\rho([a, b]) +\rho(x\trr b) -\rho(y\trr a)\\
& +\gamma([a, b]) +\gamma(x\trr b) -\gamma(y\trr a) +\Delta_{H}([x, y]),
\end{aligned}
$$
and the right hand side is equal to
\begin{eqnarray*}
&&[(a, x), (b, y)\li]_{E} \ot (b, y)\lii +(b, y)\li \ot [(a, x), (b, y)\lii]_{E} \\
&&-(a, x)\bi \ot (b, y)\cdot_{E} (a, x)\bii+ (a, x)\bi\cdot_{E} (b, y) \ot (a, x)\bii\\
&=&[a, b\li]\ot b\lii+(x\trr b\li)\ot b\lii+[x, b\loi]\ot b\loo-(b\loi\trr a)\ot b\loo\\
&&+[a, b\loo]\ot b\lmi +(x\trr b\loo)\ot b\lmi+[x, y\li]\ot y\lii-(y\li\trr a)\ot y\lii\\
&&+b\li\ot [a, b\lii]+b\li\ot (x\trr b\lii)+b\loi\ot [a, b\loo]+b\loi\ot( x\trr b\loo)\\
&&+b\loo\ot[x, b\lmi]-b\loo\ot (b\lmi\trr a)+y\li\ot[x, y\lii]-y\li\ot (y\lii\trr a)\\
&&-a\bi\ot b a\bii-a\bi\ot (y\ppr a\bii)-a_{\langle-1\rangle}\ot b a_{\langle0\rangle}-a_{\langle-1\rangle}\ot (y\ppr a_{\langle0\rangle})\\
&&+a_{\langle0\rangle}\ot y a_{\langle-1\rangle}
+a_{\langle0\rangle}\ot (b\ppl a_{\langle-1\rangle})-x\bi\ot y x\bii-x\bi\ot (b\ppl x\bii)\\
&&+ a\bi b\ot a\bii+(a\bi\ppl y)\ot a\bii+(a_{\langle-1\rangle}\ppr b)\ot a_{\langle0\rangle}
+a_{\langle-1\rangle} y\ot a_{\langle0\rangle}\\
&&-a_{\langle0\rangle} b\ot a_{\langle-1\rangle}-(a_{\langle0\rangle}\ppl y)\ot a_{\langle-1\rangle}
+x\bi y\ot x\bii+(x\bi\ppr b)\ot x\bii.
\end{eqnarray*}
Then the two sides are equal to each other if and only if

(10) $\Delta_{A}([a, b])=[a, b\li]\ot b\lii+b\li\ot[a, b\lii]-a\bi\ot b a\bii+ a\bi b\ot a\bii$

$\qquad+a_{\langle0\rangle}\ot (b\ppl a_{\langle-1\rangle})+(a_{\langle-1\rangle}\ppr b)\ot a_{\langle0\rangle}-(b\loi\trr a)\ot b\loo-b\loo\ot (b\lmi\trr a),$

(11) $\Delta_{A}(x \trr b)=(x\trr b\li)\ot b\lii+b\li\ot (x\trr b\lii)$,

(12) $\Delta_{A}(y \trr a)=a\bi\ot (y\ppr a\bii)-(a\bi\ppl y)\ot a\bii$,

(13) $\rho([a,b])=b\loi\ot[a, b\loo]-a_{\langle-1\rangle}\ot b a_{\langle0\rangle}$,

(14) $\gamma([a,b])=[a, b\loo]\ot b\lmi-a_{\langle0\rangle} b\ot a_{\langle-1\rangle}$,

(15) $\rho(x\trr b)=[x, b\loi]\ot b\loo+b\loi\ot (x\trr b\loo)-x\bi\ot (b\ppl x\bii)$,

(16) $\rho(y \trr a)=y\li\ot (y\lii\trr a)+a_{\langle-1\rangle}\ot (y\ppr a_{\langle0\rangle})- a_{\langle-1\rangle}y\ot a_{\langle0\rangle}$,

(17) $\gamma(x\trr b)=(x\trr b\loo)\ot b\lmi+b\loo\ot[x, b\lmi]+(x\bi\ppr b)\ot x\bii$,

(18) $\gamma(y \trr a)=(y\li\trr a)\ot y\lii-a_{\langle0\rangle}\ot y a_{\langle-1\rangle}+(a_{\langle0\rangle}\ppl y)\ot a_{\langle-1\rangle}$.

\noindent From (2)--(5) and (11)--(14) we have that $A$ is a noncommutative Poisson algebra and a noncommutative Poisson coalgebra in ${}^{H}_{H}\mathcal{M}{}_{H}^{H}$,
from (6)--(9) and (15)--(18) we get that $A$ is a noncommutative Poisson-Hopf bimodule over $H$, and (1) together with (19) are the conditions for $A$ to be a braided noncommutative Poisson bialgebra.

The proof is completed.
\end{proof}

\section{Unified product of noncommutative Poisson bialgebras}
\subsection{Matched pair of braided noncommutative Poisson bialgebras}

In this section, we construct noncommutative Poisson bialgebra from the double cross biproduct of a matched pair of braided noncommutative Poisson bialgebras.

Let $A, H$ be both noncommutative Poisson  algebras and noncommutative Poisson  coalgebras.  For $a, b\in A$, $x, y\in H$,  we denote linear maps
\begin{align*}
&\ppr: H \otimes A \to A, \quad \ppl: A\otimes H\to A,\\
&\ra: A\ot H \to H, \quad\leftarrow: H\ot A\to H,\\
&\trr: H\otimes A \to A, \quad \trl: H\otimes A \to H,\\
&\rho: A  \to H\otimes A, \quad  \gamma: A \to A \otimes H,\\
&\al : H\to A\ot H, \quad\beta: H\to H\ot A,\\
&\phi: A \to H \otimes A, \quad  \psi: H \to H\otimes A,
\end{align*}
by
\begin{eqnarray*}
&& \ppr (x \otimes a) = x \ppr a, \quad \ppl(a\otimes x) = a \ppl x, \\
&& \ra(a\ot x) = a\ra x, \quad \leftarrow(x\ot a) = x\leftarrow a,\\
&& \trr (x \otimes a) = x \trr a, \quad \trl(x \otimes a) = x \triangleleft a, \\
&& \rho (a)=\sum a\loi\ot a\loo, \quad \gamma (a) = \sum a\loo \ot a\lmi,\\
&& \al (x)=\sum x\qoi\ot x\qoo, \quad \beta(x)=\sum x\qoo\ot x\qi,\\
&& \phi (a)=\sum a_{\langle-1\rangle}\ot a_{\langle0\rangle}, \quad \psi (x) = \sum x_{\langle0\rangle} \ot x_{\langle1\rangle}.
\end{eqnarray*}

\begin{definition}(\cite{Bai11})
A \emph{matched pair} of  noncommutative Poisson algebras is a system $(A, \, {H}, \, \trl,\\
 \, \trr, \, \ppl,
 \, \ppr, \, \leftarrow, \, \ra)$ consisting of two noncommutative Poisson algebras $A$ and ${H}$ and six bilinear maps $\triangleleft : {H}\otimes A\to {H}$, $\trr : {H} \otimes A
\to A$, $\ppl: A \otimes {H} \to A$, $\ppr: H\otimes {A} \to {A}$, $\leftarrow: H\ot A\to H$, $\ra: A\ot H \to H$ such that  $(A, H, \trr, \trl)$  is a matched pair of Lie algebras,
 $(A, H, \ppr, \ppl, \ra, \leftarrow )$ is a matched pair of associative algebras,
and  the following compatibility conditions are satisfied for all $a, b\in A$, $x, y \in {H}$:
\begin{enumerate}
\item[(AM1)] $x\ppr[a, b]=[a, x\ppr b]-(x\leftarrow b)\trr a+(x\trr a)b+(x\trl a)\ppr b$,

\item[(AM2)] $[a, b]\ppl x=[a, b\ppl x]-(b\rightarrow x)\trr a+b(x\trr a)+b\ppl (x\trl a)$,

\item[(AM3)] $x\trr(ab)=(x\trr a)b+(x\trl a)\ppr b+a(x\trr b)+a\ppl (x\trl b)$,

\item[(AM4)]$[x, y]\leftarrow a=[x, y\leftarrow a]+x\trl(y\ppr a)-y(x\trl a)-y\leftarrow(x\trr a)$,

\item[(AM5)] $a\rightarrow [x, y]=[x, a\rightarrow y]+x\trl(a\ppl y)-(x\trl a)y-(x\trr a)\rightarrow y$,

\item[(AM6)] $(xy)\trl a=(x\trl a)y+(x\trr a)\rightarrow y+x(y\trl a)+x\leftarrow (y\trr a)$.
\end{enumerate}
\end{definition}

\begin{lemma}(\cite{Bai11})
Let  $(A, \, {H}, \, \trl, \, \trr, \, \ppl, \, \ppr, \, \leftarrow, \, \ra)$ be a matched pair of noncommutative Poisson algebras.
Then $A \, \bowtie {H}: = A \oplus  {H}$, as a vector space, with the multiplication defined for any $a, b\in A$ and $x, y\in {H}$ by
$$
\begin{aligned}
&[(a, x), (b, y)]_{E}: =([a, b] +x\trr b-y \trr a, [x, y]+x\trl b-y\trl a), \\
&(a, x) \cdot_{E} (b, y): =(a b+x\rightharpoonup b+a\ppl y, x y+x\leftarrow b+a\rightarrow y),
\end{aligned}
$$
is a noncommutative Poisson algebra which is called the \emph{bicrossed product} associated to the matched pair of noncommutative Poisson  algebras $A$ and ${H}$.
\end{lemma}

Now we introduce the notion of matched pairs of noncommutative Poisson coalgebras, which is the dual version of  matched pairs of noncommutative Poisson algebras.

\begin{definition} A \emph{matched pair} of noncommutative Poisson  coalgebras is a system $(A, \, {H}, \, \phi, \,\\ \psi, \, \rho, \, \gamma, \, \al, \, \beta)$ consisting
of two noncommutative Poisson  coalgebras $A$ and ${H}$ and six bilinear maps
$\phi: {A}\to H\otimes A$, $\psi: {H}\to H \otimes A$, $\rho: A\to H\otimes {A}$, $\gamma: A \to {A} \ot {H}$,
$\al: H\ot A\ot H$, $\beta: H\to H\to A$
such that
 $(A, \, {H}, \, \phi, \, \psi)$ is a matched pair of Lie coalgebras,
 $(A, \, {H}, \, \rho, \, \gamma, \, \al, \, \beta)$ is a matched pair of coassociative coalgebras,
and the following compatibility conditions are satisfied for any $a\in A$, $x\in {H}$:
\begin{enumerate}
\item[(CM1)]
$a\bi\ot\rho(a\bii)-a_{\langle0\rangle}\ot\beta(a_{\langle-1\rangle})=-\tau\phi(a\li)\ot a\lii-\tau\psi(a\loi)\ot a\loo
+\tau_{12}(a\loi\ot \delta_A(a\loo))$,

\item[(CM2)]
$a_{\langle-1\rangle}\ot\Delta_A(a_{\langle0\rangle})=\phi(a\li)\ot a\lii+\psi(a\loi)\ot a\loo+\tau_{12}(a\li\ot\phi(a\lii))+\tau_{12}(a\loo\ot\psi(a\lmi))$,

\item[(CM3)]
$a\bi\ot \gamma(a\bii)-a_{\langle0\rangle}\ot \alpha(a_{\langle-1\rangle})
=\delta_A(a\loo)\ot a\lmi-\tau_{12}(a\li\ot \tau\phi(a\lii))-\tau_{12}(a\loo\ot \tau\psi(a\lmi))$,

\item[(CM4)]
$x_{[1]} \otimes \beta\left(x_{[2]}\right)+x_{\langle0\rangle}\ot\rho(x_{\langle1\rangle})
=\delta_H(x\qoo)\ot x\qi+\tau_{12}(x\li\ot\psi(x\lii))+\tau_{12}(x\qoo\ot\phi(x\qi))$,

\item[(CM5)]
$ x_{\langle1\rangle}\ot\Delta_{H}(x_{\langle0\rangle})=\tau\psi(x\li)\ot x\lii+\tau\phi(x\qoi)\ot x\qoo
+\tau_{12}(x\li\ot\tau\psi(x\lii))+\tau_{12}(x\qoo\ot\tau\phi(x\qi))$,

\item[(CM6)]
$x_{[1]} \otimes \alpha\left(x_{[2]}\right)+x_{\langle0\rangle}\ot\gamma(x_{\langle1\rangle})
=\psi(x\li)\ot x\lii+\phi(x\qoi)\ot x\qoo+\tau_{12}(x\qoi\ot \delta_H(x\qoo))$.
\end{enumerate}
\end{definition}

\begin{lemma}\label{lem1} Let $(A, H)$ be a matched pair of noncommutative Poisson coalgebras. We define $E=A\lrcoprod H$ as the vector space $A\oplus H$ with   comultiplication
$$\Delta_{E}(a)=(\Delta_{A}+\rho+\gamma)(a), \quad\Delta_{E}(x)=(\Delta_{H}+\alpha+\beta)(x),$$
$$\delta_E(a)=(\delta_A+\phi-\tau\phi)(a), \quad\delta_E(x)=(\delta_H(x)+\psi-\tau\psi)(x),$$
that is
$$
\begin{aligned}
&\Delta_{E}(a)=\sum a\li \ot a\lii+\sum a\loi \ot a\loo+\sum a\mo\ot a\lmi, \\
&\Delta_{E}(x)=\sum x\li \ot x\lii+\sum  x\qoi \ot x\qoo+\sum x\qoo \ot x\qi,\\
&\delta_E(a)=\sum a\bi\ot a\bii+\sum a_{\langle-1\rangle}\ot a_{\langle0\rangle}-\sum a_{\langle0\rangle}\ot a_{\langle-1\rangle},\\
&\delta_E(x)=\sum x\bi\ot x\bii+\sum x_{\langle0\rangle}\ot x_{\langle1\rangle}-\sum x_{\langle1\rangle}\ot x_{\langle0\rangle}.
\end{aligned}
$$
Then  $A\lrcoprod H$ is a noncommutative Poisson  coalgebra which is called the \emph{bicrossed coproduct} associated to the matched pair of noncommutative Poisson coalgebras $A$ and $H$.
\end{lemma}

The proof of the above Lemma \ref{lem1} is omitted since it is by direct computations.
In the following of this section, we construct noncommutative Poisson bialgebra from the double cross biproduct of a pair of braided noncommutative Poisson bialgebras.
First we generalize the concept of noncommutative Poisson-Hopf bimodule to the case of $A$ is not necessarily a noncommutative Poisson bialgebra.
But by abuse of notation, we also call it noncommutative Poisson-Hopf module.

\begin{definition}
Let $A$ be simultaneously a noncommutative Poisson  algebra and a noncommutative Poisson  coalgebra.
 If $H$ is an $A$-bimodule, $A$-bicomodule and satisfying
\begin{enumerate}
\item[(HM1')]$\psi(x\leftarrow a)=(x_{\langle0\rangle}\leftarrow a)\ot x_{\langle1\rangle}+(x\leftarrow a\bi)\ot a\bii+x\qoo\ot[a, x\qoi]$,
\item[(HM2')]  $\tau\psi(x\leftarrow a)=x_{\langle1\rangle} a\ot x_{\langle0\rangle}+x\qi\ot(x\qoo\trl a)-a\li\ot(x\trl a\lii)$
\item[(HM3')]  $\psi(a\rightarrow x)=( a\rightarrow x_{\langle0\rangle})\ot x_{\langle1\rangle}+x\qoo\ot[a, x\qi]-(a\bi\rightarrow x)\ot a\bii$,
\item[(HM4')]    $\tau\psi(a\rightarrow x)=ax_{\langle1\rangle}\ot x_{\langle0\rangle}+x\qoi\ot(x\qoo\trl a)
-a\lii\ot(x\trl a\li)$,
\item[(HM5')] $\beta(x\trl a)=(x\trl a\li)\ot a\lii
-x_{\langle0\rangle}\ot ax_{\langle1\rangle}+(x_{\langle0\rangle}\leftarrow a)\ot x_{\langle1\rangle}$,
\item[(HM6')] $\beta(x\trl a)=(x\qoo\trl a)\ot x\qi-x\qoo\ot[a, x\qi]-(a\bi\rightarrow x)\ot a\bii$,
\item[(HM7')] $\al(x\trl a)=a\li\ot(x\trl a\lii)
    +x_{\langle1\rangle}\ot (a\rightarrow x_{\langle0\rangle})- x_{\langle1\rangle}a\ot x_{\langle0\rangle}$,
\item[(HM8')] $\al(x\trl a)= x\qoi\ot (x\qoo\trl a)-[a,x\qoi]\ot x\qoo
    +a\bi\ot(x\leftarrow a\bii)$,
\end{enumerate}
then $H$ is called a noncommutative Poisson-Hopf bimodule over $A$.
\end{definition}
We denote  the  category of noncommutative Poisson-Hopf bimodules over $A$ by ${}^{A}_{A}\mathcal{M}{}^{A}_{A}$.

\begin{definition}
Let $A$ be a noncommutative Poisson algebra and noncommutative Poisson  coalgebra and  $H$ is a noncommutative Poisson-Hopf bimodule over $A$. If $H$ is a noncommutative Poisson algebra and a noncommutative Poisson coalgebra in ${}^{A}_{A}\mathcal{M}^{A}_{A}$, then we call $H$ a \emph{braided noncommutative Poisson  bialgebra} over $A$, if the following conditions are satisfied:
\begin{enumerate}
\item[(BB1')]
$\delta_{H}(x y)=x\bi y\ot x\bii-(x_{\langle1\rangle}\rightarrow y)\ot x_{\langle0\rangle}+x y\bi\ot y\bii
-(x\leftarrow y_{\langle1\rangle})\ot y_{\langle0\rangle}\\
+y\li\ot[x, y\lii]+y\qoo\ot(x\trl y\qi)+x\lii\ot[y, x\li]+x\qoo\ot(y\trl x\qoi),$
\item[(BB2')]
$\Delta_{H}([x, y])=[x,y\li]\ot y\lii+(x\trl y\qoi)\ot y\qoo+y\li\ot[x, y\lii]+y\qoo\ot(x\trl y\qi)\\
+x\bi y\ot x\bii-(x_{\langle1\rangle}\rightarrow y)\ot x_{\langle0\rangle}-x\bi\ot yx\bii-x_{\langle0\rangle}\ot(y\leftarrow x_{\langle1\rangle}).$
\end{enumerate}
\end{definition}

\begin{definition}\label{def:dmp}
Let $A, H$ be both noncommutative Poisson  algebras and noncommutative Poisson  coalgebras. If  the following conditions hold:
\begin{enumerate}
\item[(DM1)] $\phi(a b)= (a_{\langle-1\rangle}\leftarrow b)\ot  a_{\langle0\rangle}+( a\rightarrow b_{\langle-1\rangle})\ot  b_{\langle0\rangle}+b\loi\ot[a, b\loo]+a\lmi\ot[b, a\loo]$,
\item[(DM2)] $\tau\phi(ab)= a_{\langle0\rangle}b\ot  a_{\langle-1\rangle}+a b_{\langle0\rangle}\ot  b_{\langle-1\rangle}+b\loo\ot(b\lmi\trl a)+a\loo\ot(a\loi\trl b)$,
\item[(DM3)] $\psi(xy)= x_{\langle0\rangle} y\ot x_{\langle1\rangle}+x  y_{\langle0\rangle}\ot  y_{\langle1\rangle}
+ y\qoo\ot(x\trr y\qi)+x\qoo\ot(y\trr x\qoi)$,
\item[(DM4)] $\tau\psi(xy)=(x_{\langle1\rangle}\ppl y)\ot x_{\langle0\rangle}+(x\ppr y_{\langle1\rangle})\ot y_{\langle0\rangle}-y\qoi\ot[x, y\qoo]-x\qi\ot[y, x\qoo]$,
\item[(DM5)]  $\delta_A(x\ppr b)=( x_{\langle0\rangle}\ppr b)\ot  x_{\langle1\rangle}+(x\ppr b\bi)\ot b\bii-x\qi\ot(x\qoo\trr b)+b\li\ot(x\trr b\lii)$,
\item[(DM6)]  $\delta_A(a\ppl y)=(a\ppl y_{\langle0\rangle})\ot y_{\langle1\rangle}
+(a\bi\ppl y)\ot a\bii-y\qoi\ot(y\qoo\trr a)+a\lii\ot (y\trr a\li)$,
\item[(DM7)] $\delta_H(x\leftarrow b)=(x\bi\leftarrow b)\ot x\bii-(x\leftarrow b_{\langle0\rangle})\ot b_{\langle-1\rangle}+b\loi\ot(x\trl b\loo)-x\lii\ot(x\li\trl b)$,
\item[(DM8)] $\delta_H(a\rightarrow y)=(a\rightarrow y\bi)\ot y\bii-(a_{\langle0\rangle}\rightarrow y)\ot a_{\langle-1\rangle}+a\lmi\ot(y\trl a\loo)-y\li\ot(y\lii\trl a)$,
\item[(DM9)] $\phi(x\ppr b)+\psi(x\leftarrow b)=(x_{\langle0\rangle}\leftarrow b)\ot x_{\langle1\rangle}+(x\leftarrow b\bi)\ot b\bii
+x b_{\langle-1\rangle}\ot b_{\langle0\rangle}\\
+b\loi\ot(x\trr b\loo)-x\lii\ot(x\li\trr b)+x\qoo\ot[b, x\qoi]$,
\item[(DM10)] $\tau\phi(x\ppr b)+\tau\psi(x\leftarrow b)=x_{\langle1\rangle} b\ot x_{\langle0\rangle}+(x\ppr b_{\langle0\rangle})\ot b_{\langle-1\rangle}+x\qi\ot(x\qoo\trl b)\\
-(x\bi\ppr b)\ot x\bii-b\loo\ot[x, b\lmi]-b\li\ot(x\trl b\lii)$,
\item[(DM11)] $\phi(a\ppl y)+\psi(a\rightarrow y)=a_{\langle-1\rangle}y\ot a_{\langle0\rangle}+( a\rightarrow y_{\langle0\rangle})\ot y_{\langle1\rangle}+y\qoo\ot[a, y\qi]\\
-(a\bi\rightarrow y)\ot a\bii+a\lmi\ot(y\trr a\loo)-y\li\ot(y\lii\trr a)$,
\item[(DM12)] $\tau\phi(a\ppl y)+\tau\psi(a\rightarrow y)=ay_{\langle1\rangle}\ot y_{\langle0\rangle}+( a_{\langle0\rangle}\ppl y)\ot a_{\langle-1\rangle}+y\qoi\ot(y\qoo\trl a)\\
-(a\ppl y\bi)\ot y\bii-a\loo\ot[y, a\loi]-a\lii\ot(x\trl a\li)$,
\item[(DM13)] $\rho([a, b])= (a_{\langle-1\rangle}\leftarrow b)\ot  a_{\langle0\rangle}-( b_{(-1)}\trl a)\ot  b_{(0)}
+b\loi\ot[a, b\loo]-a_{\langle-1\rangle}\ot ba_{\langle0\rangle}$,
\item[(DM14)] $\beta([x, y])= [x, y\qoo]\ot y\qi+y\qoo\ot (x\trr y\qi)
-x_{\langle0\rangle}\ot(y\ppr x_{\langle1\rangle})+x_{\langle0\rangle}y\ot x_{\langle1\rangle}$,
\item[(DM15)] $\gamma([a ,b])= a_{\langle0\rangle}\ot(b\rightarrow a_{\langle-1\rangle})- b_{(0)}\ot( b_{(1)}\trl a)
+[a,b\loo]\ot b\lmi +(a_{\langle-1\rangle}\leftarrow b)\ot a_{\langle0\rangle}$,
\item[(DM16)] $\al([x,y])= y\qoi\ot[x,y\qoo]+ (x\trr y\qoi)\ot y\qoo
-(x_{\langle1\rangle}\ppl y)\ot x_{\langle0\rangle}+x_{\langle1\rangle}\ot yx_{\langle0\rangle}$,
\item[(DM17)]  $\Delta_A(x\trr b)=(x\trr b\li)\ot b\lii+b\li\ot(x\trr b\lii)+( x_{\langle0\rangle}\ppr b)\ot  x_{\langle1\rangle}+x_{\langle1\rangle}\ot (b\ppl x_{\langle0\rangle})$,
\item[(DM18)]  $\Delta_A(y\trr a)=-(a\bi\ppl y)\ot a\bii+a\bi\ot(y\ppr a\bii)+( y\qoo\trr a)\ot  y\qi+y\qoi\ot (y\qoo\trr a)$,
\item[(DM19)] $\Delta_H(x\trl b)=(x\trl b\loo)\ot b\lmi+b\loi\ot(x\trl b\loo)+(x\bi\leftarrow b)\ot x\bii
    -x\bi\ot(b\rightarrow x\bii)$,
\item[(DM20)] $\Delta_H(y\trl a)=(y\li\trl a)\ot y\lii+y\li\ot(y\lii\trl a)+(a_{\langle0\rangle}\rightarrow y)\ot a_{\langle-1\rangle}
    +a_{\langle-1\rangle}\ot(y\leftarrow a_{\langle0\rangle})$,
\item[(DM21)] $\rho(x\trr b)+\beta(x\trl b)=(x\trl b\li)\ot b\lii+[x, b\loi]\ot b\loo+b\loi\ot(x\trr b\loo)
-x_{\langle0\rangle}\ot bx_{\langle1\rangle}\\
    +(x_{\langle0\rangle}\leftarrow b)\ot x_{\langle1\rangle}-x\bi\ot(b\ppl x\bii)$,
\item[(DM22)] $\rho(y\trr a)+\beta(y\trl a)=(y\qoo\trl a)\ot y\qi-y\qoo\ot[a, y\qi]-(a\bi\rightarrow y)\ot a\bii\\
    -a_{\langle-1\rangle}y\ot a_{\langle0\rangle}+y\li\ot(y\lii\trr a)
+a_{\langle-1\rangle}\ot (y\ppr a_{\langle0\rangle})$,
\item[(DM23)] $\gamma(x\trr b)+\al(x\trl b)=b\li\ot(x\trl b\lii)+ b\loo\ot[x,b\lmi]+(x\trr b\loo)\ot b\lmi\\
    +x_{\langle1\rangle}\ot (b\rightarrow x_{\langle0\rangle})+(x\bi\ppr b)\ot x\bii
- x_{\langle1\rangle}b\ot x_{\langle0\rangle}$,
\item[(DM24)] $\gamma(y\trr a)+\al(y\trl a)= y\qoi\ot (y\qoo\trl a)-[a,y\qoi]\ot y\qoo+(y\li\trr a)\ot y\lii-a_{\langle0\rangle}\ot ya_{\langle-1\rangle}\\
    +a\bi\ot(y\leftarrow a\bii)
+(a_{\langle0\rangle}\ppl y)\ot a_{\langle-1\rangle}$,
\end{enumerate}
\noindent then $(A, H)$ is called a \emph{double matched pair}.
\end{definition}

\begin{theorem}\label{main1}Let $(A, H)$ be matched pair of noncommutative Poisson  algebras and noncommutative Poisson  coalgebras,
$A$ is  a  braided noncommutative Poisson bialgebra in
${}^{H}_{H}\mathcal{M}{}^{H}_{H}$, $H$ is  a braided noncommutative Poisson bialgebra in
${}^{A}_{A}\mathcal{M}{}^{A}_{A}$. If we define the double cross biproduct of $A$ and
$H$, denoted by $A\lrbiprod H$, $A\lrbiprod H=A\bowtie H$ as
noncommutative Poisson algebra, $A\lrbiprod H=A\lrcoprod H$ as noncommutative  Poisson coalgebra, then
$A\lrbiprod H$ become a noncommutative Poisson bialgebra if and only if  $(A, H)$ form a double matched pair.
\end{theorem}
The proof of the above Theorem \ref{main1} is omitted since it is a special case of Theorem \ref{main2} in next subsection. 

\subsection{Cocycle bicrossproduct noncommutative Poisson bialgebras}

In this section, we construct cocycle bicrossproduct noncommutative Poisson bialgebras, which is a generalization of double cross biproduct.

Let $A, H$ be both noncommutative Poisson  algebras and noncommutative Poisson  coalgebras. For $a, b\in A$, $x, y\in H$,  we denote linear maps
\begin{align*}
&\sigma: H\otimes H \to A, \quad \theta: A\otimes A \to H,\\
&\omega: H\otimes H \to A, \quad \nu: A\otimes A \to H,\\
&p: A  \to H\otimes H, \quad  q: H \to A\otimes A,\\
&s: A  \to H\otimes H, \quad  t: H \to A\otimes A,
\end{align*}
by
\begin{eqnarray*}
&& \sigma (x, y)  \in  A, \quad \theta(a, b) \in H,\\
&& \omega (x, y)  \in  A, \quad \nu(a, b) \in H,\\
&& p(a)=\sum a_{1p}\ot  a_{2p}, \quad q(x) = \sum x_{1q} \ot x_{2q},\\
&& s(a)=\sum a_{1s}\ot  a_{2s}, \quad t(x) = \sum x_{1t} \ot x_{2t}.
\end{eqnarray*}
A pair of bilinear maps $\si, \omega: H\ot H\to A$ are called cocycles on $H$ if
\begin{enumerate}
\item[(CC1)] $x\trr\omega(y, z)+\sigma(x, yz)=\sigma(x,y)\ppl z+\omega([x, y], z)+y\ppr\sigma(x, z)+\omega(y, [x, z]).$
\end{enumerate}
A pair of bilinear maps $\theta, \nu: A\ot A\to H$ are called cocycles on $A$ if
\begin{enumerate}
\item[(CC2)] $\theta(a, bc)-\nu(b, c)\trl a=\theta(a, b)\leftarrow c+\nu([a, b], c)+b\rightarrow \theta(a, c)+\nu(b, [a, c]).$
\end{enumerate}
A pair of bilinear maps $p, s: A\to H\ot H$ are called cycles on $A$ if
\begin{enumerate}
\item[(CC3)]  $a_{\langle-1\rangle}\ot s(a_{\langle0\rangle})+a_{1p}\ot\Delta_H(a_{2p})=p(a\loo)\ot a\lmi
+\delta_H(a_{1s})\ot a_{2s}\\
+\tau_{12}(a\loi\ot p(a\loo))+\tau_{12}(a_{1s}\ot\delta_H(a_{2s}))$.
\end{enumerate}

A pair of bilinear maps $q, t: H\to A\ot A$ are called cycles on $H$ if
\begin{enumerate}
\item[(CC4)]  $ x_{1q}\ot \delta_A(x_{2q})-x_{\langle1\rangle}\ot t(x_{\langle0\rangle})=q(x\qoo)\ot x\qi+\delta_A(x_{1t})\ot x_{2t}\\
    +\tau_{12}(x\qoi\ot q(x\qoo))+\tau_{12}(x_{1t}\ot\delta_A(x_{2t}))$.
\end{enumerate}

In the following definitions, we introduced the concept of cocycle
noncommutative Poisson  algebras and cycle noncommutative Poisson  coalgebras, which are  in fact not really
ordinary noncommutative Poisson algebras and noncommutative Poisson  coalgebras, but generalized ones.

\begin{definition}
(i): Let $\si, \omega$ be cocycles on a vector space  $H$ equipped with multiplications $[,], \cdot : H \ot H \to H$, satisfying the
following cocycle associative identity:
\begin{enumerate}
\item[(CC5)] $[x, yz]+x\trl\omega(y, z)=[x, y]z+\sigma(x, y)\rightarrow z+y[x, z]+y\leftarrow\sigma(x, z)$.
\end{enumerate}
Then  $H$ is called a  cocycle $(\si, \omega)$-noncommutative Poisson algebra which is denoted by $(H, \sigma, \omega)$.

(ii): Let $\theta, \nu$ be cocycle on a vector space $A$  equipped with multiplications $[,], \cdot : A \ot A \to A$, satisfying the
following cocycle associative identity:
\begin{enumerate}
\item[(CC6)] $[a, bc]-\nu(b, c)\trr a=[a, b]c+\theta(a, b)\ppr c+b[a, c]+b\ppl \theta(a, c)$.
\end{enumerate}
Then  $A$ is called a cocycle $(\theta, \nu)$-noncommutative Poisson   algebra which is denoted by $(A, \theta, \nu)$.

(iii) Let $p, s$ be cycles on a vector space  $H$ equipped with  comultiplications $\Delta, \delta: H \to H \ot H$, satisfying the
following cycle coassociative identity:
\begin{enumerate}
\item[(CC7)] $x\bi\ot\Delta_H(x\bii)+x_{\langle0\rangle}\ot s(x_{\langle1\rangle})=\delta_H(x\li)\ot x\lii
+p(x\qoi)\ot x\qoo\\
+\tau_{12}(x\li\ot\delta_H(x\lii))+\tau_{12}(x\qoo\ot p(x\qi))$.
\end{enumerate}
\noindent Then  $H$ is called a  cycle $(p, s)$-noncommutative Poisson coalgebra which is denoted by $(H, p, s)$.

(iv) Let $q, t$ be cycles on a vector space  $A$ equipped with comultiplications $\Delta, \delta: A \to A \ot A$, satisfying the
following cycle coassociative identity:
\begin{enumerate}
\item[(CC8)] $a\bi\ot \Delta_A(a\bii)-a_{\langle0\rangle}\ot t(a_{\langle-1\rangle})=\delta_A(a\li)\ot a\lii+q(a\loi)\ot a\loo\\
    +\tau_{12}\ot(a\li\ot\delta_A(a\lii))+\tau_{12}(a\loo\ot q(a\lmi))$.
\end{enumerate}
\noindent Then  $A$ is called a  cycle $(q, t)$-noncommutative Poisson coalgebra which is denoted by $(A, q, t)$.
\end{definition}

\begin{definition}
A  \emph{cocycle cross product system } is a pair of  $(\theta, \nu)$-noncommutative Poisson algebra $A$ and $(\sigma, \omega)$-noncommutative Poisson algebra $H$,
where $\si, \omega: H\ot H\to A$ are cocycles on $H$, $\theta, \nu: A\ot A\to H$ are cocycles on $A$ and the following conditions are satisfied:
\begin{enumerate}
\item[(CP1)] $[a, x\ppr b]-(x\leftarrow b)\trr a=x\ppr[a, b]+\omega(x, \theta(a,b))-(x\trr a)b-(x\trl a)\ppr b$,
\item[(CP2)] $[a, b\ppl x]-(b\rightarrow x)\trr a=[a, b]\ppl x+\omega(\theta(a,b), x)-b(x\trr a)-b\ppl (x\trl a)$,
\item[(CP3)] $(xy)\trr a-[a, \omega(x, y)]=(x\trr a)\ppl y+\omega(x\trl a, y)+x\ppr(y\trr a)+\omega(x, y\trl a)$,
\item[(CP4)] $x\trr(ab)+\sigma(x, \nu(a, b))=(x\trr a)b+(x\trl a)\ppr b+a(x\trr b)+a\ppl (x\trl b)$,
\item[(CP5)] $x\trr(y\ppr a)+\sigma(x, y\leftarrow a)=\sigma(x, y)a+[x, y]\ppr a+y\ppr(x\trr a)+\omega(y, x\trl a) $,
\item[(CP6)] $x\trr(a\ppl y)+\sigma(x, a\rightarrow y)=a\sigma(x, y)+a\ppl [x, y]+(x\trr a)\ppl y+\omega(x\trl a, y) $,
\item[(CP7)] $[x, y\leftarrow a]+x\trl(y\ppr a)=[x, y]\leftarrow a+\nu(\sigma(x, y), a)+y(x\trl a)+y\leftarrow(x\trr a)$,
\item[(CP8)] $[x, a\rightarrow y]+x\trl(a\ppl y)=a\rightarrow [x, y]+\nu(a, \sigma(x, y))+(x\trl a)y+(x\trr a)\rightarrow y$,
\item[(CP9)] $[x, \nu(a, b)]+x\trl(ab)=(x\trl a)\leftarrow b+\nu(x\trr a, b)+a\rightarrow (x\trl b)+\nu(a, x\trr b)$,
\item[(CP10)] $(xy)\trl a-\theta(a, \omega(x, y))=(x\trl a)y+(x\trr a)\rightarrow y+x(y\trl a)+x\leftarrow (y\trr a)$,
\item[(CP11)] $\theta(a, x\ppr b)-(x\leftarrow b)\trl a=x\theta(a, b)+x\leftarrow [a, b]-(x\trl a)\leftarrow b-\nu( x\trr a, b)$,
\item[(CP12)] $\theta(a, b\ppl x)-(b\rightarrow x)\trl a=\theta(a, b)x+[a, b]\rightarrow x-b\rightarrow (x\trl a)-\nu(b, x\trr a)$.
\end{enumerate}
\end{definition}

\begin{lemma}
Let $(A, H)$ be  a  cocycle cross product system.
If we define $E=A_{\sigma, \omega}\#_{\theta, \nu} H$ as the vector space $A\oplus H$ with the   multiplication
\begin{align}
[(a, x), (b, y)]_E=\big([a, b]+x\trr b-y\trr a+\sigma(x, y), \, [x, y]+x\trl b-y\trl a+\theta(a, b)\big),
\end{align}
and
\begin{align}
(a, x)\cdot_E (b, y)=\big(ab+x\ppr b+a\ppl y+\omega(x, y), \, xy+x\leftarrow b+a\rightarrow y+\nu(a, b)\big).
\end{align}
Then $E=A_{\sigma, \omega}\#_{\theta, \nu} H$ forms a noncommutative Poisson algebra  which is called the cocycle cross product noncommutative Poisson algebra.
\end{lemma}

\begin{proof} First, it is obvious that $(E, ~[ , ])$  and  $(E, ~\cdot)$  are respectively a Lie algebra and an associative algebra. Then, we need to prove the multiplications  $\cdot$ and $[,]$    satisfying
$[(a, x), (b, y)\cdot_{E} (c, z)]_{E}=[(a, x), (b, y)]_{E}\cdot_{E} (c, z)+(b, y)\cdot_{E}[(a, x), (c, z)]_{E}$. By direct computations, the left hand side is equal to
\begin{eqnarray*}
&&[(a, x), (b, y)\cdot_{E} (c, z)]_{E}\\
&=&[(a, x),(b c+y\ppr c+b\ppl z+\omega(y, z), y z+y\leftarrow c+b\rightarrow z+\nu(b, c))]_{E}\\
&=&\big([a, bc]+[a, y\ppr c]+[a, b\ppl z]+[a, \omega(y, z)]+x\trr (bc)+x\trr(y\ppr c)\\
&&+x\trr(b\ppl z)+x\trr\omega(y, z)-(yz)\trr a-(y\leftarrow c)\trr a-(b\rightarrow z)\trr a\\
&&-\nu(b, c)\trr a+\sigma(x, yz)+\sigma(x, y\leftarrow c)+\sigma(x, b\rightarrow z)+\sigma(x, \nu(b, c)),\\
&&[x, yz]+[x, y\leftarrow c]+[x,b\rightarrow z]+[x,\nu(b, c)]+x\trl(bc)+x\trl(y\ppr c)\\
&&+x\trl(b\ppl z)+x\trl\omega(y, z)-(yz)\trl a-(y\leftarrow c)\trl a-(b\rightarrow z)\trl a\\
&&-\nu(b, c)\trl a+\theta(a, bc)+\theta(a, y\ppr c)+\theta(a, b\ppl z)+\theta(a, \omega(y, z))\big),
\end{eqnarray*}
and the right hand side is equal to
\begin{eqnarray*}
&&[(a, x), (b, y)]_{E}\cdot_{E} (c, z)+(b, y)\cdot_{E}[(a, x), (c, z)]_{E}\\
&=&([a, b]+x\trr b-y\trr a+\sigma(x, y), [x, y]+x\trl b-y\trl a+\theta(a, b))\cdot_{E}(c, z)\\
&&+(b, y)\cdot_{E}([a,c]+x\trr c-z\trr a+\sigma(x, z), [x, z]+x\trl c-z\trl a+\theta(a, c))\\
&=&\big([a, b]c+(x\trr b)c-(y\trr a)c+\sigma(x, y)c+[x, y]\ppr c+(x\trl b)\ppr c-(y\trl a)\ppr c\\
&&+\theta(a, b)\ppr c+[a, b]\ppl z+(x\trr b)\ppl z-(y\trr a)\ppl z+\sigma(x, y)\ppl z\\
&&+\omega([x, y], z)+\omega(x\trl b, z)-\omega(y\trl a, z)+\omega(\theta(a, b), z), [x, y]z+(x\trl b)z\\
&&-(y\trl a)z+\theta(a, b)z+[x, y]\leftarrow c+(x\trl b)\leftarrow c-(y\trl a)\leftarrow c+\theta(a, b)\leftarrow c\\
&&+[a, b]\rightarrow z+(x\trr b)\rightarrow z-(y\trr a)\rightarrow z+\sigma(x, y)\rightarrow z+\nu([a, b], c)+\nu(x\trr b, c)\\
&&-\nu(y\trr a, c)+\nu(\sigma(x, y), c)\big)+\big(b[a, c]+b(x\trr c)-b(z\trr a)+b\sigma(x, z)\\
&&+y\ppr[a, c]+y\ppr(x\trr c)-y\ppr(z\trr a)+y\ppr\sigma(x, z)+b\ppl [x, z] +b\ppl (x\trl c)\\
&&-b\ppl (z\trl a)+b\ppl \theta(a, c)+\omega(y, [x, z])+\omega(y, x\trl c)-\omega(y, z\trl a)+\omega(y, \theta(a, c)),\\
&&y[x, z]+y(x\trl c)-y(z\trl a)+y\theta(a, c)+y\leftarrow[a, c]+y\leftarrow(x\trr c)-y\leftarrow(z\trr a)\\
&&+y\leftarrow\sigma(x, z)+b\rightarrow[x, z]+b\rightarrow (x\trl c)-b\rightarrow (z\trl a)+b\rightarrow \theta(a, c)+\nu(b, [a, c])\\
&&+\nu(b, x\trr c)-\nu(b, z\trr a)+\nu(b, \sigma(x, z))\big)\\
&=&\big([a, b]c+(x\trr b)c-(y\trr a)c+\sigma(x, y)c+[x, y]\ppr c+(x\trl b)\ppr c-(y\trl a)\ppr c\\
&&+\theta(a, b)\ppr c+[a, b]\ppl z+(x\trr b)\ppl z-(y\trr a)\ppl z+\sigma(x, y)\ppl z+\omega([x, y], z)\\
&&+\omega(x\trl b, z)-\omega(y\trl a, z)+\omega(\theta(a, b), z)+b[a, c]+b(x\trr c)-b(z\trr a)+b\sigma(x, z)\\
&&+y\ppr[a, c]+y\ppr(x\trr c)-y\ppr(z\trr a)+y\ppr\sigma(x, z)+b\ppl [x, z] +b\ppl (x\trl c)\\
&&-b\ppl (z\trl a)+b\ppl \theta(a, c)+\omega(y, [x, z])+\omega(y, x\trl c)-\omega(y, z\trl a)+\omega(y, \theta(a, c)),\\
&&[x, y]z+(x\trl b)z-(y\trl a)z+\theta(a, b)z+[x, y]\leftarrow c+(x\trl b)\leftarrow c-(y\trl a)\leftarrow c\\
&&+\theta(a, b)\leftarrow c+[a, b]\rightarrow z+(x\trr b)\rightarrow z-(y\trr a)\rightarrow z+\sigma(x, y)\rightarrow z+\nu([a, b], c)\\
&&+\nu(x\trr b, c)-\nu(y\trr a, c)+\nu(\sigma(x, y), c)+y[x, z]+y(x\trl c)-y(z\trl a)+y\theta(a, c)\\
&&+y\leftarrow[a, c]+y\leftarrow(x\trr c)-y\leftarrow(z\trr a)+y\leftarrow\sigma(x, z)+b\rightarrow[x, z]+b\rightarrow (x\trl c)\\
&&-b\rightarrow (z\trl a)+b\rightarrow \theta(a, c)+\nu(b, [a, c])+\nu(b, x\trr c)-\nu(b, z\trr a)+\nu(b, \sigma(x, z))\big).
\end{eqnarray*}
 Thus the two sides are equal to each other if and only if (CP1)--(CP12) hold.
\end{proof}

\begin{definition}
A  \emph{cycle cross coproduct system } is a pair of   $(p, s)$-noncommutative Poisson coalgebra $A$ and  $(q, t)$-noncommutative Poisson coalgebra $H$, where $p, s: A\to H\ot H$ are cycles on $A$,  $q, t: H\to A\ot A$ are cycles over $H$ such that following conditions are satisfied:
\begin{enumerate}
\item[(CCP1)] $a\bi\ot\rho(a\bii)-a_{\langle0\rangle}\ot\beta(a_{\langle-1\rangle})=-\tau\phi(a\li)\ot a\lii-\tau\psi(a\loi)\ot a\loo\\
    +\tau_{12}(a\loi\ot\delta_A(a\loo))+\tau_{12}(a_{1s}\ot q(a_{2s}))$,
\item[(CCP2)] $a\bi\ot\gamma(a\bii)-a_{\langle0\rangle}\ot\al(a_{\langle-1\rangle})
=\delta_A(a\loo)\ot a\lmi+q(a_{1s})\ot a_{2s}\\
-\tau_{12}(a\li\ot\tau\phi(a\lii))-\tau_{12}(a\loo\ot\tau\psi(a\lmi))$,
\item[(CCP3)] $a_{\langle0\rangle}\ot\Delta_H(a_{\langle-1\rangle})-a\bi\ot s(a\bii)
=\tau\phi(a\loo)\ot a\lmi+\tau\psi(a_{1s})\ot a_{2s}\\
+\tau_{12}(a\loi\ot\tau\phi(a\loo))+\tau_{12}(a_{1s}\ot\tau\psi(a_{2s}))$,
\item[(CCP4)] $a_{\langle-1\rangle}\ot\Delta_A(a_{\langle0\rangle})+a_{1p}\ot t(a_{2p})
=\phi(a\li)\ot a\lii+\psi(a\loi)\ot a\loo\\
+\tau_{12}(a\li\ot\phi(a\lii))+\tau_{12}(a\loo\ot\psi(a\lmi))$,
\item[(CCP5)] $a_{\langle-1\rangle}\ot\rho(a_{\langle0\rangle})+a_{1p}\ot \beta(a_{2p})
=\delta_H(a\loi)\ot a\loo+p(a\li)\ot a\lii\\
+\tau_{12}(a\loi\ot\phi(a\loo))+\tau_{12}(a_{1s}\ot \psi(a_{2s}))$,
\item[(CCP6)] $a_{\langle-1\rangle}\ot\gamma(a_{\langle0\rangle})+a_{1p}\ot\al (a_{2p})
=\phi(a\loo)\ot a\lmi+\psi(a_{1s})\ot a_{2s}\\
+\tau_{12}(a\li\ot p(a\lii))+\tau_{12}(a\loo\ot \delta_H(a\lmi))$,
\item[(CCP7)] $x\bi\ot \beta(x\bii)+x_{\langle0\rangle}\ot\rho(x_{\langle1\rangle})
=\delta_H(x\qoo)\ot x\qi+p(x_{1t})\ot x_{2t}\\
+\tau_{12}(x\li\ot \psi(x\lii))+\tau_{12}(x\qoo\ot \phi(x\qi))$,
\item[(CCP8)] $x\bi\ot\al(x\bii)+x_{\langle0\rangle}\ot\gamma(x_{\langle1\rangle})
=\psi(x\li)\ot x\lii+\phi(x\qoi)\ot x\qoo\\
+\tau_{12}(x\qoi\ot\delta_H(x\qoo))+\tau_{12}(x_{1t}\ot p(x_{2t}))$,
\item[(CCP9)] $x\bi\ot t(x\bii)+x_{\langle0\rangle}\ot\Delta_A(x_{\langle1\rangle})
=\psi(x\qoo)\ot x\qi+\phi(x_{1t})\ot x_{2t}\\
+\tau_{12}(x\qoi\ot\psi(x\qoo))+\tau_{12}(x_{1t}\ot\phi(x_{2t}))$,
\item[(CCP10)] $x_{\langle1\rangle}\ot\Delta_H(x_{\langle0\rangle})-x_{1q}\ot s(x_{2q})
=\tau\psi(x\li)\ot x\lii+\tau\phi(x\qoi)\ot x\qoo\\
+\tau_{12}(x\li\ot\tau\psi(x\lii))+\tau_{12}(x\qoo\ot\tau\phi(x\qi))$,
\item[(CCP11)] $x_{\langle1\rangle}\ot\beta(x_{\langle0\rangle})-x_{1q}\ot \rho(x_{2q})
=\tau\psi(x\qoo)\ot x\qi+\tau\phi(x_{1t})\ot x_{2t}\\
-\tau_{12}(x\qoo\ot\delta_A(x\qi))-\tau_{12}(x\li\ot q(x\lii))$,
\item[(CCP12)] $x_{1q}\ot \gamma(x_{2q})-x_{\langle1\rangle}\ot \al(x_{\langle0\rangle})
=q(x\li)\ot x\lii+\delta_A(x\qoi)\ot x\qoo\\
-\tau_{12}(x\qoi\ot \tau\psi(x\qoo))-\tau_{12}(x_{1t}\ot\tau\phi(x_{2t}))$.
\end{enumerate}
\end{definition}

\begin{lemma}\label{lem2} Let $(A, H)$ be  a  cycle cross coproduct system. If we define $E=A^{p, s}\# {}^{q, t} H$ to be the vector
space $A\oplus H$ with the comultiplication
$$\delta_{E}(a)=(\delta_{A}+\phi-\tau\phi+p)(a), \quad \delta_{E}(x)=(\delta_{H}+\psi-\tau\psi+q)(x), $$
$$\Delta_{E}(a)=(\Delta_{A}+\rho+\gamma+s)(a), \quad \Delta_{E}(x)=(\Delta_{H}+\alpha+\beta+t)(x), $$
that is
$$\delta_{E}(a)= \sum a\bi \ot a\bii+\sum a_{\langle-1\rangle} \ot a_{\langle0\rangle}-\sum a_{\langle0\rangle}\ot a_{\langle-1\rangle}+\sum a_{1p}\ot a_{2p},$$
$$\delta_{E}(x)= \sum x\bi \ot x\bii+\sum x_{\langle0\rangle} \ot x_{\langle1\rangle}-\sum x_{\langle1\rangle} \ot x_{\langle0\rangle}+\sum x_{1q}\ot x_{2q},$$
$$\Delta_{E}(a)= \sum a\li \ot a\lii+\sum a\moi \ot a\mo+\sum a\mo\ot a\lmi+\sum a_{1s}\ot a_{2s},$$
$$\Delta_{E}(x)= \sum x\li \ot x\lii+ \sum x\qoi \ot x\qoo+\sum x\qoo \ot x\qi+\sum x_{1t}\ot x_{2t},$$
then  $A^{p, s}\# {}^{q, t} H$ forms a noncommutative Poisson coalgebra which we will call it the cycle cross coproduct noncommutative Poisson  coalgebra.
\end{lemma}

\begin{proof} Due to the fact that $(E, ~\delta)$  and  $(E, ~\Delta)$  are respectively a Lie coalgebra and a coassociative coalgebra, we only  need to prove
$(\id\ot\Delta_{E})\delta_{E}(a, x)=(\delta_{E}\ot\id)\Delta_{E}(a,x)+(\tau\ot\id)(\id\ot\delta_{E})\Delta_{E}(a, x)$.

 The left hand side is equal to
\begin{eqnarray*}
&&(\id\ot\Delta_{E})\delta_{E}(a, x)\\
&=&(\id\ot\Delta_{E})(a\bi\ot a\bii+a_{\langle-1\rangle}\ot a_{\langle0\rangle}-a_{\langle0\rangle}\ot a_{\langle-1\rangle}+a_{1p}\ot a_{2p}
+x\bi\ot x\bii\\
&&+x_{\langle0\rangle}\ot x_{\langle1\rangle}-x_{\langle1\rangle}\ot x_{\langle0\rangle}+x_{1q}\ot x_{2q})\\
&=&a_{[1]} \otimes \Delta_{A}\left(a_{[2]}\right)+a_{[1]} \otimes \rho\left(a_{[2]}\right)
+a_{[1]} \otimes \gamma\left(a_{[2]}\right)+a_{[1]} \otimes s\left(a_{[2]}\right)\\
&&+a_{\langle-1\rangle} \otimes \Delta_{A}\left(a_{\langle0\rangle}\right)
+a_{\langle-1\rangle} \otimes \rho\left(a_{\langle0\rangle}\right)
+a_{\langle-1\rangle} \otimes \gamma\left(a_{\langle0\rangle}\right)
+a_{\langle-1\rangle} \otimes s\left(a_{\langle0\rangle}\right)\\
&&-a_{\langle0\rangle} \otimes \Delta_{H}\left(a_{\langle-1\rangle}\right)
-a_{\langle0\rangle} \otimes \al\left(a_{\langle-1\rangle}\right)
-a_{\langle0\rangle} \otimes \beta\left(a_{\langle-1\rangle}\right)
-a_{\langle0\rangle} \otimes t\left(a_{\langle-1\rangle}\right)\\
&&+a_{1p}\ot \Delta_H(a_{2p})+a_{1p}\ot \al(a_{2p})+a_{1p}\ot \beta(a_{2p})+a_{1p}\ot t(a_{2p})\\
&&+x_{[1]} \otimes \Delta_{H}\left(x_{[2]}\right)+x\bi\ot \al(x\bii)
+x\bi\ot \beta(x\bii)+x\bi\ot t(x\bii)\\
&&+x_{\langle0\rangle} \otimes \Delta_A\left(x_{\langle1\rangle}\right)+x_{\langle0\rangle} \otimes \rho\left(x_{\langle1\rangle}\right)
+x_{\langle0\rangle} \otimes \gamma\left(x_{\langle1\rangle}\right)
+x_{\langle0\rangle} \otimes s\left(x_{\langle1\rangle}\right)\\
&&-x_{\langle1\rangle} \otimes \Delta_H\left(x_{\langle0\rangle}\right)-x_{\langle1\rangle} \otimes \al\left(x_{\langle0\rangle}\right)-x_{\langle1\rangle} \otimes \beta\left(x_{\langle0\rangle}\right)
-x_{\langle1\rangle} \otimes t\left(x_{\langle0\rangle}\right)\\
&&+x_{1q}\ot \Delta_A(x_{2q})+x_{1q}\ot \rho(x_{2q})+x_{1q}\ot \gamma(x_{2q})+x_{1q}\ot s(x_{2q}),
\end{eqnarray*}
and the right hand side is equal to
\begin{eqnarray*}
&&(\delta_{E}\ot\id)\Delta_{E}(a, x)+(\tau\ot\id)(\id\ot\delta_{E})\Delta_{E}(a, x)\\
&=&(\delta_{E}\ot\id)(a\li\ot a\lii+a\loi\ot a\loo+a\loo\ot a\lmi+a_{1s}\ot a_{2s}+x\li\ot x\lii\\
&&+x\qoi\ot x\qoo+x\qoo\ot x\qi+x_{1t}\ot x_{2t})+(\tau\ot\id)(\id\ot\delta_{E})(a\li\ot a\lii\\
&&+a\loi\ot a\loo+a\loo\ot a\lmi+a_{1s}\ot a_{2s}+x\li\ot x\lii+x\qoi\ot x\qoo\\
&&+x\qoo\ot x\qi+x_{1t}\ot x_{2t})\\
&=&\delta_{A}\left(a_{1}\right) \otimes a_{2}+\phi\left(a_{1}\right) \otimes a_{2}-\tau\phi\left(a_{1}\right) \otimes a_{2}+p(a\li)\ot a\lii
+\delta_{H}\left(a_{(-1)}\right) \otimes a_{(0)}\\
&&+\psi(a\loi)\ot a\loo-\tau\psi(a\loi)\ot a\loo+q(a\loi)\ot a\loo+\delta_{A}\left(a_{(0)}\right) \otimes a_{(1)}\\
&&+\phi\left(a_{(0)}\right) \otimes a_{(1)}
-\tau\phi\left(a_{(0)}\right) \otimes a_{(1)}+p(a\loo)\ot a\lmi+\delta_H(a_{1s})\ot a_{2s}\\
&&+\psi(a_{1s})\ot a_{2s}-\tau\psi(a_{1s})\ot a_{2s}+q(a_{1s})\ot a_{2s}+\delta_{H}\left(x_{1}\right) \otimes x_{2}\\
&&+\psi(x\li)\ot x\lii-\tau\psi(x\li)\ot x\lii+q(x\li)\ot x\lii+\delta_A(x\qoi)\ot x\qoo+\phi(x\qoi)\ot x\qoo\\
&&-\tau\phi(x\qoi)\ot x\qoo+p(x\qoi)\ot x\qoo+\delta_H(x\qoo)\ot x\qi
+\psi(x\qoo)\ot x\qi\\
&&-\tau\psi(x\qoo)\ot x\qi+q(x\qoo)\ot x\qi+\delta_A(x_{1t})\ot x_{2t}+\phi(x_{1t})\ot x_{2t}\\
&&-\tau\phi(x_{1t})\ot x_{2t}+p(x_{1t})\ot x_{2t}+\tau_{12}(a\li\ot\delta_{A}(a\lii))+\tau_{12}(a\li\ot\phi(a\lii))\\
&&-\tau_{12}(a\li\ot\tau\phi(a\lii))+\tau_{12}(a\li\ot p(a\lii))+\tau_{12}(a\loi\ot\delta_{A}(a\loo))\\
&&+\tau_{12}(a\loi\ot\phi(a\loo))-\tau_{12}(a\loi\ot\tau\phi(a\loo))+\tau_{12}(a\loi\ot p(a\loo))\\
&&+\tau_{12}(a\loo\ot\delta_{H}(a\lmi))+\tau_{12}(a\loo\ot\psi(a\lmi))-\tau_{12}(a\loo\ot\tau\psi(a\lmi))\\
&&+\tau_{12}(a\loo\ot q(a\lmi))+\tau_{12}(a_{1s}\ot \delta_H(a_{2s}))+\tau_{12}(a_{1s}\ot \psi(a_{2s}))\\
&&-\tau_{12}(a_{1s}\ot \tau\psi(a_{2s}))+\tau_{12}(a_{1s}\ot q(a_{2s}))+\tau_{12}(x\li\ot\delta_{H}(x\lii))\\
&&+\tau_{12}(x\li\ot\psi(x\lii))-\tau_{12}(x\li\ot\tau\psi(x\lii))+\tau_{12}(x\li\ot q(x\lii))\\
&&+\tau_{12}(x\qoi\ot \delta_H(x\qoo))+\tau_{12}(x\qoi\ot \psi(x\qoo))-\tau_{12}(x\qoi\ot \tau\psi(x\qoo))\\
&&+\tau_{12}(x\qoi\ot q(x\qoo))+\tau_{12}(x\qoo\ot \delta_A(x\qi))+\tau_{12}(x\qoo\ot \phi(x\qi))\\
&&-\tau_{12}(x\qoo\ot \tau\phi(x\qi))+\tau_{12}(x\qoo\ot p(x\qi))+\tau_{12}(x_{1t}\ot \delta_A(x_{2t}))\\
&&+\tau_{12}(x_{1t}\ot \phi(x_{2t}))-\tau_{12}(x_{1t}\ot \tau\phi(x_{2t}))+\tau_{12}(x_{1t}\ot p(x_{2t})).
\end{eqnarray*}
 Thus the two sides are equal to each other if and only if (CCP1)--(CCP12) hold.
\end{proof}

\begin{definition}\label{cocycledmp}
Let $A, H$ be both noncommutative Poisson algebras and noncommutative Poisson  coalgebras. If  the following conditions hold:
\begin{enumerate}
\item[(CDM1)] $\phi(a b)+\psi(\nu(a, b))= (a_{\langle-1\rangle}\leftarrow b)\ot  a_{\langle0\rangle}+( a\rightarrow b_{\langle-1\rangle})\ot  b_{\langle0\rangle}+b\loi\ot[a, b\loo]\\
    +a\lmi\ot[b, a\loo]+\nu(a\bi, b)\ot a\bii+\nu(a, b\bi)\ot b\bii-b_{1s}\ot(b_{2s}\trr a)-a_{2s}\ot(a_{1s}\trr b)$,
\item[(CDM2)] $\tau\phi(ab)+\tau\psi(\nu(a, b))= a_{\langle0\rangle}b\ot  a_{\langle-1\rangle}+a b_{\langle0\rangle}\ot  b_{\langle-1\rangle}
+b\loo\ot(b\lmi\trl a)+a\loo\ot(a\loi\trl b)\\
-(a_{1p}\ppr b)\ot a_{2p}-(a\ppl b_{1p})\ot b_{2p}-b\li\ot\theta(a,b\lii)-a\lii\ot\theta(b,a\li)$,
\item[(CDM3)] $\psi(xy)+\phi(\omega(x, y))= x_{\langle0\rangle} y\ot x_{\langle1\rangle}+x  y_{\langle0\rangle}\ot  y_{\langle1\rangle}
+ y\qoo\ot(x\trr y\qi)+x\qoo\ot(y\trr x\qoi)\\
+(x_{1q}\rightarrow y)\ot x_{2q}+(x\leftarrow y_{1q})\ot y_{2q}+y\li\ot\sigma(x, y\lii)+x\lii\ot\sigma(y, x\li)$,
\item[(CDM4)] $\tau\psi(xy)+\tau\phi(\omega(x,y))=(x_{\langle1\rangle}\ppl y)\ot x_{\langle0\rangle}+(x\ppr y_{\langle1\rangle})\ot y_{\langle0\rangle}
-y\qoi\ot[x, y\qoo]\\
-x\qi\ot[y, x\qoo]-\omega(x\bi, y)\ot x\bii-\omega(x, y\bi)\ot y\bii-y_{1t}\ot(x\trl y_{2t})-x_{2t}\ot(y\trl x_{1t})$,
\item[(CDM5)]  $\delta_A(x\ppr b)+q(x\leftarrow b)=( x_{\langle0\rangle}\ppr b)\ot  x_{\langle1\rangle}+(x\ppr b\bi)\ot b\bii-x\qi\ot(x\qoo\trr b)\\
+b\li\ot(x\trr b\lii)+x_{1q}b\ot x_{2q}+\omega(x, b_{\langle-1\rangle})\ot b_{\langle0\rangle}+b\loo\ot\sigma(x, b\lmi)+x_{2t}\ot[b, x_{1t}]$,
\item[(CDM6)]  $\delta_A(a\ppl y)+q(a\rightarrow y)=(a\ppl y_{\langle0\rangle})\ot y_{\langle1\rangle}
+(a\bi\ppl y)\ot a\bii-y\qoi\ot(y\qoo\trr a)+a\lii\ot (y\trr a\li)\\
+ay_{1q}\ot y_{2q}+\omega(a_{\langle-1\rangle}, y)\ot a_{\langle0\rangle}+y_{1t}\ot [a, y_{2t}]+a\loo\ot \sigma(y, a\loi)$,
\item[(CDM7)] $\delta_H(x\leftarrow b)+p(x\ppr b)=(x\bi\leftarrow b)\ot x\bii-(x\leftarrow b_{\langle0\rangle})\ot b_{\langle-1\rangle}+b\loi\ot(x\trl b\loo)\\
-x\lii\ot(x\li\trl b)-\nu(x_{\langle1\rangle}, b)\ot x_{\langle0\rangle}+xb_{1p}\ot b_{2p}+b_{1s}\ot[x, b_{2s}]
+x\qoo\ot\theta(b, x\qoi)$,
\item[(CDM8)] $\delta_H(a\rightarrow y)+p(a\ppl y)=(a\rightarrow y\bi)\ot y\bii-(a_{\langle0\rangle}\rightarrow y)\ot a_{\langle-1\rangle}+a\lmi\ot(y\trl a\loo)\\
-y\li\ot(y\lii\trl a)-\nu(a, y_{\langle1\rangle})\ot y_{\langle0\rangle}+a_{1p}y\ot a_{2p}+a_{2s}\ot[y, a_{1s}]
+y\qoo\ot\theta(a, y\qi)$,
\item[(CDM9)] $\phi(x\ppr b)+\psi(x\leftarrow b)=(x_{\langle0\rangle}\leftarrow b)\ot x_{\langle1\rangle}+(x\leftarrow b\bi)\ot b\bii
+x b_{\langle-1\rangle}\ot b_{\langle0\rangle}\\
+b\loi\ot(x\trr b\loo)-x\lii\ot(x\li\trr b)+x\qoo\ot[b, x\qoi]+\nu(x_{1q}, b)\ot x_{2q}+b_{1s}\ot\sigma(x, b_{2s})$,
\item[(CDM10)] $\tau\phi(x\ppr b)+\tau\psi(x\leftarrow b)=x_{\langle1\rangle} b\ot x_{\langle0\rangle}+(x\ppr b_{\langle0\rangle})\ot b_{\langle-1\rangle}+x\qi\ot(x\qoo\trl b)\\
-(x\bi\ppr b)\ot x\bii-b\loo\ot[x, b\lmi]-b\li\ot(x\trl b\lii)-\omega(x, b_{1p})\ot b_{2p}-x_{2t}\ot\theta(b, x_{1t})$,
\item[(CDM11)] $\phi(a\ppl y)+\psi(a\rightarrow y)=a_{\langle-1\rangle}y\ot a_{\langle0\rangle}+( a\rightarrow y_{\langle0\rangle})\ot y_{\langle1\rangle}+y\qoo\ot[a, y\qi]\\
-(a\bi\rightarrow y)\ot a\bii+a\lmi\ot(y\trr a\loo)-y\li\ot(y\lii\trr a)+\nu(a, y_{1q})\ot y_{2q}+a_{2s}\ot\sigma(y, a_{1s})$,
\item[(CDM12)] $\tau\phi(a\ppl y)+\tau\psi(a\rightarrow y)=ay_{\langle1\rangle}\ot y_{\langle0\rangle}+( a_{\langle0\rangle}\ppl y)\ot a_{\langle-1\rangle}+y\qoi\ot(y\qoo\trl a)\\
-(a\ppl y\bi)\ot y\bii-a\loo\ot[y, a\loi]-a\lii\ot(y\trl a\li)-\omega(a_{1p}, y)\ot a_{2p}-y_{1t}\ot\theta(a, y_{2t})$,
\item[(CDM13)] $\rho([a, b])+\beta(\theta(a, b))= (a_{\langle-1\rangle}\leftarrow b)\ot  a_{\langle0\rangle}-( b_{(-1)}\trl a)\ot  b_{(0)}
+b\loi\ot[a, b\loo]\\
-a_{\langle-1\rangle}\ot ba_{\langle0\rangle}+\theta(a, b\li)\ot b\lii-b_{1s}\ot(b_{2s}\trr a)+\nu(a\bi, b)\ot a\bii-a_{1p}\ot(b\ppl a_{2p})$,
\item[(CDM14)] $\beta([x, y])+\rho(\sigma(x, y))= [x, y\qoo]\ot y\qi+y\qoo\ot (x\trr y\qi)
-x_{\langle0\rangle}\ot(y\ppr x_{\langle1\rangle})\\
+x_{\langle0\rangle}y\ot x_{\langle1\rangle}+(x\trl y_{1t})\ot y_{2t}+y\li\ot\sigma(x, y\lii)+(x_{1q}\rightarrow y)\ot x_{2q}-x\bi\ot\omega(y, x\bii)$,
\item[(CDM15)] $\gamma([a ,b])+\alpha(\theta(a,b))= a_{\langle0\rangle}\ot(b\rightarrow a_{\langle-1\rangle})- b_{(0)}\ot( b_{(1)}\trl a)
+[a,b\loo]\ot b\lmi\\
-a\bi\ot \nu(b, a\bii)+b\li\ot\theta(a,b\lii)-(b_{1s}\trr a)\ot b_{2s}+\nu(a\bi, b)\ot a\bii +(a_{\langle-1\rangle}\leftarrow b)\ot a_{\langle0\rangle}$,
\item[(CDM16)] $\al([x,y])+\gamma(\sigma(x,y))= y\qoi\ot[x,y\qoo]+ (x\trr y\qoi)\ot y\qoo
-(x_{\langle1\rangle}\ppl y)\ot x_{\langle0\rangle}\\
+x_{\langle1\rangle}\ot yx_{\langle0\rangle}+y_{1t}\ot(x\trl y_{2t})+\sigma(x,y\li)\ot y\lii-x_{1q}\ot(y\leftarrow x_{2q}) +\omega(x\bi, y)\ot x\bii$,
\item[(CDM17)]  $\Delta_A(x\trr b)+t(x\trl b)=(x\trr b\li)\ot b\lii+b\li\ot(x\trr b\lii)+( x_{\langle0\rangle}\ppr b)\ot  x_{\langle1\rangle}\\
    +x_{\langle1\rangle}\ot (b\ppl x_{\langle0\rangle})+\sigma(x, b\loi)\ot b\loo+b\loo\ot\sigma(x, b\lmi)+x_{1q}b\ot x_{2q}-x_{1q}\ot bx_{2q}$,
\item[(CDM18)]  $\Delta_A(y\trr a)+t(y\trl a)=-(a\bi\ppl y)\ot a\bii+a\bi\ot(y\ppr a\bii)+( y\qoo\trr a)\ot  y\qi\\
    +y\qoi\ot (y\qoo\trr a)-[a, y_{1t}]\ot y_{2t}-y_{1t}\ot[a, y_{2t}]
    -a_{\langle0\rangle}\ot\omega(y, a_{\langle-1\rangle})-\omega(a_{\langle-1\rangle}, y)\ot a_{\langle0\rangle}$,
\item[(CDM19)] $\Delta_H(x\trl b)+s(x\trr b)=(x\trl b\loo)\ot b\lmi+b\loi\ot(x\trl b\loo)+(x\bi\leftarrow b)\ot x\bii\\
    -x\bi\ot(b\rightarrow x\bii)+[x, b_{1s}]\ot b_{2s}+b_{1s}\ot[x, b_{2s}]-\nu(x_{\langle1\rangle}, b)\ot x_{\langle0\rangle}
    -x_{\langle0\rangle}\ot \nu(b, x_{\langle1\rangle})$,
\item[(CDM20)] $\Delta_H(y\trl a)+s(y\trr a)=(y\li\trl a)\ot y\lii+y\li\ot(y\lii\trl a)+(a_{\langle0\rangle}\rightarrow y)\ot a_{\langle-1\rangle}\\
    +a_{\langle-1\rangle}\ot(y\leftarrow a_{\langle0\rangle})
    -\theta(a, y\qoi)\ot y\qoo-y\qoo\ot\theta(a, y\qi)-a_{1p}y\ot a_{2p}+a_{1p}\ot ya_{2p}$,
\item[(CDM21)] $\rho(x\trr b)+\beta(x\trl b)=(x\trl b\li)\ot b\lii+[x, b\loi]\ot b\loo+b\loi\ot(x\trr b\loo)
-x_{\langle0\rangle}\ot bx_{\langle1\rangle}\\
    +(x_{\langle0\rangle}\leftarrow b)\ot x_{\langle1\rangle}-x\bi\ot(b\ppl x\bii)
+b_{1s}\ot \sigma(x, b_{2s})+\nu(x_{1q}, b)\ot x_{2q}$,
\item[(CDM22)] $\rho(y\trr a)+\beta(y\trl a)=(y\qoo\trl a)\ot y\qi-y\qoo\ot[a, y\qi]-(a\bi\rightarrow y)\ot a\bii\\
    -a_{\langle-1\rangle}y\ot a_{\langle0\rangle}+y\li\ot(y\lii\trr a)
+a_{\langle-1\rangle}\ot (y\ppr a_{\langle0\rangle})-\theta(a, y_{1t})\ot y_{2t}+a_{1p}\ot \omega(y, a_{2p})$,
\item[(CDM23)] $\gamma(x\trr b)+\al(x\trl b)=b\li\ot(x\trl b\lii)+ b\loo\ot[x,b\lmi]+(x\trr b\loo)\ot b\lmi\\
    +x_{\langle1\rangle}\ot (b\rightarrow x_{\langle0\rangle})+(x\bi\ppr b)\ot x\bii
- x_{\langle1\rangle}b\ot x_{\langle0\rangle}+\sigma(x,b_{1s})\ot b_{2s}-x_{1q}\ot\nu(b, x_{2q})$,
\item[(CDM24)] $\gamma(y\trr a)+\al(y\trl a)= y\qoi\ot (y\qoo\trl a)-[a,y\qoi]\ot y\qoo+(y\li\trr a)\ot y\lii-a_{\langle0\rangle}\ot ya_{\langle-1\rangle}\\
    +a\bi\ot(y\leftarrow a\bii)
+(a_{\langle0\rangle}\ppl y)\ot a_{\langle-1\rangle}-y_{1t}\ot\theta(a, y_{2t})-\omega(a_{1p}, y)\ot a_{2p}$.
\end{enumerate}
\noindent then $(A, H)$ is called a \emph{cocycle double matched pair}.
\end{definition}

\begin{definition}\label{cocycle-braided}
(i) A \emph{cocycle braided noncommutative Poisson bialgebra} $A$ is simultaneously a cocycle noncommutative Poisson algebra $(A, \theta, \nu)$ and  a cycle noncommutative Poisson coalgebra $(A, q, t)$ satisfying the  conditions
\begin{enumerate}
\item[(CBB1)] $\delta_{A}(ab)+q(\nu(a,b))=a\bi b\ot a\bii+(a_{\langle-1\rangle}\ppr b)\ot a_{\langle0\rangle}
+ab\bi\ot b\bii+(a\ppl b_{\langle-1\rangle})\ot b_{\langle0\rangle}\\
+b\li\ot[a, b\lii]-b\loo\ot(b\lmi\trr a)+a\lii\ot[b, a\li]-a\loo\ot(a\loi\trr b),$
\item[(CBB2)] $\Delta_{A}([a, b])+t(\theta(a, b))=[a, b\li]\ot b\lii-(b\loi\trr a)\ot b\loo+b\li\ot[a, b\lii]
-b\loo\ot(b\lmi\trr a)\\
+a\bi b\ot a\bii+(a_{\langle-1\rangle}\ppr b)\ot a_{\langle0\rangle}-a\bi\ot ba\bii+a_{\langle0\rangle}\ot(b\ppl a_{\langle-1\rangle}).$
    \end{enumerate}
(ii) A \emph{cocycle braided noncommutative Poisson bialgebra} $H$ is simultaneously a cocycle noncommutative Poisson algebra $(H, \sigma, \omega)$ and a cycle noncommutative Poisson coalgebra $(H, p, s)$ satisfying the conditions
\begin{enumerate}
\item[(CBB3)] $\delta_{H}(xy)+p(\omega(x, y))=x\bi y\ot x\bii-(x_{\langle1\rangle}\rightarrow y)\ot x_{\langle0\rangle}
+xy\bi\ot y\bii-(x\leftarrow y_{\langle1\rangle})\ot y_{\langle0\rangle}\\
+y\li\ot[x, y\lii]+y\qoo\ot(x\trl y\qi)+x\lii\ot[y, x\li]+x\qoo\ot(y\trl x\qoi),$
\item[(CBB4)] $\Delta_{H}([x, y])+s(\sigma(x, y))=[x, y\li]\ot y\lii+(x\trl y\qoi)\ot y\qoo+y\li\ot[x, y\lii]
+y\qoo\ot(x\trl y\qi)\\
+x\bi y\ot x\bii-(x_{\langle1\rangle}\rightarrow y)\ot x_{\langle0\rangle}-x\bi\ot yx\bii-x_{\langle0\rangle}\ot(y\leftarrow x_{\langle1\rangle}).$
    \end{enumerate}
\end{definition}

The next theorem says that we can obtain an ordinary noncommutative Poisson bialgebra from two cocycle braided noncommutative Poisson bialgebras.
\begin{theorem}\label{main2}
Let $A$, $H$ be  cocycle braided noncommutative Poisson bialgebras, $(A, H)$ be a cocycle cross product system and a cycle cross coproduct system.
Then the cocycle cross product noncommutative Poisson  algebra and cycle cross coproduct noncommutative Poisson  coalgebra fit together to become an ordinary
noncommutative Poisson bialgebra if and only if $(A, H)$ forms a cocycle double matched pair. We will call it the cocycle bicrossproduct noncommutative Poisson bialgebra and denote it by $A^{p, s}_{\sigma, \omega}\# {}^{q, t}_{\theta, \nu}H$.
\end{theorem}

\begin{proof}  We only need to check the compatibility conditions
 $$
\begin{aligned}
\delta_{E}((a, x)\cdot_{E} (b,y))=& (a, x)\cdot_{E} (b, y)\bi \ot (b, y)\bii +(a, x)\bi \cdot_{E} (b, y) \ot (a, x)\bii \\
&+(b, y)\li \ot [(a, x), (b, y)\lii]_{E} +(a, x)\lii \ot [(b, y), (a, x)\li]_{E},
 \end{aligned}
$$
$$
\begin{aligned}
\Delta_{E} ([(a, x), (b, y)]_{E})=&[(a, x), (b, y)\li]_{E} \ot (b, y)\lii +(b, y)\li \ot [(a, x), (b, y)\lii]_{E} \\
&-(a, x)\bi \ot (b, y)\cdot_{E} (a, x)\bii+ (a, x)\bi\cdot_{E} (b, y) \ot (a, x)\bii .
\end{aligned}
$$
For the first equation, the left hand side is equal to
\begin{eqnarray*}
&&\delta_{E}((a, x)\cdot_{E} (b, y))\\
&=&\delta_{E}(a b+x\rightharpoonup b +a\ppl y+\omega(x, y), x y+x\leftarrow b+a\ra y+\nu(a, b))\\
&=&\delta_{A}(a b)+ \delta_{A}(x\rightharpoonup b) +\delta_{A}(a\ppl y)+\delta_A(\omega(x, y)) +\phi(a b) +\phi(x\rightharpoonup b) \\
&&+\phi(a\ppl y)+\phi(\omega(x, y)) -\tau\phi(a b) -\tau\phi(x\rightharpoonup b) -\tau\phi(a\ppl y) -\tau\phi(\omega(x, y))\\
&&+p(ab)+p(x\rightharpoonup b) +p(a\ppl y)+p(\omega(x, y))+\delta_{H}(x y)+\delta_{H}(x\leftarrow b)\\
&&+\delta_{H}(a\ra y)+\delta_H(\nu(a, b))+\psi(xy)+\psi(x\leftarrow b)+\psi(a\ra y)+\psi(\nu(a, b))\\
&&-\tau\psi(xy)-\tau\psi(x\leftarrow b)-\tau\psi(a\ra y)-\tau\psi(\nu(a, b))+q(x y)+q(x\leftarrow b)\\
&&+q(a\ra y)+q(\nu(a, b)),
\end{eqnarray*}
and the right hand side is equal to
\begin{eqnarray*}
&&(a, x)\cdot_{E} (b, y)\bi \ot (b, y)\bii +(a, x)\bi \cdot_{E} (b, y) \ot (a, x)\bii \\
&&+(b, y)\li \ot [(a, x), (b, y)\lii]_{E} +(a, x)\lii \ot [(b, y), (a, x)\li]_{E}\\
&=&a b\bi\ot b\bii+(x\ppr b\bi)\ot b\bii+(x\leftarrow b\bi)\ot b\bii+\nu(a, b\bi)\ot b\bii\\
&&+(a\ppl b_{\langle-1\rangle})\ot b_{\langle0\rangle}+\omega(x, b_{\langle-1\rangle})\ot b_{\langle0\rangle}+x b_{\langle-1\rangle}\ot b_{\langle0\rangle}
+(a\rightarrow b_{\langle-1\rangle})\ot b_{\langle0\rangle}\\
&&-a b_{\langle0\rangle}\ot b_{\langle-1\rangle}-(x\ppr b_{\langle0\rangle})\ot b_{\langle-1\rangle}
-(x\leftarrow b_{\langle0\rangle})\ot b_{\langle-1\rangle}-\nu(a,b_{\langle0\rangle})\ot b_{\langle-1\rangle}\\
&&+(a\ppl b_{1p})\ot b_{2p}+\omega(x, b_{1p})\ot b_{2p}+xb_{1p}\ot b_{2p}+(a\rightarrow b_{1p})\ot b_{2p}\\
&&+(a\ppl y\bi)\ot y\bii+\omega(x, y\bi)\ot y\bii+(a\rightarrow y\bi)\ot y\bii +x y\bi\ot y\bii
 \\
&&+(a\ppl y_{\langle0\rangle})\ot y_{\langle1\rangle}+\omega(x, y_{\langle0\rangle})\ot y_{\langle1\rangle}+x y_{\langle0\rangle} \ot y_{\langle1\rangle}+(a\rightarrow y_{\langle0\rangle})\ot y_{\langle1\rangle}\\
&&-a y_{\langle1\rangle}\ot y_{\langle0\rangle}-(x\ppr y_{\langle1\rangle})\ot y_{\langle0\rangle}
-(x\leftarrow y_{\langle1\rangle})\ot y_{\langle0\rangle}-\nu(a, y_{\langle1\rangle})\ot y_{\langle0\rangle}\\
&&+ay_{1q}\ot y_{2q}+(x\ppr y_{1q})\ot y_{2q}+(x\leftarrow y_{1q})\ot y_{2q}+\nu(a, y_{1q})\ot y_{2q}\\
&&+a\bi b\ot a\bii +(a\bi\ppl y)\ot a\bii +(a\bi\rightarrow y)\ot a\bii+\nu(a\bi, b)\ot a\bii\\
&&+(a_{\langle-1\rangle}\ppr b)\ot a_{\langle0\rangle} +\omega(a_{\langle-1\rangle}, y)\ot a_{\langle0\rangle}+a_{\langle-1\rangle} y\ot a_{\langle0\rangle}
+(a_{\langle-1\rangle}\leftarrow b)\ot a_{\langle0\rangle}\\
&&- a_{\langle0\rangle} b\ot a_{\langle-1\rangle}-(a_{\langle0\rangle}\ppl y)\ot a_{\langle-1\rangle} -( a_{\langle0\rangle}\rightarrow y)\ot a_{\langle-1\rangle}-\nu(a_{\langle0\rangle}, b)\ot a_{\langle-1\rangle}\\
&&+(a_{1p}\ppr b)\ot a_{2p}+\omega(a_{1p}, y)\ot a_{2p}+a_{1p}y\ot a_{2p}+(a_{1p}\leftarrow b)\ot a_{2p}\\
&&+(x\bi\ppr b)\ot x\bii+\omega(x\bi, y)\ot x\bii +x\bi y\ot x\bii+(x\bi\leftarrow b)\ot x\bii\\
&&+(x_{\langle0\rangle}\ppr b)\ot x_{\langle1\rangle} +\omega(x_{\langle0\rangle}, y)\ot x_{\langle1\rangle}+x_{\langle0\rangle} y\ot x_{\langle1\rangle}
+(x_{\langle0\rangle}\leftarrow b)\ot x_{\langle1\rangle}\\
&&-x_{\langle1\rangle}b\ot x_{\langle0\rangle}-(x_{\langle1\rangle}\ppl y)\ot x_{\langle0\rangle}
-(x_{\langle1\rangle}\rightarrow y)\ot x_{\langle0\rangle}-\nu(x_{\langle1\rangle}, b)\ot x_{\langle0\rangle}\\
&&+x_{1q}b\ot x_{2q}+(x_{1q}\ppl y)\ot x_{2q}+(x_{1q}\rightarrow y)\ot x_{2q}+\nu(x_{1q},b)\ot x_{2q}\\
&&+b\li\ot [a, b\lii]+b\li\ot (x\trr b\lii)+b\li\ot (x\trl b\lii)+b\li\ot\theta(a, b\lii)\\
&&+b\loi\ot [a, b\loo]+b\loi\ot (x\trr b\loo)+b\loi\ot (x\trl b\loo)+b\loi\ot\theta(a, b\loo)\\
&&-b\loo\ot (b\lmi\trr a)+b\loo\ot\sigma(x, b\lmi)+b\loo\ot [x, b\lmi]-b\loo\ot (b\lmi\trl a)\\
&&-b_{1s}\ot(b_{2s}\trr a)+b_{1s}\ot\sigma(x, b_{2s})+b_{1s}\ot[x, b_{2s}]-b_{1s}\ot(b_{2s}\trl a)\\
&&-y\li\ot (y\lii\trr a)+y\li\ot\sigma(x, y\lii)+y\li\ot [x,y\lii]-y\li\ot (y\lii\trl a)\\
&&-y\qoi\ot(y\qoo\trr a)+y\qoi\ot\sigma(x, y\qoo)+y\qoi\ot[x, y\qoo]-y\qoi\ot(y\qoo\trl a)\\
&&+y\qoo\ot[a, y\qi]+y\qoo\ot (x\trr y\qi)+y\qoo\ot(x\trl y\qi)+y\qoo\ot\theta(a, y\qi)\\
&&+y_{1t}\ot[a, y_{2t}]+y_{1t}\ot(x\trr y_{2t})+y_{1t}\ot(x\trl y_{2t})+y_{1t}\ot\theta(a, y_{2t})\\
&&+a\lii\ot [b, a\li]+a\lii\ot (y\trr a\li)+a\lii\ot (y\trl a\li)+a\lii\ot\theta(b, a\li)\\
&&-a\loo\ot (a\loi\trr b)+a\loo\ot\sigma(y, a\loi)+a\loo\ot[y, a\loi]-a\loo\ot (a\loi\trl b)\\
&&+a\lmi\ot [b, a\loo]+a\lmi\ot (y\trr a\loo)+a\lmi\ot (y\trl a\loo)+a\lmi\ot\theta(b, a\loo)\\
&&-a_{2s}\ot(a_{1s}\trr b)+a_{2s}\ot\sigma(y, a_{1s})+a_{2s}\ot[y, a_{1s}]-a_{2s}\ot(a_{1s}\trl b)\\
&&-x\lii\ot (x\li\trr b)+x\lii\ot\sigma(y,x\li)+x\lii\ot[y,x\li]-x\lii\ot (x\li\trl b)\\
&&+x\qoo\ot[b, x\qoi]+x\qoo\ot(y\trr x\qoi)+x\qoo\ot(y\trl x\qoi)+x\qoo\ot\theta(b, x\qoi)\\
&&-x\qi\ot(x\qoo\trr b)+x\qi\ot\sigma(y, x\qoo)+x\qi\ot[y, x\qoo]-x\qi\ot(x\qoo\trl b)\\
&&+x_{2t}\ot[b, x_{1t}]+x_{2t}\ot(y\trr x_{1t})+x_{2t}\ot(y\trl x_{1t})+x_{2t}\ot\theta(b, x_{1t}).
\end{eqnarray*}
If we compare both the two sides item by item, one will find all the  cocycle double matched pair conditions (CDM1)--(CDM12) in
Definition \ref{cocycledmp}.

For the second equation, the left hand side is equal to
\begin{eqnarray*}
&&\Delta_{E} ([(a, x), (b, y)]_{E})\\
&=&\Delta_E( [a, b]+x\trr b-y\trr a+\sigma(x, y), [x, y]+x\trl b-y\trl a+\theta(a, b))\\
&=&\Delta_A( [a, b])+\Delta_A(x\trr b)-\Delta_A(y\trr a)+\Delta_A(\sigma(x, y))+\rho([a, b])+\rho(x\trr b)\\
&&-\rho(y\trr a)+\rho(\sigma(x, y))+\gamma([a, b])+\gamma(x\trr b)-\gamma(y\trr a)+\gamma(\sigma(x, y))\\
&&+s([a, b])+s(x\trr b)-s(y\trr a)+s(\sigma(x, y))+\Delta_{H}([x, y])+\Delta_{H}(x\trl b)\\
&&-\Delta_{H}(y\trl a)+\Delta_H(\theta(a, b))+\al([x, y])+\al(x\trl b)-\al(y\trl a)+\al(\theta(a, b))\\
&&+\beta([x, y])+\beta(x\trl b)-\beta(y\trl a)+\beta(\theta(a, b))+t([x, y])+t(x\trl b)\\
&&-t(y\trl a)+t(\theta(a, b)),
\end{eqnarray*}
and the right hand side is equal to
\begin{eqnarray*}
&&[(a, x), (b, y)\li]_{E} \ot (b, y)\lii +(b, y)\li \ot [(a, x), (b, y)\lii]_{E} \\
&&-(a, x)\bi \ot (b, y)\cdot_{E} (a, x)\bii+ (a, x)\bi\cdot_{E} (b, y) \ot (a, x)\bii \\
&=&[a, b\li]\ot b\lii+(x\trr b\li)\ot b\lii+(x\trl b\li)\ot b\lii+\theta(a, b\li)\ot b\lii\\
&&-(b\loi\trr a)\ot b\loo+\sigma(x, b\loi)\ot b\loo+[x, b\loi]\ot b\loo-(b\loi\trl a)\ot b\loo\\
&&+[a, b\loo]\ot b\lmi+(x\trr b\loo)\ot b\lmi+(x\trl b\loo)\ot b\lmi+\theta(a, b\loo)\ot b\lmi\\
&&-(b_{1s}\trr a)\ot b_{2s}+\sigma(x, b_{1s})\ot b_{2s}+[x, b_{1s}]\ot b_{2s}-(b_{1s}\trl a)\ot b_{2s}\\
&&-(y\li\trr a)\ot y\lii+\sigma(x, y\li)\ot y\lii+[x, y\li]\ot y\lii-(y\li\trl a)\ot y\lii\\
&&+[a, y\qoi]\ot y\qoo+(x\trr y\qoi)\ot y\qoo+(x\trl y\qoi)\ot y\qoo+\theta(a, y\qoi)\ot y\qoo\\
&&-(y\qoo\trr a)\ot y\qi+\sigma(x, y\qoo)\ot y\qi+[x, y\qoo]\ot y\qi-(y\qoo\trl a)\ot y\qi\\
&&+[a, y_{1t}]\ot y_{2t}+(x\trr y_{1t})\ot y_{2t}+(x\trl y_{1t})\ot y_{2t}+\theta(a, y_{1t})\ot y_{2t}\\
&&+b\li\ot[a, b\lii]+b\li\ot(x\trr b\lii)+b\li\ot(x\trl b\lii)+b\li\ot\theta(a, b\lii)\\
&&+b\loi\ot[a, b\loo]+b\loi\ot(x\trr b\loo)+b\loi\ot(x\trl b\loo)+b\loi\ot\theta(a, b\loo)\\
&&-b\loo\ot(b\lmi\trr a)+b\loo\ot\sigma(x, b\lmi)+b\loo\ot[x, b\lmi]-b\loo\ot(b\lmi\trl a)\\
&&-b_{1s}\ot(b_{2s}\trr a)+b_{1s}\ot\sigma(x, b_{2s})+b_{1s}\ot[x, b_{2s}]-b_{1s}\ot(b_{2s}\trl a)\\
&&-y\li\ot(y\lii\trr a)+y\li\ot\sigma(x, y\lii)+y\li\ot[x, y\lii]-y\li\ot(y\lii\trl a)\\
&&-y\qoi\ot(y\qoo\trr a)+y\qoi\ot\sigma(x, y\qoo)+y\qoi\ot[x, y\qoo]-y\qoi\ot(y\qoo\trl a)\\
&&+y\qoo\ot[a, y\qi]+y\qoo\ot(x\trr y\qi)+y\qoo\ot(x\trl y\qi)+y\qoo\ot\theta(a, y\qi)\\
&&+y_{1t}\ot[a, y_{2t}]+y_{1t}\ot(x\trr y_{2t})+y_{1t}\ot(x\trl y_{2t})+y_{1t}\ot\theta(a, y_{2t})\\
&&-a\bi\ot b a\bii-a\bi\ot (y\ppr a\bii)-a\bi\ot(y\leftarrow a\bii)-a\bi\ot\nu(b, a\bii)\\
&&-a_{\langle-1\rangle}\ot b a_{\langle0\rangle}-a_{\langle-1\rangle}\ot(y\ppr a_{\langle0\rangle})
-a_{\langle-1\rangle}\ot(y\leftarrow a_{\langle0\rangle})-a_{\langle-1\rangle}\ot\nu(b, a_{\langle0\rangle})\\
&&+a_{\langle0\rangle}\ot(b\ppl a_{\langle-1\rangle})
+a_{\langle0\rangle}\ot\omega(y, a_{\langle-1\rangle})+a_{\langle0\rangle}\ot y a_{\langle-1\rangle}
+a_{\langle0\rangle}\ot(b\rightarrow a_{\langle-1\rangle})\\
&&-a_{1p}\ot(b\ppl a_{2p})-a_{1p}\ot\omega(y, a_{2p})-a_{1p}\ot ya_{2p}-a_{1p}\ot(b\rightarrow a_{2p})\\
&&-x\bi\ot(b\ppl x\bii)-x\bi\ot\omega(y, x\bii)-x\bi\ot y x\bii-x\bi\ot(b\rightarrow x\bii)\\
&&-x_{\langle0\rangle}\ot bx_{\langle1\rangle}-x_{\langle0\rangle}\ot(y\ppr x_{\langle1\rangle})
-x_{\langle0\rangle}\ot(y\leftarrow x_{\langle1\rangle})-x_{\langle0\rangle}\ot\nu(b, x_{\langle1\rangle})\\
&&+x_{\langle1\rangle}\ot(b\ppl x_{\langle0\rangle})+x_{\langle1\rangle}\ot\omega(y, x_{\langle0\rangle})
+x_{\langle1\rangle}\ot yx_{\langle0\rangle}+x_{\langle1\rangle}\ot(b\rightarrow x_{\langle0\rangle})\\
&&-x_{1q}\ot bx_{2q}-x_{1q}\ot(y\ppr x_{2q})-x_{1q}\ot(y\leftarrow x_{2q})-x_{1q}\ot\nu(b, x_{2q})\\
&&+a\bi b\ot a\bii+(a\bi\ppl y)\ot a\bii+(a\bi\rightarrow y)\ot a\bii+\nu(a\bi, b)\ot a\bii\\
&&+(a_{\langle-1\rangle}\ppr b)\ot a_{\langle0\rangle}+\omega(a_{\langle-1\rangle}, y)\ot a_{\langle0\rangle}+a_{\langle-1\rangle}y\ot a_{\langle0\rangle}+(a_{\langle-1\rangle}\leftarrow b)\ot a_{\langle0\rangle}\\
&&-a_{\langle0\rangle}b\ot a_{\langle-1\rangle}-(a_{\langle0\rangle}\ppl y)\ot a_{\langle-1\rangle}
-(a_{\langle0\rangle}\rightarrow y)\ot a_{\langle-1\rangle}-\nu(a_{\langle0\rangle}, b)\ot a_{\langle-1\rangle}\\
&&+(a_{1p}\ppr b)\ot a_{2p}+\omega(a_{1p}, y)\ot a_{2p}+a_{1p}y\ot a_{2p}+(a_{1p}\leftarrow b)\ot a_{2p}\\
&&+(x\bi\ppr b)\ot x\bii+\omega( x\bi, y)\ot x\bii+x\bi y\ot x\bii+(x\bi\leftarrow b)\ot x\bii\\
&&+(x_{\langle0\rangle}\ppr b)\ot x_{\langle1\rangle}+\omega(x_{\langle0\rangle}, y)\ot x_{\langle1\rangle}+x_{\langle0\rangle} y\ot x_{\langle1\rangle}
+(x_{\langle0\rangle}\leftarrow b)\ot x_{\langle1\rangle}\\
&&-x_{\langle1\rangle} b\ot x_{\langle0\rangle}-(x_{\langle1\rangle}\ppl y)\ot x_{\langle0\rangle}-( x_{\langle1\rangle}\rightarrow y)\ot x_{\langle0\rangle}-\nu(x_{\langle1\rangle}, b)\ot x_{\langle0\rangle}\\
&&+x_{1q}b\ot x_{2q}+(x_{1q}\ppl y)\ot x_{2q}+(x_{1q}\rightarrow y)\ot x_{2q}+\nu(x_{1q}, b)\ot x_{2q}.
\end{eqnarray*}
If we compare both the two sides term by term, one obtain all the  cocycle double matched pair conditions (CDM13)--(CDM24) in
Definition \ref{cocycledmp}.

This complete the proof.
\end{proof}

\section{Extending structures for noncommutative Poisson bialgebras}
In this section, we will study the extending problem for noncommutative Poisson  bialgebras.
We will find some special cases when the braided noncommutative Poisson  bialgebra is reduced into an ordinary     noncommutative Poisson  bialgebra.
It is proved that the extending problem can be solved by using of the non-abelian cohomology theory based on our cocycle bicrossedproduct for braided noncommutative Poisson bialgebras in last section.

\subsection{Extending structures for noncommutative Poisson algebras }
First we are going to study extending problem for noncommutative Poisson algebras.

There are two cases for $A$ to be a noncommutative Poisson algebra in the cocycle cross product system defined in last section, see condition (CC6). The first case is when we let $\ppr$,$\ppl$ and $\trr$ to be trivial and $\theta\neq 0, \nu\neq 0$,  then from conditions (CP1) and (CP4) we get $\si(x, \nu(a, b))=\omega(x, \theta(a, b))=0$, since $\theta\neq 0, \nu\neq 0$ we assume $\sigma=0, \omega=0$ for simplicity, thus  we obtain the following type $(a1)$  unified product for noncommutative Poisson algebras.

\begin{lemma}(\cite{AM5})
Let ${A}$ be a noncommutative Poisson algebra and $V$ a vector space. An extending datum of ${A}$ by $V$ of type (a1)  is  $\Omega^{(1)}({A}, V)$ consisting of bilinear maps
\begin{eqnarray*}
&& \triangleleft : V \otimes A \to V,  \quad \theta: A\otimes A \to V,
\quad \leftarrow: V \otimes A \to V, \quad \rightarrow: A \otimes V \to V,  \quad \nu: A\otimes A \to V.
\end{eqnarray*}
Denote by $A_{}\#_{\theta, \nu}V$ the vector space $E={A}\oplus V$ together with the multiplication given by
\begin{eqnarray}
\left [(a, x), \, (b, y) \right ] &: =& \bigl( [a, \, b], \, \,\,
[x, y]+x\trl b - y \trl a + \theta(a, b) \bigl),\\
(a, \, x) \cdot (b, \, y) &: =& \bigl( ab , \,\, xy+x\leftarrow b + a\rightarrow y  +
\nu(a, b) \bigl). \label{multunifJ}
\label{brackunifJ}
\end{eqnarray}
Then $A_{}\# {}_{\theta, \nu}V$ is a noncommutative Poisson algebra if and only if the following compatibility conditions hold for all $a$, $b\in {A}$, $x$, $y$, $z\in V$:
\begin{enumerate}
\item[(A0)] $\bigl(\rightarrow, \, \leftarrow  \, \nu)$ is an algebra extending system of the associative algebra
$A$ trough $V$ and $\bigl(\trl,  \,
\theta \bigl)$ is a Lie extending system of the
Lie algebra $A$ trough $V$,
\item[(A1)] $[x, y\leftarrow a]=[x, y]\leftarrow a+y(x\trl a)$,
\item[(A2)] $[x, a\rightarrow y]=a\rightarrow [x, y]+(x\trl a)y$,
\item[(A3)] $[x, \nu(a, b)]+x\trl(ab)=(x\trl a)\leftarrow b+a\rightarrow (x\trl b)$,
\item[(A4)] $(xy)\trl a=(x\trl a)y+x(y\trl a)$,
\item[(A5)] $(x\leftarrow b)\trl a=(x\trl a)\leftarrow b-x\theta(a, b)-x\leftarrow [a, b]$,
\item[(A6)] $(b\rightarrow x)\trl a=b\rightarrow (x\trl a)-\theta(a, b)x-[a, b]\rightarrow x$,
\item[(A7)] $[x, yz]=[x, y]z+y[x, z]$.
\end{enumerate}
\end{lemma}
Note that (A1)--(A6)  are deduced from (CP7)--(CP12) and by (A7)  we obtain that $V$ is a noncommutative Poisson algebra. Furthermore, $V$ is in fact a noncommutative Poisson subalgebra of $A_{}\#_{\theta, \nu}V$  but $A$ is not although $A$ is itself a noncommutative Poisson algebra.

Denote the set of all  algebraic extending datum of ${A}$ by $V$ of type (a1)  by $\mathcal{A}^{(1)}({A}, V)$.


In the following, we always assume that $A$ is a subspace of a vector space $E$, there exists a projection map $p: E \to{A}$ such that $p(a) = a$, for all $a \in {A}$.
Then the kernel space $V: = \ker(p)$ is also a subspace of $E$ and a complement of ${A}$ in $E$.

\begin{lemma}\label{lem:33-1}(\cite{AM5})
Let ${A}$ be a noncommutative Poisson algebra and $E$ a vector space containing ${A}$ as a subspace.
Suppose that there is a noncommutative Poisson algebra structure on $E$ such that $V$ is a noncommutative Poisson subalgebra of $E$
and the canonical projection map $p: E\to A$ is a noncommutative Poisson algebra homomorphism.
Then there exists a noncommutative Poisson algebraic extending datum $\Omega^{(1)}({A}, V)$ of ${A}$ by $V$ such that
$E\cong A_{}\#_{\theta, \nu}V$.
\end{lemma}

\begin{proof}
Since $V$ is a noncommutative Poisson subalgebra of $E$, we have $x\cdot_E y\in V$ for all $x, y\in V$.
We define the extending datum of ${A}$ through $V$ by the following formulas:
\begin{eqnarray*}
\trl: V\otimes {A} \to V, \qquad {x} \triangleleft {a} &: =&[{x}, {a}]_E-p([x, a]_E),\\
\theta: A\otimes A \to V, \qquad \theta(a, b) &: =&[a, b]_E-p \bigl([a, b]_E\bigl),\\
{[ , ]_V}: V \otimes V \to V, \qquad [x, y]_V&: =& [{x}, {y}]_E, \\
\leftarrow: V\otimes {A} \to V, \qquad {x} \leftarrow {a} &: =&{x}\cdot_E {a}-p({x}\cdot_E {a}),\\
\rightarrow: A\otimes {V} \to V, \qquad {a} \rightarrow {x} &: =&{a}\cdot_E {x}-p({a}\cdot_E {x}),\\
\nu: A\otimes A \to V, \qquad \nu(a, b) &: =&a\cdot_E b-p \bigl(a\cdot_E b\bigl),\\
{\cdot_V}: V \otimes V \to V, \qquad {x}\cdot_V {y}&: =& {x}\cdot_E{y},
\end{eqnarray*}
for any $a , b\in {A}$ and $x, y\in V$. It is easy to see that the above maps are  well defined and
$\Omega^{(1)}({A}, V)$ is an extending system of
${A}$ trough $V$ and
\begin{eqnarray*}
\varphi:A_{}\#_{\theta, \nu}V\to E, \qquad \varphi(a, x): = a+x
\end{eqnarray*}
is an isomorphism of noncommutative Poisson algebras.
\end{proof}

\begin{lemma}\label{deform-01}
Let $\Omega^{(1)}(A, V) = \bigl(\leftarrow, \rightarrow,  \trl, \theta, \nu, \cdot, [ , ] \bigl)$ and $\Omega'^{(1)}(A, V) = \bigl(\leftarrow ', \rightarrow', \trl', \theta', \nu', \cdot ' , [ ,]'\bigl)$
be two algebraic extending datums of ${A}$ by $V$ of type (a1) and $A_{}\#_{\theta, \nu} V$, $A_{}\#_{\theta', \nu'} V$ be the corresponding unified products. Then there exists a bijection between the set of all homomorphisms of noncommutative Poisson algebras $\varphi: A_{\theta,\nu}\#_{\leftarrow, \rightarrow, \trl} V\to A_{\theta', \nu'}\#_{\leftarrow', \rightarrow', \trl'} V$ whose restriction on ${A}$ is the identity map and the set of pairs $(r, s)$, where $r: V\rightarrow {A}$ and $s: V\rightarrow V$ are two linear maps satisfying
\begin{eqnarray*}
&&{r}(x\trl a)=[{r}(x), a],\\
&&[a, b]'=[a, b]+r\theta(a, b),\\
&&{r}([x, y])=[{r}(x), {r}(y)]',\\
&&{s}(x)\trl' a+\theta'(r(x), a)={s}(x\trl a),\\
&&\theta'(a, b)=s\theta(a, b),\\
&&{s}([x, y])=[{s}(x), {s}(y)]'+{s}(x)\trl'{r}(y)-{s}(y)\trl'{r}(x)+\theta'(r(x), r(y)),\\
&&{r}(x\leftarrow a)={r}(x)\cdot' a,\\
&&{r}(a\rightarrow y)=a\cdot' {r}(x),\\
&&a\cdot' b=ab+r\nu(a, b),\\
&&{r}(xy)={r}(x)\cdot' {r}(y),\\
&&{s}(x)\leftarrow' a+\nu'(r(x), a)={s}(x\leftarrow a),\\
&&a\rightarrow' {s}(x)+\nu'(a, r(x))={s}(a\rightarrow x),\\
&&\nu'(a,b)=s\nu(a, b),\\
&&{s}(xy)={s}(x)\cdot' {s}(y)+{s}(x)\leftarrow'{r}(y)+{r}(x)\rightarrow' {s}(y)+\nu'(r(x), r(y)),
\end{eqnarray*}
for all $a, b\in{A}$ and $x$, $y\in V$.

Under the above bijection the homomorphism of noncommutative Poisson algebras $\varphi=\varphi_{r, s}: A_{}\#_{\theta, \nu}V\to A_{}\#_{\theta', \nu'} V$ to $(r, s)$ is given  by $\varphi(a, x)=(a+r(x), s(x))$ for all $a\in {A}$ and $x\in V$. Moreover, $\varphi=\varphi_{r, s}$ is an isomorphism if and only if $s: V\rightarrow V$ is a linear isomorphism.
\end{lemma}

\begin{proof}
Let $\varphi: A_{}\#_{\theta, \nu}V\to A_{}\#_{\theta', \nu'} V$  be a noncommutative Poisson algebra homomorphism  whose restriction on ${A}$ is the identity map. Then $\varphi$ is determined by two linear maps $r: V\rightarrow {A}$ and $s: V\rightarrow V$ such that
$\varphi(a, x)=(a+r(x), s(x))$ for all $a\in {A}$ and $x\in V$.
In fact, we have to show
$$\varphi([(a, x), (b, y)])=[\varphi(a, x), \varphi(b, y)]',$$
$$\varphi((a, x)(b, y))=\varphi(a, x)\cdot'\varphi(b, y).$$
For the first equation, the left hand side is equal to
\begin{eqnarray*}
&&\varphi([(a, x), (b, y)])\\
&=&\varphi\left([a, b], \,  x\trl b-y\trl a+[x, y]+\theta(a, b)\right)\\
&=&\big([a, b]+ r(x\trl b)-r(y\trl a)+r([x, y])+r\theta(a, b),\\
&&\qquad\quad s(x\trl b)-s(y\trl a)+s([x, y])+s\theta(a, b)\big),
\end{eqnarray*}
and the right hand side is equal to
\begin{eqnarray*}
&&[\varphi(a, x), \varphi(b, y)]'\\
&=&[(a+r(x), s(x)),  (b+r(y), s(y))]'\\
&=&\big([a+r(x), b+r(y)]',  s(x)\trl'(b+r(y))-s(y)\trl'(a+r(x))\\
&&\qquad\qquad +[s(x), s(y)]'+\theta'(a+r(x), b+r(y))\big).
\end{eqnarray*}
For the second equation, the left hand side is equal to
\begin{eqnarray*}
&&\varphi((a, x)(b, y))\\
&=&\varphi\left({ab},\,  x\leftarrow b+a\rightarrow y+{xy}+\nu(a, b)\right)\\
&=&\big({ab}+ r(x\leftarrow b)+r(a\rightarrow y)+r({xy})+r\nu(a, b),\\
&&\qquad\quad s(x\leftarrow b)+s(a\rightarrow y)+s({xy})+s\nu(a, b)\big),
\end{eqnarray*}
and the right hand side is equal to
\begin{eqnarray*}
&&\varphi(a, x)\cdot' \varphi(b, y)\\
&=&(a+r(x), s(x))\cdot'  (b+r(y), s(y))\\
&=&\big((a+r(x))\cdot' (b+r(y)),  s(x)\leftarrow'(b+r(y))+(a+r(x))\rightarrow' s(y)\\
&&\qquad\qquad +s(x)\cdot' s(y)+\nu'(a+r(x), b+r(y))\big).
\end{eqnarray*}
Thus $\varphi$ is a homomorphism of noncommutayive Poisson algebras if and only if the above conditions hold.
\end{proof}

The second case is when $\theta=0, \nu=0$,  we obtain the following type (a2)  unified product.

%
%
%

\begin{theorem}(\cite{AM5})
Let ${A}$ be a noncommutative Poisson algebra and $V$ a vector space. An extending datum of ${A}$ through $V$ of type (a1)  is  $\Omega^{(2)}({A}, V)$ consisting of bilinear maps
\begin{eqnarray*}
&& \triangleleft: V \otimes A \to V, \quad \triangleright: V
\otimes A \to A, \quad \sigma: V\otimes V \to A , \quad\rightharpoonup: V\otimes A \to A ,\\
&&\leftharpoonup: A \otimes V \to A , \quad\rightarrow: A\otimes V \to V ,
\quad\leftarrow: V\otimes A \to V ,\quad \omega: V\otimes V \to A.
\end{eqnarray*}
Denote by $A_{\sigma, \omega}\#_{}V$ the vector space $E={A}\oplus V$ together with the multiplication given by
\begin{eqnarray}
\left [(a, x), \, (b, y) \right ] &: =& \bigl( [a, \, b] + x
\trr b - y \trr a + \sigma (x, y), \,\,
[x, y]+x\trl b - y \trl a \bigl),\\
(a, \, x) \cdot (b, \, y) &: =& \bigl( ab + x \ppr b +
a\ppl y + \omega(x, y), \,\, xy+x\leftarrow b + a\rightarrow y  \bigl). \label{multunifJ}
\label{brackunifJ}
\end{eqnarray}
Then $A_{\sigma, \omega}\# {}_{}V$ is a noncommutative Poisson algebra if and only if the following compatibility conditions hold for all $a$, $b\in {A}$, $x$, $y$, $z\in V$:
\begin{enumerate}
\item[(B0)] $\bigl(\ppr, \, \ppl, \, \rightarrow, \, \leftarrow, \, \omega)$ is an algebra extending system of the associative algebra
$A$ trough $V$ and $\bigl(\trr, \, \trl, \,
\sigma \bigl)$ is a Lie extending system of the
Lie algebra $A$ trough $V$,
\item[(B1)] $[a, x\ppr b]-(x\leftarrow b)\trr a=x\ppr[a, b]-(x\trr a)b-(x\trl a)\ppr b$,
\item[(B2)] $[a, b\ppl x]-(b\rightarrow x)\trr a=[a, b]\ppl x-b(x\trr a)-b\ppl (x\trl a)$,
\item[(B3)] $(xy)\trr a-[a, \omega(x, y)]=(x\trr a)\ppl y+\omega(x\trl a, y)+x\ppr(y\trr a)+\omega(x, y\trl a)$,
\item[(B4)] $x\trr(ab)=(x\trr a)b+(x\trl a)\ppr b+a(x\trr b)+a\ppl (x\trl b)$,
\item[(B5)] $x\trr(y\ppr a)+\sigma(x, y\leftarrow a)=\sigma(x, y)a+[x, y]\ppr a+y\ppr(x\trr a)+\omega(y, x\trl a) $,
\item[(B6)] $x\trr(a\ppl y)+\sigma(x, a\rightarrow y)=a\sigma(x, y)+a\ppl [x, y]+(x\trr a)\ppl y+\omega(x\trl a, y) $,
\item[(B7)] $[x, y\leftarrow a]+x\trl(y\ppr a)=[x, y]\leftarrow a+y(x\trl a)+y\leftarrow(x\trr a)$,
\item[(B8)] $[x, a\rightarrow y]+x\trl(a\ppl y)=a\rightarrow [x, y]+(x\trl a)y+(x\trr a)\rightarrow y$,
\item[(B9)] $x\trl(ab)=(x\trl a)\leftarrow b+a\rightarrow (x\trl b)$,
\item[(B10)] $(xy)\trl a=(x\trl a)y+(x\trr a)\rightarrow y+x(y\trl a)+x\leftarrow (y\trr a)$,
\item[(B11)] $(x\leftarrow b)\trl a=(x\trl a)\leftarrow b-x\leftarrow [a, b]$,
\item[(B12)] $(b\rightarrow x)\trl a=b\rightarrow (x\trl a)-[a, b]\rightarrow x$,
\item[(B13)] $[x, yz]=[x, y]z+y[x, z]+\sigma(x, y)\rightarrow z+y\leftarrow \sigma(x, z)-x\trl \omega(y, z)$.
\end{enumerate}
\end{theorem}

\begin{theorem}(\cite{AM5})
Let $A$ be a noncommutative Poisson algebra, $E$ a  vector space containing $A$ as a subspace.
If there is a noncommutative Poisson algebra structure on $E$ such that $A$ is a noncommutative Poisson subalgebra of $E$. Then there exists a noncommutative Poisson algebraic extending structure $\Omega(A, V) = \bigl(\triangleleft, \, \triangleright, \,
\leftharpoonup, \, \rightharpoonup, \, \leftarrow, \,\rightarrow, \,\sigma, \,\omega \bigl)$ of $A$ through $V$ such that there is
an isomorphism of noncommutative Poisson algebras $E\cong A_{\sigma, \omega}\#_{}V$.
\end{theorem}

\begin{lemma}
Let $\Omega^{(1)}(A, V) = \bigl(\trr, \trl, \leftharpoonup,  \rightharpoonup, \leftarrow, \rightarrow, \sigma,  \omega,  \cdot , [,]\bigl)$ and
$\Omega'^{(1)}(A, V) = \bigl(\trr', \trl ', \leftharpoonup ',  \rightharpoonup ', \leftarrow', \rightarrow', \sigma ', \omega', \cdot ', [,]' \bigl)$
be two noncommutative Poisson  algebraic extending structures of $A$ through $V$ and $A{}_{\sigma, \omega}\#_{}V$, $A{}_{\sigma', \omega'}\#_{}  V$ the  associated unified
products. Then there exists a bijection between the set of all
homomorphisms of algebras $\psi: A{}_{\sigma, \omega}\#_{}V\to A{}_{\sigma', \omega'}\#_{}  V$which
stabilize $A$ and the set of pairs $(r, s)$, where $r: V \to
A$, $s: V \to V$ are linear maps satisfying the following
compatibility conditions for any $a, b \in A$, $x$, $y \in V$:
\begin{eqnarray*}
&&r([x, y]) = [r(x), r(y)]' + \sigma ' (s(x), s(y)) - \sigma(x, y) + s(x) \trr' r(y) - s(y) \trr' r(x),\\
&&s([x, y]) = s(x) \trl ' r(y) - s(y)\trl ' r(x) + [s(x), s(y)]',\\
&&r(x\trl  {a}) = [r(x), {a}]' + s(x) \trr' {a} -x\trr a,\\
&&s(x\trl {a}) = s(x)\trl' {a},\\
&&r(x \cdot y) = r(x)\cdot'r(y) + \omega' (s(x), s(y)) - \omega(x, y) + s(x) \ppr' r(y) +  r(x)\ppl' s(y),\\
&&s(x \cdot y) = r(x) \rightarrow' s(y)+ s(x)\leftarrow ' r(y) + s(x) \cdot ' s(y),\\
&&r(x\leftarrow  {b}) = r(x)\cdot' {b} - x \ppr {b} + s(x) \ppr' {b},\\
&&r(a\rightarrow y) = a\cdot' r(y) - a \ppl {y} + a\ppl' s(y),\\
&&s(x\leftarrow {b}) = s(x)\leftarrow' {b},\\
&&s(a\rightarrow {y}) = a\rightarrow' s(y).
\end{eqnarray*}
Under the above bijection the homomorphism of algebras $\varphi =\varphi _{(r, s)}: A_{\sigma, \omega}\# {}_{}V \to A_{\sigma', \omega'}\# {}_{}V$ corresponding to
$(r, s)$ is given for any $a\in A$ and $x \in V$ by:
$$\varphi(a, x) = (a + r(x), s(x)).$$
Moreover, $\varphi  = \varphi _{(r, s)}$ is an isomorphism if and only if $s: V \to V$ is an isomorphism linear map.
\end{lemma}

The proof of the above is similar as to the proof of Lemma \ref{deform-01}, so we omit the details.

Let ${A}$ be a noncommutative Poisson algebra and $V$ a vector space. Two algebraic extending systems $\Omega^{(i)}({A}, V)$ and ${\Omega'^{(i)}}({A}, V)$  are called equivalent if $\varphi_{r, s}$ is an isomorphism.  We denote it by $\Omega^{(i)}({A}, V)\equiv{\Omega'^{(i)}}({A}, V)$.
From the above lemmas, we obtain the following result.

\begin{theorem}\label{thm3-1}
Let ${A}$ be a noncommutative Poisson algebra, $E$ a vector space containing ${A}$ as a subspace and
$V$ be a complement of ${A}$ in $E$.
Denote $\mathcal{HA}(V, {A}): =\mathcal{A}^{(1)}({A}, V)\sqcup \mathcal{A}^{(2)}({A}, V) /\equiv$. Then the map
\begin{eqnarray}
\notag&&\Psi: \mathcal{HA}(V, {A})\rightarrow Extd(E, {A}),\\
&&\overline{\Omega^{(1)}({A}, V)}\mapsto A_{}\#_{\theta, \nu} V, \quad \overline{\Omega^{(2)}({A}, V)}\mapsto A_{\sigma, \omega}\# {}_{} V
\end{eqnarray}
is bijective, where $\overline{\Omega^{(i)}({A}, V)}$ is the equivalence class of
$\Omega^{(i)}({A}, V)$ under $\equiv$.
\end{theorem}

\subsection{Extending structures for noncommutative  Poisson  coalgebras}

Next we consider the noncommutative Poisson coalgebra structures on $E=A^{p, s}\# {}^{q, t}V$.

There are two cases for $(A, \Delta_A, \delta_A)$ to be a noncommutative Poisson  coalgebra. The first case is  when $q=0, t=0$,  then we obtain the following type (c1) unified product for noncommutative Poisson  coalgebras.
\begin{lemma}\label{cor02}
Let $({A}, \Delta_A, \delta_A)$ be a noncommutative Poisson  coalgebra and $V$ a vector space.
An  extending datum  of ${A}$ by $V$ of  type (c1) is  $\Omega^{(3)}({A}, V)=(\phi, {\psi}, \rho, \gamma, \alpha, \beta, p, s, \Delta_V, \delta_V)$ with  linear maps
\begin{eqnarray*}
&&\Delta_V: V\to V\ot V, \quad \delta_V: V\to V\otimes V,\\
&&\phi: A \to V \otimes A, \quad  \psi: V \to V\otimes A,\\
&&\rho: A  \to V\otimes A, \quad  \gamma: A \to A \otimes V,\\
&&\alpha: V  \to A\otimes V, \quad  \beta: V \to V \otimes A,\\
&& {p}: A\to {V}\otimes {V}, \quad s: A\to V\otimes V.
\end{eqnarray*}
 Denote by $A^{p, s}\# {}^{} V$ the vector space $E={A}\oplus V$ with the linear maps $\delta_E: E\rightarrow E\otimes E$ , $\Delta_E: E\rightarrow E\otimes E$ given by
$$\delta_{E}(a)=(\delta_{A}+\phi-\tau\phi+p)(a), \quad \delta_{E}(x)=(\delta_{V}+\psi-\tau\psi)(x), $$
$$\Delta_{E}(a)=(\Delta_{A}+\rho+\gamma+s)(a), \quad \Delta_{E}(x)=(\Delta_{V}+\alpha+\beta)(x), $$
that is
$$\delta_{E}(a)= a\bi \ot a\bii+ a_{\langle-1\rangle} \ot a_{\langle0\rangle}-a_{\langle0\rangle}\ot a_{\langle-1\rangle}+a_{1p}\ot a_{2p},$$
$$\delta_{E}(x)= x\bi \ot x\bii+ x_{\langle0\rangle} \ot x_{\langle1\rangle}-x_{\langle1\rangle} \ot x_{\langle0\rangle},$$
$$\Delta_{E}(a)= a\li \ot a\lii+ a\moi \ot a\mo+a\mo\ot a\lmi+a_{1s}\ot a_{2s},$$
$$\Delta_{E}(x)= x\li \ot x\lii+ x\qoi \ot x\qoo+x\qoo \ot x\qi.$$
Then $A^{p, s}\# {}^{} V$  is a  Poisson  coalgebra with the comultiplication given above if and only if the following compatibility conditions hold:
\begin{enumerate}
\item[(C0)] $\bigl(\rho, \, \gamma, \, \alpha, \, \beta, \, s)$ is an algebra extending system of the associative coalgebra
$A$ trough $V$ and $\bigl(\phi, \, \psi, \,
p \bigl)$ is a Lie extending system of the
Lie coalgebra $A$ trough $V$,
\item[(C1)] $a\bi\ot\rho(a\bii)-a_{\langle0\rangle}\ot\beta(a_{\langle-1\rangle})=-\tau\phi(a\li)\ot a\lii-\tau\psi(a\loi)\ot a\loo+\tau_{12}(a\loi\ot\delta_A(a\loo))$,
\item[(C2)] $a\bi\ot\gamma(a\bii)-a_{\langle0\rangle}\ot\al(a_{\langle-1\rangle})
=\delta_A(a\loo)\ot a\lmi-\tau_{12}(a\li\ot\tau\phi(a\lii))-\tau_{12}(a\loo\ot\tau\psi(a\lmi))$,
\item[(C3)] $a_{\langle0\rangle}\ot\Delta_V(a_{\langle-1\rangle})-a\bi\ot s(a\bii)
=\tau\phi(a\loo)\ot a\lmi+\tau\psi(a_{1s})\ot a_{2s}\\
+\tau_{12}(a\loi\ot\tau\phi(a\loo))+\tau_{12}(a_{1s}\ot\tau\psi(a_{2s}))$,
\item[(C4)] $a_{\langle-1\rangle}\ot\Delta_A(a_{\langle0\rangle})
=\phi(a\li)\ot a\lii+\psi(a\loi)\ot a\loo+\tau_{12}(a\li\ot\phi(a\lii))+\tau_{12}(a\loo\ot\psi(a\lmi))$,
\item[(C5)] $a_{\langle-1\rangle}\ot\rho(a_{\langle0\rangle})+a_{1p}\ot \beta(a_{2p})
=\delta_V(a\loi)\ot a\loo+p(a\li)\ot a\lii\\
+\tau_{12}(a\loi\ot\phi(a\loo))+\tau_{12}(a_{1s}\ot \psi(a_{2s}))$,
\item[(C6)] $a_{\langle-1\rangle}\ot\gamma(a_{\langle0\rangle})+a_{1p}\ot\al (a_{2p})
=\phi(a\loo)\ot a\lmi+\psi(a_{1s})\ot a_{2s}\\
+\tau_{12}(a\li\ot p(a\lii))+\tau_{12}(a\loo\ot \delta_V(a\lmi))$,
\item[(C7)] $x\bi\ot \beta(x\bii)+x_{\langle0\rangle}\ot\rho(x_{\langle1\rangle})
=\delta_V(x\qoo)\ot x\qi+\tau_{12}(x\li\ot \psi(x\lii))+\tau_{12}(x\qoo\ot \phi(x\qi))$,
\item[(C8)] $x\bi\ot\al(x\bii)+x_{\langle0\rangle}\ot\gamma(x_{\langle1\rangle})
=\psi(x\li)\ot x\lii+\phi(x\qoi)\ot x\qoo+\tau_{12}(x\qoi\ot\delta_V(x\qoo))$,
\item[(C9)] $x_{\langle0\rangle}\ot\Delta_A(x_{\langle1\rangle})
=\psi(x\qoo)\ot x\qi+\tau_{12}(x\qoi\ot\psi(x\qoo))$,
\item[(C10)] $x_{\langle1\rangle}\ot\Delta_V(x_{\langle0\rangle})
=\tau\psi(x\li)\ot x\lii+\tau\phi(x\qoi)\ot x\qoo
+\tau_{12}(x\li\ot\tau\psi(x\lii))+\tau_{12}(x\qoo\ot\tau\phi(x\qi))$,
\item[(C11)] $x_{\langle1\rangle}\ot\beta(x_{\langle0\rangle})
=\tau\psi(x\qoo)\ot x\qi-\tau_{12}(x\qoo\ot\delta_A(x\qi))$,
\item[(C12)] $x_{\langle1\rangle}\ot \al(x_{\langle0\rangle})
=\tau_{12}(x\qoi\ot \tau\psi(x\qoo))-\delta_A(x\qoi)\ot x\qoo$,
\item[(C13)] $x\bi\ot\Delta_V(x\bii)+x_{\langle0\rangle}\ot s(x_{\langle1\rangle})\\
=\delta_V(x\li)\ot x\lii+p(x\qoi)\ot x\qoo+\tau_{12}(x\li\ot \delta_V(x\lii))+\tau_{12}(x\qoo\ot p(x\qi))$.
\end{enumerate}
\end{lemma}
Denote the set of all  coalgebraic extending datum of ${A}$ by $V$ of type (c1) by $\mathcal{C}^{(3)}({A},V)$.

\begin{lemma}\label{lem:33-3}
Let $({A}, \Delta_A, \delta_A)$ be a noncommutative Poisson  coalgebra and $E$  a vector space containing ${A}$ as a subspace. Suppose that there is a noncommutative Poisson  coalgebra structure $(E, \Delta_E, \delta_E)$ on $E$ such that  $p: E\to {A}$ is a noncommutative  Poisson  coalgebra homomorphism. Then there exists a noncommutative Poisson coalgebraic extending system $\Omega^{(3)}({A}, V)$ of $({A}, \Delta_A, \delta_A)$ by $V$ such that $(E, \Delta_E, \delta_E)\cong A^{p, s}\# {}^{} V$.
\end{lemma}

\begin{proof}
Let $p: E\to {A}$ and $\pi: E\to V$ be the projection map and $V=\ker({p})$.
Then the extending datum of $({A}, \Delta_A, \delta_A)$ by $V$ is defined as follows:
\begin{eqnarray*}
&&{\phi}: A\rightarrow V\ot {A},~~~~{\phi}(a)=(\pi\otimes {p})\delta_E(a),\\
&&{\psi}: V\rightarrow V\ot A,~~~~{\psi}(x)=(\pi\otimes {p})\delta_E(x),\\
&&{\rho}: A\rightarrow V\ot A,~~~~{\rho}(a)=(\pi\otimes {p})\Delta_E(a),\\
&&{\gamma}: A\rightarrow A\ot {V},~~~~{\gamma}(a)=(p\otimes \pi)\Delta_E(a),\\
&&{\alpha}: V\rightarrow A\ot {V},~~~~{\alpha}(x)=(p\otimes \pi)\Delta_E(x),\\
&&{\beta}: V\rightarrow V\ot {A},~~~~{\beta}(x)=(\pi\otimes {p})\Delta_E(x),\\
&&\delta_V: V\rightarrow V\otimes V,~~~~\delta_V(x)=(\pi\otimes \pi)\delta_E(x),\\
&&\Delta_V: V\rightarrow V\otimes V,~~~~\Delta_V(x)=(\pi\otimes \pi)\Delta_E(x),\\
&&p: A\rightarrow {V}\otimes {V},~~~~p(a)=({\pi}\otimes {\pi})\delta_E(a),\\
&&s: A\rightarrow {V}\otimes {V},~~~~s(a)=({\pi}\otimes {\pi})\Delta_E(a).
\end{eqnarray*}
One check that  $\varphi: A^{p, s}\# {}^{} V\to E$ given by $\varphi(a, x)=a+x$ for all $a\in A, x\in V$ is a
noncommutative Poisson coalgebra isomorphism.
\end{proof}

\begin{lemma}\label{lem-c1}
Let $$\Omega^{(3)}({A}, V)=(\phi, {\psi}, \rho, \gamma, \alpha, \beta, p, s, \delta_V, \Delta_V)$$
and
$${\Omega'^{(3)}}({A}, V)=(\phi', {\psi'}, \rho', \gamma', \alpha', \beta',  p', s', \delta'_V, \Delta'_V)$$
be two noncommutative Poisson   coalgebraic extending datums of $({A}, \Delta_A, \delta_A)$ by $V$. Then there exists a bijection between the set of noncommutative  Poisson    coalgebra homomorphisms $\varphi: A^{p, s}\# {}^{} V\rightarrow A^{p', s'}\# {}^{} V$ whose restriction on ${A}$ is the identity map and the set of pairs $(r, s)$, where $r: V\rightarrow {A}$ and $s: V\rightarrow V$ are two linear maps satisfying
\begin{eqnarray*}
\label{comorph11}&&p'(a)=s(a_{1p})\ot s(a_{2p}),\\
\label{comorph121}&&\phi'(a)={s}(a_{\langle-1\rangle})\ot a_{\langle0\rangle}+s(a_{1p})\ot r(a_{2p}),\\
\label{comorph13}&&\delta'_A(a)=\delta_A(a)+{r}(a_{\langle-1\rangle})\ot a_{\langle0\rangle}-a_{\langle0\rangle}\ot {r}(a_{\langle-1\rangle})+r(a_{1p})\ot r(a_{2p}),\\
\label{comorph21}&&\delta_V'({s}(x))+p'(r(x))=({s}\otimes {s})\delta_V(x),\\
\label{comorph221}&&{\psi}'({s}(x))+\phi'(r(x))=s(x\bi)\ot r(x\bii)+s(x_{\langle0\rangle})\ot x_{\langle1\rangle},\\
\label{comorph23}&&\delta'_A({r}(x))=r(x\bi)\ot r(x\bii)-x_{\langle1\rangle}\ot r(x_{\langle0\rangle})+r(x_{\langle0\rangle})\ot x_{\langle1\rangle},\\
\label{comorph31}&&s'(a)=s(a_{1s})\ot s(a_{2s}),\\
\label{comorph321}&&\rho'(a)={s}(a\loi)\ot a\loo+s(a_{1s})\ot r(a_{2s}),\\
\label{comorph322}&&\gamma'(a)=a\loo\ot {s}(a\lmi)+r(a_{1s})\ot s(a_{2s}),\\
\label{comorph33}&&\Delta'_A(a)=\Delta_A(a)+{r}(a\loi)\ot a\loo+a\loo\ot {r}(a\lmi)+r(a_{1s})\ot r(a_{2s}),\\
\label{comorph41}&&\Delta_V'({s}(x))+s'(r(x))=({s}\otimes {s})\Delta_V(x),\\
\label{comorph421}&&{\beta}'({s}(x))+\rho'(r(x))=s(x\li)\ot r(x\lii)+s(x\qoo)\ot x\qi,\\
\label{comorph422}&&{\al}'({s}(x))+\gamma'(r(x))=r(x\li)\ot s(x\lii)+x\qoi\ot s(x\qoo),\\
\label{comorph43}&&\Delta'_A({r}(x))=r(x\li)\ot r(x\lii)+x\qoi\ot r(x\qoo)+r(x\qoo)\ot x\qi.
\end{eqnarray*}
Under the above bijection the noncommutative Poisson coalgebra homomorphism $\varphi=\varphi_{r, s}: A^{p, s}\# {}^{} V\rightarrow A^{p', s'}\# {}^{} V$ to $(r, s)$ is given by $\varphi(a+x)=(a+r(x), s(x))$ for all $a\in {A}$ and $x\in V$. Moreover, $\varphi=\varphi_{r, s}$ is an isomorphism if and only if $s: V\rightarrow V$ is a linear isomorphism.
\end{lemma}
\begin{proof}
Let $\varphi: A^{p, s}\# {}^{} V\rightarrow A^{p', s'}\# {}^{} V$  be a noncommutative Poisson  coalgebra homomorphism  whose restriction on ${A}$ is the identity map. Then $\varphi$ is determined by two linear maps $r: V\rightarrow {A}$ and $s: V\rightarrow V$ such that
$\varphi(a+x)=(a+r(x), s(x))$ for all $a\in {A}$ and $x\in V$. We will prove that
$\varphi$ is a homomorphism of noncommutative  Poisson   coalgebras if and only if the above conditions hold.
First   it is easy to see that  $\delta'_E\varphi(a)=(\varphi\otimes \varphi)\delta_E(a)$ for all $a\in {A}$.
\begin{eqnarray*}
\delta'_E\varphi(a)&=&\delta'_E(a)=\delta'_A(a)+\phi'(a)-\tau\phi'(a)+p'(a),
\end{eqnarray*}
and
\begin{eqnarray*}
&&(\varphi\otimes \varphi)\delta_E(a)\\
&=&(\varphi\otimes \varphi)\left(\delta_A(a)+\phi(a)-\tau\phi(a)+p(a)\right)\\
&=&\delta_A(a)+{r}(a_{\langle-1\rangle})\ot a_{\langle0\rangle}+{s}(a_{\langle-1\rangle})\ot a_{\langle0\rangle}-a_{\langle0\rangle}\ot {r}(a_{\langle-1\rangle}) -a_{\langle0\rangle}\ot {s}(a_{\langle-1\rangle})\\
&&+r(a_{1p})\ot r(a_{2p})+r(a_{1p})\ot s(a_{2p})+s(a_{1p})\ot r(a_{2p})+s(a_{1p})\ot s(a_{2p}).
\end{eqnarray*}
Thus we obtain that $\delta'_E\varphi(a)=(\varphi\otimes \varphi)\delta_E(a)$  if and only if the conditions \eqref{comorph11}, \eqref{comorph121} and \eqref{comorph13} hold.
Then we consider that $\delta'_E\varphi(x)=(\varphi\otimes \varphi)\delta_E(x)$ for all $x\in V$.
\begin{eqnarray*}
\delta'_E\varphi(x)&=&\delta'_E({r}(x)+{s}(x))=\delta'_E({r}(x))+\delta'_E({s}(x))\\
&=&\delta'_A({r}(x))+\phi'(r(x))-\tau\phi'(r(x))+p'(r(x))+\delta'_V({s}(x))+{\psi}'({s}(x))-\tau\psi'({s}(x)),
\end{eqnarray*}
and
\begin{eqnarray*}
&&(\varphi\otimes \varphi)\delta_E(x)\\
&=&(\varphi\otimes \varphi)(\delta_V(x)+\psi(x)-\tau\psi(x))\\
&=&(\varphi\otimes \varphi)(x\bi \ot x\bii+ x_{\langle0\rangle} \ot x_{\langle1\rangle}-x_{\langle1\rangle} \ot x_{\langle0\rangle})\\
&=&r(x\bi)\ot r(x\bii)+r(x\bi)\ot s(x\bii)+s(x\bi)\ot r(x\bii)+s(x\bi)\ot s(x\bii)\\
&&-x_{\langle1\rangle}\ot r(x_{\langle0\rangle})-x_{\langle1\rangle}\ot s(x_{\langle0\rangle})+r(x_{\langle0\rangle})\ot x_{\langle1\rangle}+s(x_{\langle0\rangle})\ot x_{\langle1\rangle}.
\end{eqnarray*}
Thus we obtain that $\delta'_E\varphi(x)=(\varphi\otimes \varphi)\delta_E(x)$ if and only if the conditions  \eqref{comorph21},  \eqref{comorph221} and \eqref{comorph23}  hold.

Then   it is easy to see that  $\Delta'_E\varphi(a)=(\varphi\otimes \varphi)\Delta_E(a)$ for all $a\in {A}$.
\begin{eqnarray*}
\Delta'_E\varphi(a)&=&\Delta'_E(a)=\Delta'_A(a)+\rho'(a)+\gamma'(a)+s'(a),
\end{eqnarray*}
and
\begin{eqnarray*}
&&(\varphi\otimes \varphi)\Delta_E(a)\\
&=&(\varphi\otimes \varphi)\left(\Delta_A(a)+\rho(a)+\gamma(a)+s(a)\right)\\
&=&\Delta_A(a)+{r}(a\lmoi)\ot a\lmo+{s}(a\lmoi)\ot a\lmo+a\lmo\ot {r}(a\lmi) +a\lmo\ot {s}(a\lmi)\\
&&+r(a_{1s})\ot r(a_{2s})+r(a_{1s})\ot s(a_{2s})+s(a_{1s})\ot r(a_{2s})+s(a_{1s})\ot s(a_{2s}).
\end{eqnarray*}
Thus we obtain that $\Delta'_E\varphi(a)=(\varphi\otimes \varphi)\Delta_E(a)$  if and only if the conditions \eqref{comorph31}, \eqref{comorph321}, \eqref{comorph322} and \eqref{comorph33} hold.
Then we consider that $\Delta'_E\varphi(x)=(\varphi\otimes \varphi)\Delta_E(x)$ for all $x\in V$.
\begin{eqnarray*}
\Delta'_E\varphi(x)&=&\Delta'_E({r}(x)+{s}(x))=\Delta'_E({r}(x))+\Delta'_E({s}(x))\\
&=&\Delta'_A({r}(x))+\rho'(r(x))+\gamma'(r(x))+s(r(x))+\Delta'_V({s}(x))+{\alpha}'({s}(x))+{\beta}'({s}(x)),
\end{eqnarray*}
and
\begin{eqnarray*}
&&(\varphi\otimes \varphi)\Delta_E(x)\\
&=&(\varphi\otimes \varphi)(\Delta_V(x)+{\gamma}(x)+{\tau\gamma}(x))\\
&=&(\varphi\otimes \varphi)(x\li\ot x\lii+ x\qoi \ot x\qoo+x\qoo \ot x\qi)\\
&=&r(x\li)\ot r(x\lii)+r(x\li)\ot s(x\lii)+s(x\li)\ot r(x\lii)+s(x\li)\ot s(x\lii)\\
&&+x\qoi\ot r(x\qoo)+x\qoi\ot s(x\qoo)+r(x\qoo)\ot x\qi+s(x\qoo)\ot x\qi.
\end{eqnarray*}
Thus we obtain that $\Delta'_E\varphi(x)=(\varphi\otimes \varphi)\Delta_E(x)$ if and only if the conditions\eqref{comorph41}, \eqref{comorph421}, \eqref{comorph422} and \eqref{comorph43} hold.
By definition, we obtain that $\varphi=\varphi_{r, s}$ is an isomorphism if and only if $s: V\rightarrow V$ is a linear isomorphism.
\end{proof}

The second case is $\phi=0$, $\rho=0$ and $\gamma=0$, we obtain  the following type (c2) unified coproduct for  coalgebras.
\begin{lemma}\label{cor02}
Let $({A}, \Delta_A, \delta_A)$ be a noncommutative Poisson coalgebra and $V$ a vector space.
An  extending datum  of $({A}, \Delta_A, \delta_A)$ by $V$ of type (c2)  is  $\Omega^{(4)}({A}, V)=(\psi, \al, \beta, {q},
t, \Delta_V, \delta_V)$ with  linear maps
\begin{eqnarray*}
&&\psi: V  \to V\otimes A, \quad \delta_{V}: V \to V\otimes V, \quad q: V \to A\otimes A, \quad \al: V \to A \otimes V,\\
&& \beta: V \to V\otimes A, \quad \Delta_{V}: V \to V\otimes V, \quad  t: V\to A\ot A.
\end{eqnarray*}
 Denote by $A^{}\# {}^{q, t} V$ the vector space $E={A}\oplus V$ with the comultiplication
$\Delta_E: E\rightarrow E\otimes E,  \delta_E: E\to E\ot E$ given by
$$\delta_{E}(a)=\delta_{A}(a), \quad \delta_{E}(x)=(\delta_{V}+\psi-\tau\psi+q)(x), $$
$$\Delta_{E}(a)=\Delta_{A}(a), \quad \Delta_{E}(x)=(\Delta_{V}+\al+\beta+t)(x), $$
that is
$$\delta_{E}(a)= a\bi \ot a\bii,~~\delta_{E}(x)= x\bi \ot x\bii+ x_{\langle0\rangle} \ot x_{\langle1\rangle}-x_{\langle1\rangle} \ot x_{\langle0\rangle}+x_{1q}\ot x_{2q},$$
$$\Delta_{E}(a)= a\li \ot a\lii,~~\Delta_{E}(x)= x\li \ot x\lii+ x\qoi \ot x\qoo+x\qoo \ot x\qi+x_{1t}\ot x_{2t}.$$
Then $A^{}\# {}^{q, t} V$  is a noncommutative Poisson  coalgebra with the comultiplication given above if and only if the following compatibility conditions hold:

\begin{enumerate}
\item[(D0)] $\bigl(\al, \,\beta, \, t)$ is an algebra extending system of the associative coalgebra
$A$ trough $V$ and $\bigl(\psi, \,
q \bigl)$ is a Lie extending system of the
Lie coalgebra $A$ trough $V$,
\item[(D1)] $x\bi\ot \beta(x\bii)
=\delta_V(x\qoo)\ot x\qi+\tau_{12}(x\li\ot \psi(x\lii))$,
\item[(D2)] $x\bi\ot\al(x\bii)
=\psi(x\li)\ot x\lii+\tau_{12}(x\qoi\ot\delta_V(x\qoo))$,
\item[(D3)] $x\bi\ot t(x\bii)+x_{\langle0\rangle}\ot\Delta_A(x_{\langle1\rangle})
=\psi(x\qoo)\ot x\qi+\tau_{12}(x\qoi\ot\psi(x\qoo))$,
\item[(D4)] $x_{\langle1\rangle}\ot\Delta_V(x_{\langle0\rangle})
=\tau\psi(x\li)\ot x\lii+\tau_{12}(x\li\ot\tau\psi(x\lii))$,
\item[(D5)] $x_{\langle1\rangle}\ot\beta(x_{\langle0\rangle})
=\tau\psi(x\qoo)\ot x\qi-\tau_{12}(x\qoo\ot\delta_A(x\qi))-\tau_{12}(x\li\ot q(x\lii))$,
\item[(D6)] $x_{\langle1\rangle}\ot \al(x_{\langle0\rangle})
=\tau_{12}(x\qoi\ot \tau\psi(x\qoo))-q(x\li)\ot x\lii-\delta_A(x\qoi)\ot x\qoo$,
\item[(D7)] $x\bi\ot\Delta_V(x\bii)
=\delta_V(x\li)\ot x\lii+\tau_{12}(x\li\ot \delta_V(x\lii))$,
\item[(D8)] $x_{1q}\ot\Delta_A(x_{2q})-x_{\langle1\rangle}\ot t(x_{\langle0\rangle})\\
=q(x\qoo)\ot x\qi+\delta_A(x_{1t})\ot x_{2t}+\tau_{12}(x\qoi\ot q(x\qoo))+\tau_{12}(x_{1t}\ot \delta_A(x_{2t}))$.
\end{enumerate}
\end{lemma}
Note that in this case $(V, \Delta_V, \delta_V)$ is a noncommutative Poisson  coalgebra.

Denote the set of all noncommutative Poisson  coalgebraic extending datum of ${A}$ by $V$ of type (c2) by $\mathcal{C}^{(4)}({A}, V)$.

Similar to the noncommutative Poisson   algebra case,  one  show that any noncommutative  Poisson   coalgebra structure on $E$ containing ${A}$ as a  noncommutative Poisson   subcoalgebra is isomorphic to such a unified coproduct.
\begin{lemma}\label{lem:33-4}
Let $({A}, \Delta_A, \delta_A)$ be a noncommutative Poisson  coalgebra and $E$  a vector space containing ${A}$ as a subspace. Suppose that there is a noncommutative Poisson   coalgebra structure $(E, \Delta_E, \delta_E)$ on $E$ such that  $({A}, \Delta_A, \delta_A)$ is a noncommutative Poisson  subcoalgebra of $E$. Then there exists a noncommutative Poisson coalgebraic extending system $\Omega^{(2)}({A}, V)$ of $({A}, \Delta_A, \delta_A)$ by $V$ such that $(E, \Delta_E, \delta_E)\cong A^{}\# {}^{q, t} V$.
\end{lemma}

\begin{proof}
Let $p: E\to {A}$ and $\pi: E\to V$ be the projection map and $V=ker({p})$.
Then the extending datum of $({A}, \Delta_A, \delta_A)$ by $V$ is defined as follows:
\begin{eqnarray*}
&&{\psi}: V\rightarrow V\ot {A},~~~~{\phi}(x)=(\pi\otimes {p})\delta_E(x),\\
&&\delta_V: V\rightarrow V\otimes V,~~~~\delta_V(x)=(\pi\otimes \pi)\delta_E(x),\\
&&q: V\rightarrow {A}\otimes {A},~~~~q(x)=({p}\otimes {p})\delta_E(x),\\
&&{\al}: V\rightarrow A\ot {V},~~~~{\gamma}(x)=(p \otimes \pi)\Delta_E(x),\\
&&{\beta}: V\rightarrow V\ot {A},~~~~{\gamma}(x)=(\pi\otimes {p})\Delta_E(x),\\
&&\Delta_V: V\rightarrow V\otimes V,~~~~\Delta_V(x)=(\pi\otimes \pi)\Delta_E(x),\\
&&t: V\rightarrow {A}\otimes {A},~~~~t(x)=({p}\otimes {p})\Delta_E(x).
\end{eqnarray*}
One check that  $\varphi: A^{}\# {}^{q, t} V\to E$ given by $\varphi(a, x)=a+x$ for all $a\in A, x\in V$ is a noncommutative Poisson   coalgebra isomorphism.
\end{proof}

\begin{lemma}\label{lem-c2}
Let $\Omega^{(4)}({A}, V)=(\psi, \al, \beta, {q}, t, \delta_V, \Delta_V)$ and ${\Omega'^{(4)}}({A}, V)=(\psi', \al', \beta', {q'}, t', \delta'_V, \Delta'_V)$ be two noncommutative  Poisson   coalgebraic extending datums of $({A}, \Delta_A, \delta_A)$ by $V$. Then there exists a bijection between the set of noncommutative  Poisson    coalgebra homomorphisms $\varphi: A \# {}^{q, t} V\rightarrow A \# {}^{q', t'} V$ whose restriction on ${A}$ is the identity map and the set of pairs $(r, s)$, where $r: V\rightarrow {A}$ and $s:V\rightarrow V$ are two linear maps satisfying
\begin{eqnarray*}
\label{comorph2}&&{\psi}'({s}(x))=s(x\bi)\ot r(x\bii)+s(x_{\langle0\rangle})\ot x_{\langle1\rangle},\\
\label{comorph3}&&\delta_V'({s}(x))=({s}\otimes {s})\delta_V(x),\\
\label{comorph4}&&\delta'_A({r}(x))+{q'}({s}(x))=r(x\bi)\ot r(x\bii)-x_{\langle1\rangle}\ot r(x_{\langle0\rangle})+r(x_{\langle0\rangle})\ot x_{\langle1\rangle}+{q}(x),\\
\label{comorph5}&&{\al}'({s}(x))=r(x\li)\ot s(x\lii)+x\qoi\ot s(x\qoo),\\
\label{comorph6}&&{\beta}'({s}(x))=s(x\li)\ot r(x\lii)+s(x\qoo)\ot x\qi,\\
\label{comorph7}&&\Delta_V'({s}(x))=({s}\otimes {s})\Delta_V(x),\\
\label{comorph8}&&\Delta'_A({r}(x))+{t'}({s}(x))=r(x\li)\ot r(x\lii)+x\qoi\ot r(x\qoo)+r(x\qoo)\ot x\qi+{t}(x).
\end{eqnarray*}
Under the above bijection the noncommutative   Poisson  coalgebra homomorphism $\varphi=\varphi_{r, s}: A^{ }\# {}^{q, t} V\rightarrow A^{ }\# {}^{q', t'} V$ to $(r,s)$ is given by $\varphi(a, x)=(a+r(x), s(x))$ for all $a\in {A}$ and $x\in V$. Moreover, $\varphi=\varphi_{r, s}$ is an isomorphism if and only if $s: V\rightarrow V$ is a linear isomorphism.
\end{lemma}
\begin{proof} The proof is similar as the proof of Lemma \ref{lem-c1}.
Let $\varphi: A^{ }\# {}^{q, t} V\rightarrow A^{}\# {}^{q', t'} V$  be a noncommutative Poisson   coalgebra homomorphism  whose restriction on ${A}$ is the identity map.
First  it is easy to see that  $\delta'_E\varphi(a)=(\varphi\otimes \varphi)\delta_E(a)$ for all $a\in {A}$.
Then we consider that $\delta'_E\varphi(x)=(\varphi\otimes \varphi)\delta_E(x)$ for all $x\in V$.
\begin{eqnarray*}
\delta'_E\varphi(x)&=&\delta'_E({r}(x), {s}(x))=\delta'_E({r}(x))+\delta'_E({s}(x))\\
&=&\delta'_A({r}(x))+\delta'_V({s}(x))+{\psi}'({s}(x))-{\tau\psi}'({s}(x))+{q}'({s}(x)),
\end{eqnarray*}
and
\begin{eqnarray*}
&&(\varphi\otimes \varphi)\delta_E(x)\\
&=&(\varphi\otimes \varphi)(\delta_V(x)+{\psi}(x)-{\tau\psi}(x)+{q}(x))\\
&=&(\varphi\otimes \varphi)(x\bi \ot x\bii+ x_{\langle0\rangle} \ot x_{\langle1\rangle}-x_{\langle1\rangle} \ot x_{\langle0\rangle}+q(x))\\
&=&r(x\bi)\ot r(x\bii)+r(x\bi)\ot s(x\bii)+s(x\bi)\ot r(x\bii)+s(x\bi)\ot s(x\bii)\\
&&-x_{\langle1\rangle}\ot r(x_{\langle0\rangle})-x_{\langle1\rangle}\ot s(x_{\langle0\rangle})+r(x_{\langle0\rangle})\ot x_{\langle1\rangle}+s(x_{\langle0\rangle})\ot x_{\langle1\rangle}+{q}(x).
\end{eqnarray*}
Thus we obtain that $\delta'_E\varphi(x)=(\varphi\otimes \varphi)\delta_E(x)$ if and only if the conditions \eqref{comorph2},  \eqref{comorph3} and \eqref{comorph4} hold.

First  it is easy to see that  $\Delta'_E\varphi(a)=(\varphi\otimes \varphi)\Delta_E(a)$ for all $a\in {A}$.
Then we consider that $\Delta'_E\varphi(x)=(\varphi\otimes \varphi)\Delta_E(x)$ for all $x\in V$.
\begin{eqnarray*}
\Delta'_E\varphi(x)&=&\Delta'_E({r}(x), {s}(x))=\Delta'_E({r}(x))+\Delta'_E({s}(x))\\
&=&\Delta'_A({r}(x))+\Delta'_V({s}(x))+{\al}'({s}(x))+{\beta}'({s}(x))+{t}'({s}(x)),
\end{eqnarray*}
and
\begin{eqnarray*}
&&(\varphi\otimes \varphi)\Delta_E(x)\\
&=&(\varphi\otimes \varphi)(\Delta_V(x)+{\al}(x)+{\beta}(x)+{t}(x))\\
&=&(\varphi\otimes \varphi)(x\li \ot x\lii+ x\qoi \ot x\qoo+x\qoo \ot x\qi+{t}(x))\\
&=&r(x\li)\ot r(x\lii)+r(x\li)\ot s(x\lii)+s(x\li)\ot r(x\lii)+s(x\li)\ot s(x\lii)\\
&&+x\qoi\ot r(x\qoo)+x\qoi\ot s(x\qoo)+r(x\qoo)\ot x\qi+s(x\qoo)\ot x\qi+{t}(x).
\end{eqnarray*}
Thus we obtain that $\Delta'_E\varphi(x)=(\varphi\otimes \varphi)\Delta_E(x)$ if and only if the conditions  \eqref{comorph5}, \eqref{comorph6},  \eqref{comorph7} and \eqref{comorph8} hold.
By definition, we obtain that $\varphi=\varphi_{r, s}$ is an isomorphism if and only if $s: V\rightarrow V$ is a linear isomorphism.
\end{proof}

Let $({A}, \Delta_A, \delta_A)$ be a noncommutative Poisson  coalgebra and $V$  a vector space. Two noncommutative Poisson   coalgebraic extending systems $\Omega^{(i)}({A}, V)$ and ${\Omega'^{(i)}}({A}, V)$  are called equivalent if $\varphi_{r, s}$ is an isomorphism.  We denote it by $\Omega^{(i)}({A}, V)\equiv{\Omega'^{(i)}}({A}, V)$.
From the above lemmas, we obtain the following result.
\begin{theorem}\label{thm3-2}
Let $({A}, \Delta_A, \delta_A)$ be a noncommutative Poisson   coalgebra, $E$ be a vector space containing ${A}$ as a subspace and
$V$ be a ${A}$-complement in $E$. Denote $\mathcal{HC}(V, {A}): =\mathcal{C}^{(3)}({A}, V)\sqcup\mathcal{C}^{(4)}({A}, V) /\equiv$. Then the map
\begin{eqnarray*}
&&\Psi: \mathcal{HC}_{{A}}^2(V, {A})\rightarrow CExtd(E, {A}),\\
&&\overline{\Omega^{(3)}({A}, V)}\mapsto A^{p, s}\# {}^{} V,
 \quad \overline{\Omega^{(4)}({A}, V)}\mapsto A^{}\# {}^{q, t} V
\end{eqnarray*}
is bijective, where $\overline{\Omega^{(i)}({A}, V)}$ is the equivalence class of $\Omega^{(i)}({A}, V)$ under $\equiv$.
\end{theorem}

\subsection{Extending structures for noncommutative Poisson bialgebras}
Let $(A,\cdot, [,], \Delta_A, \delta_A)$ be a noncommutative Poisson bialgebra. From (CBB1) and (CBB2) we have the following two cases.

The first case is that we assume $q=0, t=0$ and $\ppr, \ppl, \trr$ to be trivial. Then by the above Theorem \ref{main2}, we obtain the following result.

\begin{theorem}\label{thm-41}
Let $(A,\cdot, [ , ], \Delta_A,\delta_A)$ be a noncommutative Poisson bialgebra and $V$ a vector space.
An extending datum of ${A}$ by $V$ of type (I) is
$$\Omega^{(I)}({A},V)=(\rightarrow, \leftarrow, \trl, \phi, \psi, \rho, \gamma, \alpha, \beta, p, s, \theta, \nu,  \cdot_V, [ , ]_V, \Delta_V, \delta_V)$$
 consisting of  linear maps
\begin{eqnarray*}
&&\trl: V\otimes {A}\rightarrow V, ~~~\theta:  A\otimes A \rightarrow {V}, ~~~[ , ]_V:V\otimes V \rightarrow V,\\
&&\phi :A \to V\otimes A, \quad{\psi}: V\to  V\otimes A, ~~~p: A\rightarrow {V}\otimes {V},\\
&&\delta_V: V\rightarrow V\otimes V, ~~~ \rightarrow: A\otimes {V}\rightarrow V, ~~~ \leftarrow: V\otimes {A}\rightarrow V,\\
&&\nu:  A\otimes A \rightarrow {V}, ~~~\cdot_V: V \otimes V \rightarrow V, ~~~\rho :A \to V\otimes A, \\
&&\gamma: A\to  A\otimes V, ~~~\alpha: V\to A\ot V, ~~~\beta: V\to V\ot A,\\
&&{s}: A\rightarrow {V}\otimes {V}, ~~~\Delta_V: V\rightarrow V\otimes V.
\end{eqnarray*}
Then the unified product $A^{p,s}_{}\# {}^{}_{\theta,\nu}\, V$ with product
\begin{align}
[(a, x), (b, y)]&=\big([a, b], \, [x, y]+x\trl b-y\trl a+\theta(a, b)\big),\\
(a, x)\cdot (b, y)&=\big(ab, \, xy+x\leftarrow b+a\rightarrow y+\nu(a, b)\big),
\end{align}
and coproduct
\begin{align}
\delta_{E}(a)&=\delta_{A}(a)+\phi(a)-\tau\phi(a)+p(a),\quad \delta_{E}(x)=\delta_{V}(x)+\psi(x)-\tau\psi(x), \\
\Delta_{E}(a)&=\Delta_{A}(a)+\rho(a)+\gamma(a)+s(a),\quad \Delta_{E}(x)=\Delta_{V}(x)+\alpha(x)+\beta(x)
\end{align}
forms a noncommutative Poisson bialgebra if and only if $A_{}\# {}_{\theta,\nu} V$ forms a noncommutative Poisson algebra, $A^{p,s}\# {}^{} \, V$ forms a noncommutative  Poisson coalgebra and the following conditions are satisfied:
\begin{enumerate}
\item[(E0)] $\bigl(\leftarrow, \, \rightarrow, \, \nu, \, \rho, \, \gamma, \, \alpha, \, \beta, \, s)$ is an algebra extending system of the associative algebra and coassociative coalgebra
$A$ trough $V$ and $\bigl(\trl, \, \theta, \phi, \psi, p \bigl)$ is a Lie extending system of the
Lie algebra and Lie coalgebra $A$ trough $V$,
\item[(E1)] $\phi(a b)+\psi(\nu(a, b))= (a_{\langle-1\rangle}\leftarrow b)\ot  a_{\langle0\rangle}+( a\rightarrow b_{\langle-1\rangle})\ot  b_{\langle0\rangle}+b\loi\ot[a, b\loo]\\
    +a\lmi\ot[b, a\loo]+\nu(a\bi, b)\ot a\bii+\nu(a, b\bi)\ot b\bii$,
\item[(E2)] $\tau\phi(ab)+\tau\psi(\nu(a, b))= a_{\langle0\rangle}b\ot  a_{\langle-1\rangle}+a b_{\langle0\rangle}\ot  b_{\langle-1\rangle}
+b\loo\ot(b\lmi\trl a)+a\loo\ot(a\loi\trl b)\\
-b\li\ot\theta(a,b\lii)-a\lii\ot\theta(b,a\li)$,
\item[(E3)] $\psi(xy)= x_{\langle0\rangle} y\ot x_{\langle1\rangle}+x  y_{\langle0\rangle}\ot  y_{\langle1\rangle}$,
\item[(E4)] $\tau\psi(xy)=-y\qoi\ot[x, y\qoo]-x\qi\ot[y, x\qoo]$,
\item[(E5)] $\delta_V(x\leftarrow b)=(x\bi\leftarrow b)\ot x\bii-(x\leftarrow b_{\langle0\rangle})\ot b_{\langle-1\rangle}+b\loi\ot(x\trl b\loo)\\
-x\lii\ot(x\li\trl b)-\nu(x_{\langle1\rangle}, b)\ot x_{\langle0\rangle}+xb_{1p}\ot b_{2p}+b_{1s}\ot[x, b_{2s}]
+x\qoo\ot\theta(b, x\qoi)$,
\item[(E6)] $\delta_V(a\rightarrow y)=(a\rightarrow y\bi)\ot y\bii-(a_{\langle0\rangle}\rightarrow y)\ot a_{\langle-1\rangle}+a\lmi\ot(y\trl a\loo)\\
-y\li\ot(y\lii\trl a)-\nu(a, y_{\langle1\rangle})\ot y_{\langle0\rangle}+a_{1p}y\ot a_{2p}+a_{2s}\ot[y, a_{1s}]
+y\qoo\ot\theta(a, y\qi)$,
\item[(E7)] $\psi(x\leftarrow b)=(x_{\langle0\rangle}\leftarrow b)\ot x_{\langle1\rangle}+(x\leftarrow b\bi)\ot b\bii
+x b_{\langle-1\rangle}\ot b_{\langle0\rangle}+x\qoo\ot[b, x\qoi]$,
\item[(E8)] $\tau\psi(x\leftarrow b)=x_{\langle1\rangle} b\ot x_{\langle0\rangle}+x\qi\ot(x\qoo\trl b)
-b\loo\ot[x, b\lmi]-b\li\ot(x\trl b\lii)$,
\item[(E9)] $\psi(a\rightarrow y)=a_{\langle-1\rangle}y\ot a_{\langle0\rangle}+( a\rightarrow y_{\langle0\rangle})\ot y_{\langle1\rangle}+y\qoo\ot[a, y\qi]-(a\bi\rightarrow y)\ot a\bii$,
\item[(E10)] $\tau\psi(a\rightarrow y)=ay_{\langle1\rangle}\ot y_{\langle0\rangle}+y\qoi\ot(y\qoo\trl a)
-a\loo\ot[y, a\loi]-a\lii\ot(y\trl a\li)$,
\item[(E11)] $\rho([a, b])+\beta(\theta(a, b))= (a_{\langle-1\rangle}\leftarrow b)\ot  a_{\langle0\rangle}-( b_{(-1)}\trl a)\ot  b_{(0)}
+b\loi\ot[a, b\loo]\\
-a_{\langle-1\rangle}\ot ba_{\langle0\rangle}+\theta(a, b\li)\ot b\lii+\nu(a\bi, b)\ot a\bii$,
\item[(E12)] $\beta([x, y])= [x, y\qoo]\ot y\qi+x_{\langle0\rangle}y\ot x_{\langle1\rangle}$,
\item[(E13)] $\gamma([a ,b])+\alpha(\theta(a,b))= a_{\langle0\rangle}\ot(b\rightarrow a_{\langle-1\rangle})- b_{(0)}\ot( b_{(1)}\trl a)
+[a,b\loo]\ot b\lmi\\
-a\bi\ot \nu(b, a\bii)+b\li\ot\theta(a,b\lii)+\nu(a\bi, b)\ot a\bii +(a_{\langle-1\rangle}\leftarrow b)\ot a_{\langle0\rangle}$,
\item[(E14)] $\al([x,y])= y\qoi\ot[x,y\qoo]+x_{\langle1\rangle}\ot yx_{\langle0\rangle}$,
\item[(E15)] $\Delta_V(x\trl b)=(x\trl b\loo)\ot b\lmi+b\loi\ot(x\trl b\loo)+(x\bi\leftarrow b)\ot x\bii\\
    -x\bi\ot(b\rightarrow x\bii)+[x, b_{1s}]\ot b_{2s}+b_{1s}\ot[x, b_{2s}]-\nu(x_{\langle1\rangle}, b)\ot x_{\langle0\rangle}
    -x_{\langle0\rangle}\ot \nu(b, x_{\langle1\rangle})$,
\item[(E16)] $\Delta_V(y\trl a)=(y\li\trl a)\ot y\lii+y\li\ot(y\lii\trl a)+(a_{\langle0\rangle}\rightarrow y)\ot a_{\langle-1\rangle}\\
    +a_{\langle-1\rangle}\ot(y\leftarrow a_{\langle0\rangle})
    -\theta(a, y\qoi)\ot y\qoo-y\qoo\ot\theta(a, y\qi)-a_{1p}y\ot a_{2p}+a_{1p}\ot ya_{2p}$,
\item[(E17)] $\beta(x\trl b)=(x\trl b\li)\ot b\lii+[x, b\loi]\ot b\loo
-x_{\langle0\rangle}\ot bx_{\langle1\rangle}+(x_{\langle0\rangle}\leftarrow b)\ot x_{\langle1\rangle}$,
\item[(E18)] $\beta(y\trl a)=(y\qoo\trl a)\ot y\qi-y\qoo\ot[a, y\qi]-(a\bi\rightarrow y)\ot a\bii
    -a_{\langle-1\rangle}y\ot a_{\langle0\rangle}$,
\item[(E19)] $\al(x\trl b)=b\li\ot(x\trl b\lii)+ b\loo\ot[x,b\lmi]
    +x_{\langle1\rangle}\ot (b\rightarrow x_{\langle0\rangle})-x_{\langle1\rangle}b\ot x_{\langle0\rangle}$,
\item[(E20)] $\al(y\trl a)= y\qoi\ot (y\qoo\trl a)-[a,y\qoi]\ot y\qoo-a_{\langle0\rangle}\ot ya_{\langle-1\rangle}+a\bi\ot(y\leftarrow a\bii)$,
\item[(E21)] $\delta_{V}(xy)=x\bi y\ot x\bii-(x_{\langle1\rangle}\rightarrow y)\ot x_{\langle0\rangle}
+xy\bi\ot y\bii-(x\leftarrow y_{\langle1\rangle})\ot y_{\langle0\rangle}\\
+y\li\ot[x,y\lii]+y\qoo\ot(x\trl y\qi)+x\lii\ot[y,x\li]+x\qoo\ot(y\trl x\qoi),$
\item[(E22)] $\Delta_{V}([x,y])=[x,y\li]\ot y\lii+(x\trl y\qoi)\ot y\qoo+y\li\ot[x,y\lii]
+y\qoo\ot(x\trl y\qi)\\
+x\bi y\ot x\bii-(x_{\langle1\rangle}\rightarrow y)\ot x_{\langle0\rangle}-x\bi\ot yx\bii-x_{\langle0\rangle}\ot(y\leftarrow x_{\langle1\rangle}).$
\end{enumerate}
Conversely, any noncommutative Poisson bialgebra structure on $E$ with the canonical projection map $p: E\to A$ both a noncommutative Poisson algebra homomorphism and a noncommutative Poisson  coalgebra homomorphism is of this form.
\end{theorem}
Note that in this case, $(V, \cdot, [,], \Delta_V, \delta_V)$ is a  braided noncommutative Poisson bialgebra. Although $(A, \cdot, [,], \Delta_A, \delta_A)$ is not a noncommutative Poisson sub-bialgebra of $E=A^{p, s}_{}\# {}^{}_{\theta, \nu}\, V$, but it is indeed a noncommutative Poisson bialgebra and a subspace $E$.
Denote the set of all noncommutative Poisson bialgebraic extending datum of type (I) by $\mathcal{IB}^{(I)}({A},V)$.

The second case is that we assume $p=0, s=0, \theta=0, \nu=0$ and $\phi, \rho, \gamma$ to be trivial. Then by the above Theorem \ref{main2}, we obtain the following result.

\begin{theorem}\label{thm-42}
Let $A$ be a Poisson bialgebra and $V$ a vector space.
An extending datum of ${A}$ by $V$ of type (II) is  $\Omega^{(II)}({A}, V)=(\ppr, \ppl, \rightarrow, \leftarrow, \trr, \trl, \sigma, \omega, \psi, \al, \beta, q, t,  \cdot_V, [,]_V, \delta_V, \Delta_V)$ consisting of  linear maps
\begin{eqnarray*}
&&\trl: V\otimes {A}\rightarrow {V},~~~~\trr: A\otimes {V}\rightarrow V,~~~~\sigma:  V\otimes V \rightarrow {A},\\
&&[,]_V: V\otimes V \rightarrow V,~~~\psi: V\to  V\otimes A,~~~~{q}: V\rightarrow {A}\otimes {A},\\
&&\delta_V: V\rightarrow V\otimes V,~~~~\ppl: A\otimes {V}\rightarrow {A},~~~~\ppr: V\otimes {A}\rightarrow A,\\
&&\rightarrow: A\ot V\to V, ~~~\leftarrow: V\ot A\to V, ~~~\omega: V\otimes V \rightarrow {A},\\
&&\cdot_V: V\otimes V \rightarrow V,~~~~{\alpha}: V\to  A\otimes V, ~~~\beta: V\to V\ot A,\\
&&{t}: V\rightarrow {A}\otimes {A},~~~~\Delta_V: V\rightarrow V\otimes V.
\end{eqnarray*}
Then the unified product $A^{}_{\sigma, \omega}\# {}^{q, t}_{}\, V$ with product
\begin{align}
[(a, x), (b, y)]_E&=\big([a, b]+x\trr b-y\trr a+\sigma(x, y), \, [x, y]+x\trl b-y\trl a\big),\\
(a, x)\cdot_E (b, y)&=\big(ab+x\ppr b+a\ppl y+\omega(x, y), \, xy+x\leftarrow b+a\rightarrow y\big),
\end{align}
and coproduct
\begin{align}
\delta_{E}(a)&= \delta_A(a),~~~\delta_{E}(x)= \delta_{V}(x)+\psi(x)-\tau\psi(x)+q(x),\\
\Delta_{E}(a)&= \Delta_A(a),~~~\Delta_{E}(x)= \Delta_{V}(x)+\al(x)+\beta(x)+t(x),
\end{align}
forms a noncommutative Poisson bialgebra if and only if $A_{\sigma, \omega}\# {}_{} V$ forms a noncommutative Poisson algebra, $A^{}\# {}^{q, t}V$ forms a noncommutative Poisson coalgebra and the following conditions are satisfied:
\begin{enumerate}
\item[(F0)] $\bigl(\ppr, \, \ppl, \, \rightarrow, \, \leftarrow, \, \omega, \, \al, \, \beta, \, t)$ is an algebra extending system of the associative algebra and coassociative coalgebra
$A$ trough $V$ and $\bigl(\trr, \, \trl, \, \sigma, \, \psi, \, q \bigl)$ is a Lie extending system of the
Lie algebra and Lie coalgebra $A$ trough $V$,
\item[(F1)] $\psi(xy)= x_{\langle0\rangle} y\ot x_{\langle1\rangle}+x  y_{\langle0\rangle}\ot  y_{\langle1\rangle}
+ y\qoo\ot(x\trr y\qi)+x\qoo\ot(y\trr x\qoi)\\
+(x_{1q}\rightarrow y)\ot x_{2q}+(x\leftarrow y_{1q})\ot y_{2q}+y\li\ot\sigma(x, y\lii)+x\lii\ot\sigma(y, x\li)$,
\item[(F2)] $\tau\psi(xy)=(x_{\langle1\rangle}\ppl y)\ot x_{\langle0\rangle}+(x\ppr y_{\langle1\rangle})\ot y_{\langle0\rangle}
-y\qoi\ot[x, y\qoo]\\
-x\qi\ot[y, x\qoo]-\omega(x\bi, y)\ot x\bii-\omega(x, y\bi)\ot y\bii-y_{1t}\ot(x\trl y_{2t})-x_{2t}\ot(y\trl x_{1t})$,
\item[(F3)]  $\delta_A(x\ppr b)+q(x\leftarrow b)=( x_{\langle0\rangle}\ppr b)\ot  x_{\langle1\rangle}+(x\ppr b\bi)\ot b\bii-x\qi\ot(x\qoo\trr b)\\
+b\li\ot(x\trr b\lii)+x_{1q}b\ot x_{2q}+x_{2t}\ot[b, x_{1t}]$,
\item[(F4)]  $\delta_A(a\ppl y)+q(a\rightarrow y)=(a\ppl y_{\langle0\rangle})\ot y_{\langle1\rangle}
+(a\bi\ppl y)\ot a\bii-y\qoi\ot(y\qoo\trr a)+a\lii\ot (y\trr a\li)\\
+ay_{1q}\ot y_{2q}+y_{1t}\ot [a, y_{2t}]$,
\item[(F5)] $\delta_V(x\leftarrow b)=(x\bi\leftarrow b)\ot x\bii
-x\lii\ot(x\li\trl b)+b_{1s}\ot[x, b_{2s}]$,
\item[(F6)] $\delta_V(a\rightarrow y)=(a\rightarrow y\bi)\ot y\bii
-y\li\ot(y\lii\trl a)+a_{2s}\ot[y, a_{1s}]$,
\item[(F7)] $\psi(x\leftarrow b)=(x_{\langle0\rangle}\leftarrow b)\ot x_{\langle1\rangle}
+(x\leftarrow b\bi)\ot b\bii-x\lii\ot(x\li\trr b)+x\qoo\ot[b, x\qoi]$,
\item[(F8)] $\tau\psi(x\leftarrow b)=x_{\langle1\rangle} b\ot x_{\langle0\rangle}+x\qi\ot(x\qoo\trl b)
-(x\bi\ppr b)\ot x\bii-b\li\ot(x\trl b\lii)$,
\item[(F9)] $\psi(a\rightarrow y)=( a\rightarrow y_{\langle0\rangle})\ot y_{\langle1\rangle}
+y\qoo\ot[a, y\qi]-(a\bi\rightarrow y)\ot a\bii-y\li\ot(y\lii\trr a)$,
\item[(F10)] $\tau\psi(a\rightarrow y)=ay_{\langle1\rangle}\ot y_{\langle0\rangle}+y\qoi\ot(y\qoo\trl a)
-(a\ppl y\bi)\ot y\bii-a\lii\ot(y\trl a\li)$,
\item[(F11)] $\beta([x, y])= [x, y\qoo]\ot y\qi+y\qoo\ot (x\trr y\qi)
-x_{\langle0\rangle}\ot(y\ppr x_{\langle1\rangle})\\
+x_{\langle0\rangle}y\ot x_{\langle1\rangle}+(x\trl y_{1t})\ot y_{2t}+y\li\ot\sigma(x, y\lii)+(x_{1q}\rightarrow y)\ot x_{2q}-x\bi\ot\omega(y, x\bii)$,
\item[(F12)] $\al([x,y])= y\qoi\ot[x,y\qoo]+ (x\trr y\qoi)\ot y\qoo
-(x_{\langle1\rangle}\ppl y)\ot x_{\langle0\rangle}\\
+x_{\langle1\rangle}\ot yx_{\langle0\rangle}+y_{1t}\ot(x\trl y_{2t})+\sigma(x,y\li)\ot y\lii-x_{1q}\ot(y\leftarrow x_{2q}) +\omega(x\bi, y)\ot x\bii$,
\item[(F13)]  $\Delta_A(x\trr b)+t(x\trl b)=(x\trr b\li)\ot b\lii+b\li\ot(x\trr b\lii)+( x_{\langle0\rangle}\ppr b)\ot  x_{\langle1\rangle}\\
    +x_{\langle1\rangle}\ot (b\ppl x_{\langle0\rangle})+x_{1q}b\ot x_{2q}-x_{1q}\ot bx_{2q}$,
\item[(F14)]  $\Delta_A(y\trr a)+t(y\trl a)=-(a\bi\ppl y)\ot a\bii+a\bi\ot(y\ppr a\bii)+( y\qoo\trr a)\ot  y\qi\\
    +y\qoi\ot (y\qoo\trr a)-[a, y_{1t}]\ot y_{2t}-y_{1t}\ot[a, y_{2t}]$,
\item[(F15)] $\Delta_V(x\trl b)=(x\bi\leftarrow b)\ot x\bii-x\bi\ot(b\rightarrow x\bii)$,
\item[(F16)] $\Delta_V(y\trl a)=(y\li\trl a)\ot y\lii+y\li\ot(y\lii\trl a)$,
\item[(F17)] $\beta(x\trl b)=(x\trl b\li)\ot b\lii-x_{\langle0\rangle}\ot bx_{\langle1\rangle}
    +(x_{\langle0\rangle}\leftarrow b)\ot x_{\langle1\rangle}-x\bi\ot(b\ppl x\bii)$,
\item[(F18)] $\beta(y\trl a)=(y\qoo\trl a)\ot y\qi-y\qoo\ot[a, y\qi]-(a\bi\rightarrow y)\ot a\bii
    +y\li\ot(y\lii\trr a)$,
\item[(F19)] $\al(x\trl b)=b\li\ot(x\trl b\lii)+x_{\langle1\rangle}\ot (b\rightarrow x_{\langle0\rangle})
    +(x\bi\ppr b)\ot x\bii-x_{\langle1\rangle}b\ot x_{\langle0\rangle}$,
\item[(F20)] $\al(y\trl a)= y\qoi\ot (y\qoo\trl a)-[a,y\qoi]\ot y\qoo+(y\li\trr a)\ot y\lii
    +a\bi\ot(y\leftarrow a\bii)$,
\item[(F21)] $\delta_{V}(xy)=x\bi y\ot x\bii-(x_{\langle1\rangle}\rightarrow y)\ot x_{\langle0\rangle}
+xy\bi\ot y\bii-(x\leftarrow y_{\langle1\rangle})\ot y_{\langle0\rangle}\\
+y\li\ot[x, y\lii]+y\qoo\ot(x\trl y\qi)+x\lii\ot[y, x\li]+x\qoo\ot(y\trl x\qoi),$
\item[(F22)] $\Delta_{V}([x, y])=[x, y\li]\ot y\lii+(x\trl y\qoi)\ot y\qoo+y\li\ot[x, y\lii]
+y\qoo\ot(x\trl y\qi)\\
+x\bi y\ot x\bii-(x_{\langle1\rangle}\rightarrow y)\ot x_{\langle0\rangle}-x\bi\ot yx\bii-x_{\langle0\rangle}\ot(y\leftarrow x_{\langle1\rangle}).$
\end{enumerate}
Conversely, any noncommutative Poisson bialgebra structure on $E$ with the canonical injection map $i: A\to E$ both a noncommutative Poisson algebra homomorphism and a noncommutative Poisson coalgebra homomorphism is of this form.
\end{theorem}
Note that in this case, $(A, \cdot, [, ], \Delta_A, \delta_A)$ is a noncommutative Poisson sub-bialgebra of $E=A^{}_{\sigma, \omega}\# {}^{q, t}_{}\, V$ and $(V, \cdot, [,], \Delta_V, \delta_V)$ is a  braided noncommutative Poisson bialgebra. Denote the set of all noncommutative Poisson bialgebraic extending datum of type (II) by $\mathcal{IB}^{(II)}({A}, V)$.

In the above two cases, we find that  the braided noncommutative Poisson bialgebra $V$ play a special role in the extending problem of noncommutative Poisson bialgebra $A$.
Note that $A^{p, s}_{}\# {}^{}_{\theta,\nu}\, V$ and $A^{}_{\sigma, \omega}\# {}^{q, t}_{}\, V$ are all noncommutative Poisson bialgebra structures on $E$.
Conversely, any noncommutative Poisson bialgebra extending system $E$ of ${A}$  through $V$ is isomorphic to such two types.
Now from Theorem \ref{thm-41}, Theorem \ref{thm-42} we obtain the main result of in this section,
which solve the extending problem for noncommutative Poisson bialgebra.


\begin{theorem}\label{bim1}
Let $({A}, \cdot, [,], \Delta_A, \delta_A)$ be a noncommutative Poisson bialgebra, $E$ a vector space containing ${A}$ as a subspace and $V$ be a complement of ${A}$ in $E$.
Denote by
$$\mathcal{HLB}(V, {A}): =\mathcal{IB}^{(I)}({A}, V)\sqcup\mathcal{IB}^{(II)}({A}, V)/\equiv.$$
Then the map
\begin{eqnarray}
\notag&&\Upsilon: \mathcal{HLB}(V, {A})\rightarrow BExtd(E, {A}),\\
&&\overline{\Omega^{(I)}({A}, V)}\mapsto A^{p, s}_{}\# {}^{}_{\theta, \nu}\, V, \quad   \overline{\Omega^{(II)}({A}, V)}\mapsto A^{}_{\sigma, \omega}\# {}^{q, t}_{}\, V
\end{eqnarray}
is bijective, where $\overline{\Omega^{(i)}({A}, V)}$ is the equivalence class of $\Omega^{(i)}({A}, V)$ under $\equiv$.
\end{theorem}

\section*{Acknowledgements}
This is a primary edition, something should be modified in the future.

\vskip7pt
\footnotesize{
\noindent Tao Zhang\\
College of Mathematics and Information Science,\\
Henan Normal University, Xinxiang 453007, P. R. China;\\
 E-mail address: \texttt{{zhangtao@htu.edu.cn}}

 \vskip7pt
\footnotesize{
\noindent Fang Yang\\
College of Mathematics and Information Science,\\
Henan Normal University, Xinxiang 453007, P. R. China;\\
 E-mail address: \texttt{{htuyangfang@163.com}}

\end{document}